\documentclass{article}
\usepackage{amsmath,amssymb,amsthm}
\usepackage{dsfont}
\usepackage{enumerate}
\allowdisplaybreaks[4]
\usepackage[colorlinks=true,linkcolor=blue,citecolor=green]{hyperref}

\usepackage[margin=2cm]{geometry}

\newlength{\bibitemsep}
\newlength{\bibparskip}
\let\oldthebibliography\thebibliography
\renewcommand\thebibliography[1]{\oldthebibliography{#1}\setlength{\parskip}{\bibitemsep}\setlength{\itemsep}{\bibparskip}}

\theoremstyle{thmstyleone}%
\newtheorem{theorem}{Theorem}%
\newtheorem{proposition}{Proposition}[section]%

\theoremstyle{thmstyletwo}%
\newtheorem{remark}{Remark}%

\theoremstyle{thmstylethree}%
\newtheorem{definition}{Definition}%

\usepackage{dsfont}
\usepackage{enumerate}

\allowdisplaybreaks[4]
\newtheorem{lemma}[proposition]{Lemma}

\numberwithin{equation}{section}

\begin{document}

\title{The Plasma-Charge Model: Boundary Effects and Global Well-posedness}
\author{Jingpeng Wu\thanks{Department of Applied Mathematics, Nanjing Forestry University, Nanjing, 210000, China ({\protect e-mail:\url{jp.wu@njfu.edu.cn}}).}}

\maketitle

\begin{abstract}
This article focuses on the Vlasov-Poisson system with point charges in bounded convex domains, accounting the interactions of point charges with the self-consistent electric field and the boundary, which were not addressed in the previous work \cite{Wu24}. We provide a rigorous characterization of the charge-boundary effect and establish the global well-posedness of the associated initial-boundary value problems under various boundary conditions. The analysis rests upon delicate estimates of the Green and Robin functions in convex domains, the Pfaffelmoser and Lions-Perthame methods for global well-posedness, and a desingularization argument that justifies the point-charge dynamics.
{\bf MR Subject Classification:} 35Q83 (Primary), 35A09 (Secondary), 82C40, 82D10
\end{abstract}

\tableofcontents

\section{Introduction}\label{sec-intr}

In this paper, we consider the time evolution of a one species plasma coupled with $M$ point charges of same sign in a bounded convex domain $\Omega\in\mathbb{R}^3$, which can be described by the Vlasov-Poisson system coupled with $M$ Newton's equations of motion, referred to as the plasma-charge model for short:
\begin{align}
&\partial_tf+v\cdot\nabla_xf+\left(\underbrace{\nabla_x\phi_{\#}}_{\text{plasma-plasma}}+ \sum_{\alpha}\underbrace{\nabla_xG_{\#}(x,\xi_{\alpha}(t))}_{\text{plasma-charge}}\right)\cdot\nabla_vf=0,\label{eq-Vlasov}\\
&\Delta\phi_{\#}=\rho,\quad \rho(t,x):=\int_{\mathbb{R}^3}f(t,x,v)\,\mathrm{d}v,\label{eq-Poisson}\\
&\dot{\xi}_{\alpha}=\eta_{\alpha},\quad\dot{\eta}_{\alpha}=\underbrace{\nabla_x\phi_{\#}(t,\xi_{\alpha})}_{\text{charge-plasma}}+\sum_{\beta:\beta\ne\alpha}\underbrace{\nabla_xG_{\#}(\xi_{\alpha},\xi_{\beta})}_{\text{charge-charge}}+\underbrace{\nabla_xH_{\#}(\xi_{\alpha})}_{\text{charge-boundary}},\label{eq-Newton}
\end{align}
where
\begin{equation*}
H_{\#}(x)=\frac{1}{2}R_{\#}(x)-\int_{\partial\Omega}G_{\#}(x,y)h_{\#}^{\rm cha}(y)\,\mathrm{d}S_y.
\end{equation*}
Here $f=f(t,x,v)\ge 0$ denotes the distribution function of the plasma particles at time $t\ge 0$, position $x\in\Omega$, and velocity $v\in\mathbb{R}^3$. $\xi_{\alpha}(t)\in\Omega$, $\eta_{\alpha}(t)\in\mathbb{R}^3$ denote the position and velocity of the $\alpha$-th point charge at time $t$. $\rho$ denotes the spatial density, which determines the plasma-induced electric potential $\phi_{\#}$ according to Coulomb's law, subject to the Neumann boundary if $\#=\mathrm{N}$ or the Dirichlet boundary if $\#=\mathrm{D}$. $G_{\#},R_{\#}$ denote the Green function and Robin function respectively, which will be specified later. 

Throughout this paper, $h_{\mathrm{D}}^{\rm cha}=0$ and $h_{\mathrm{N}}^{\rm cha}\in C^3(\partial\Omega;\mathbb{R}_+)$ be a given function satisfying
\begin{equation*}
\int_{\partial\Omega}h_{\mathrm{N}}^{\rm cha}(x)\,\mathrm{d}S_x=M.
\end{equation*}

The initial data associated to \eqref{eq-Vlasov}-\eqref{eq-Newton} are
\begin{equation}\label{eq-initial-data}
f(0,x,v)=\mathring{f}(x,v),\quad\xi_{\alpha}(0)=\mathring{\xi}_{\alpha},\quad\eta_{\alpha}(0)=\mathring{\eta}_{\alpha},\quad\alpha=1,\dots,M.
\end{equation}

To introduce the boundary conditions, We denote 
\begin{equation*}
\Pi=\Omega\times\mathbb{R}^3,\quad\bar{\Pi}=\bar{\Omega}\times\mathbb{R}^3,\quad\Gamma=\bar{\Pi}\setminus\Pi=\partial\Omega\times\mathbb{R}^3.
\end{equation*}
We split $\Gamma$ into three parts
\begin{align*}
\Gamma_1&=\left\{(x,v)\in\Gamma:  v\cdot n_x=0\right\},\\
\Gamma_+&=\left\{(x,v)\in\Gamma:  -v\cdot n_x>0\right\},\\
\Gamma_-&=\left\{(x,v)\in\Gamma:  -v\cdot n_x<0\right\}.
\end{align*}
We equip \eqref{eq-Vlasov} with either the absorption boundary
\begin{equation}\label{eq-absorption-boundary}
f(t,x,v)=g^{\rm pla}(t,x,v),\quad(x,v)\in\Gamma_+,
\end{equation}
or the reflection boundary
\begin{equation}\label{eq-reflection-boundary}
f(t,x,v)=f(t,x,\mathcal{R}_xv),\quad(x,v)\in\Gamma_+,
\end{equation}
where $g^{\rm pla}$ is a given function.  The reflection operator $\mathcal{R}_x:\mathbb{R}^3\to\mathbb{R}^3$ for $x\in\partial\Omega$ is defined by
\begin{equation*}
\mathcal{R}_xv=v-2(v\cdot n_x)n_x
\end{equation*}
where $n_x$ is the outer normal to $\partial\Omega$ at $x$.

We equip \eqref{eq-Poisson} with either the Dirichlet boundary for $\#=\mathrm{D}$
\begin{equation}\label{eq-Dirichlet-boundary}
\phi_{\mathrm{D}}(t,x)=h_{\mathrm{D}}(x)=0,\quad\forall x\in\partial\Omega,
\end{equation}
or the Neumann boundary for $\#=\mathrm{N}$
\begin{equation}\label{eq-Neumann-boundary}
n_x\cdot\nabla\phi_{\mathrm{N}}(t,x)=h_{\mathrm{N}}(x),\quad\forall x\in\partial\Omega,
\end{equation}
where $h_{\mathrm{N}}$ is a given function.

Note the zero functions $h_{\mathrm{D}}^{\rm cha},h_{\mathrm{D}}$ are introduced only for convenience. For the sake of brevity, we use the following notations throughout the paper:
\begin{align*}
\phi_{\#}^{\rm pla}(t,x)&=\phi_{\#}(t,x)+\sum_{\alpha}G_{\#}(x,\xi_{\alpha}(t)),\\
\phi_{\#,\alpha}^{\rm cha}(t,x)&=\phi_{\#}(t,x)+\sum_{\beta:\beta\ne\alpha}G_{\#}(x,\xi_{\beta}(t))+H_{\#}(x),\\
E_{\#}(t,x)&=\nabla_x\phi_{\#}(t,x),\quad F_{\#}(t,x)=\sum_{\alpha}\nabla_xG_{\#}(x,\xi_{\alpha}(t)),\\
F_{\#,\alpha}(t)&=\sum_{\beta:\beta\ne\alpha}\nabla_xG_{\#}(\xi_{\alpha}(t),\xi_{\beta}(t)),\\
E_{\#}^{\rm pla}(t,x)&=\nabla_x\phi_{\#}^{\rm pla}(t,x)=E_{\#}(t,x)+ F_{\#}(t,x),\\
E_{\#,\alpha}^{\rm cha}(t)&=\nabla_x\phi_{\#,\alpha}^{\rm cha}(t,\xi_{\alpha}(t))=E_{\#}(t,\xi_{\alpha}(t))+F_{\#,\alpha}(t)+\nabla_xH_{\#}(\xi_{\alpha}(t)).
\end{align*}
The equations \eqref{eq-Vlasov}, \eqref{eq-Newton} can then be rewritten as
\begin{align}
&\partial_tf+v\cdot\nabla_xf+E_{\#}^{\rm pla}\cdot\nabla_vf=0,\label{eq-Vlasov-short}\\
&\dot{\xi}_{\alpha}(t)=\eta_{\alpha}(t),\quad\dot{\eta}_{\alpha}(t)=E_{\#,\alpha}^{\rm cha}(t).\label{eq-Newton-short}
\end{align}

There are six types of effects among the plasma, point charges and boundary, i.e., the plasma-plasma effect, the plasma-charge effect, the charge-plasma effect, the charge-charge effect, the plasma-boundary effect, the charge-boundary effect. For A-B effect, we means the effect of B on A. Note the plasma-boundary effect has been implied in the boundary conditions on $\phi$ and $f$, hence we only mark the other five effects in \eqref{eq-Vlasov}-\eqref{eq-Newton}. One of the central difficulties in studying the plasma-charge model in domains is to mathematically control these effects, as either they are contained in the nonlinear terms of the equations or they may develop singularities on certain singular sets defined later.

\subsection{Background}

With the absence of the point charges and boundary, the Vlasov-Poisson system has been extensively studied. There is only the plasma-plasma effect in the system. Local in time classical solutions were studied in \cite{Kur52}. Global classical solutions were obtained in \cite{Ior61,UO78,Wol80a}  for low dimensions and \cite{Bat77,BD85,Hor81,Hor82} for special initial data.  Three dimensional global well-posedness with general data were established finally by two different methods in \cite{Pfa92} and \cite{LP91}, called the Pfaffelmoser's method and the Lions-Perthame's method respectively. We refer the reader to \cite{Rei07} for a more complete literature review. 

In the presence of boundary, the initial-boundary value problems of the Vlasov-Poisson system are more delicate to study. At this point, the plasma-boundary effect manifests. To handle this effect, the so called velocity lemma was introduced in \cite{Guo94,Guo95}. Then,  the Lions-Perthame's method was developed to handle the plasma-plasma effect and the global well-posedness in a half space with reflection-Neumann or absorption-Dirichlet boundary condition was obtained in \cite{Guo94}, which was further extended in \cite{Hwa04} for the case of convex domain with absorption-Dirichlet boundary condition. The global well-posedness in a half space with reflection-Dirichlet boundary condition was established in \cite{HV09} by a new proof of the velocity lemma and the extended Pfaffelmoser's method. Based on tools from differential geometry, the global well-posedness for the case of convex domain with reflection-Neumann or reflection-Dirichlet boundary condition was finally obtained in \cite{HV10,HJV13}.

The presence of point charges can be considered as a singular perturbation of the Vlasov-Poisson system and significantly complicates the analysis. In this case, the effects involving point charges appear. When all plasma particles and point charges have same sign, they mutually repel each other. At this point, the singularities arising from effects related to point charges can be avoided by applying the energy conservation of the system and the bound of a pointwise energy function. This idea was initiated in \cite{CM10} for the global well-posedness with initial data vanishing in the neighbor of the singular set in $\mathbb{R}^2$. Then, it was further developed for three dimensional case by \cite{MMP11}, in which the Pfaffelmoser's method was extended. The Lions-Perthame's method was then developed in \cite{DMS15,WZ21,WZ23PC}, in which the plasma particles were allowed to overlap the point charges. Benefited from the vigorous development of the Vlasov-Poisson systems, ample results for the repulsive plasma-charge model have been obtained very recently, see \cite{Wu24} for more detailed literature review. In contrast, the attractive plasma-charge model has been rarely studied. The global existence of strong solutions has been established in \cite{CMMP12} for two dimensions, without proving uniqueness. The global existence of classical solutions for three dimensions is totally open.

In \cite{CM10}, the global well-posedness of the plasma-charge model in a bounded domain of the plane with electric confinement was established. At this point, an external gradient force keeps the system inside the domain, hence the effects involving boundary did not actually appear. Recently, the case with the presence of both point charges and boundary was studied in \cite{Wu24} by adapting the Pfaffelmoser's method, in which, for the global well-posedness, the point charges are assumed to be fixed at certain locations within the domain.

So far, the charge-boundary effect has not been considered in the literature. This paper is devoted to investigating this problem, offering a rigorous characterization of charge-boundary effect and the corresponding well-posedness result.

\subsection{Main results}

Before describing the main results, we introduce the local coordinate transformation and the conditions that the initial data and boundary data must satisfy.

We assume that
\begin{equation}\label{eq-assum-Omega}
\Omega\;\text{is a }\;C^4\;\text{convex bounded domain with a small coefficience}\;\delta_0>0
\end{equation}
such that the set of points $(x,v)\in([\partial\Omega+B_{\delta_0}^3]\cap\Omega)\times\mathbb{R}^3$ can be parameterized uniquely by means of the unique values $(\mu_1,\mu_2,x_{\perp},w_1,w_2,v_{\perp})$ solving the equations:
\begin{equation*}
\begin{split}x&=x_{\parallel}(\mu_1,\mu_2)-x_{\perp}n(\mu_1,\mu_2),\\
v&=\sum_{i=1,2}w_iu_i-v_{\perp}n(\mu_1,\mu_2),\quad u_i:=\partial_{\mu_i} x_{\parallel}(\mu_1,\mu_2),
\end{split}
\end{equation*}
where $(\mu_1,\mu_2)$ is a set of coordinates on the surface $\partial\Omega$ satisfying that the parametric lines are lines of curvature, see \cite[p81,(6-6),(6-8)]{str88}. $x_{\parallel}(\mu_1,\mu_2)$ denote the point of $\partial\Omega$ characterized by the values of the parameters $(\mu_1,\mu_2)$ and $n(\mu_1,\mu_2)$ or $n_{x_{\parallel}}$ the outer normal to $\partial\Omega$ at the point $x_{\parallel}(\mu_1,\mu_2)$. We denote $v_{\parallel}(\mu_1,\mu_2)=w_1u_1+w_2u_2$ as the tangential component of $v$. By the definition of reflection, for $x\in\partial\Omega$, we have
\begin{equation*}
\mathcal{R}_xv=v_{\parallel}(\mu_1,\mu_2)+v_{\perp}n(\mu_1,\mu_2).
\end{equation*}
The three parts of $\Gamma$ can be rewritten by
\begin{align*}
\Gamma_1&=\left\{(x,v)\in\Gamma: v_{\perp}=0\right\},\\
\Gamma_+&=\left\{(x,v)\in\Gamma: v_{\perp}>0\right\},\\
\Gamma_-&=\left\{(x,v)\in\Gamma: v_{\perp}<0\right\}.
\end{align*}

Under this new coordinates, the Vlasov equation takes the form given in the following lemma.
\begin{lemma}\label{lem-Vlasov-new-coordinates}
The Vlasov equation \eqref{eq-Vlasov-short} can be rewritten for $(x,v)\in([\partial\Omega+B_{\delta_0}^3]\cap\Omega)\times\mathbb{R}^3$ using the set of coordinates $(\mu_1,\mu_2,x_{\perp},w_1,w_2,v_{\perp})$ in the form
\begin{equation*}
\partial_{t}f+\sum_{i=1,2}\frac{w_i\partial_{\mu_i}f}{1+\kappa_ix_{\perp}}+v_{\perp}\partial_{x_{\perp}}f+\sum_{i=1,2}\bar{E}_{\#,i}^{\rm pla}\partial_{w_i}f+\bar{E}_{\#,\perp}^{\rm pla}\partial_{v_{\perp}}f=0,
\end{equation*} 
where
\begin{equation*}
\bar{E}_{\#,i}^{\rm pla}=E_{\#,i}^{\rm pla}-\frac{v_{\perp}w_i\kappa_i}{1+\kappa_ix_{\perp}}-\sum_{j,\ell=1,2}\frac{\Gamma^i_{j,\ell}w_jw_{\ell}}{1+\kappa_jx_{\perp}},\quad \bar{E}_{\#,\perp}^{\rm pla}=E_{\#,\perp}^{\rm pla}+\sum_{j=1,2}\frac{w_j^2b_j}{1+\kappa_jx_{\perp}},
\end{equation*}
$\kappa_j$'s are the principal curvatures, $b_j$'s are the coefficients $e,g$ from the second fundamental form according to the notation in \cite{str88} and $\Gamma_{j,\ell}^i$'s are the Christoffel symbols of the surface $\partial\Omega$. The electric field $E_{\#}^{\rm pla}$ has been written in the form
\begin{equation*}
E_{\#}^{\rm pla}=E_{\#,1}^{\rm pla}u_1+E_{\#,2}^{\rm pla}u_2-E_{\#,\perp}^{\rm pla}n(\mu_1,\mu_2).
\end{equation*}
For convenience, $\delta_0$ has been set to be less than $(2+2\max_{j=1,2,x\in\partial\Omega}\vert\kappa_j(x)\vert)^{-1}$.
\end{lemma}

\begin{proof}
It is standard by Gauss-Weingarten equations and the fact that the parametric lines are lines of curvature.
\end{proof}


Firstly, we introduce the conditions for avoiding singularities from boundary and point charges. There are two types of singular sets we have to deal with. The first one is the well known grazing set $\Gamma_1$ which is brought by the boundary effect. The second is defined by
\begin{equation*}
\Gamma_2(t,v):=\left\{(x,v)\in\bar{\Pi}: \exists \alpha\in\mathcal{I}_M \;\text{satisfying}\;x=\xi_{\alpha}(t)\right\},
\end{equation*}
which is brought by the point charges. To make sure that the solutions stay away from the singular sets, we assume that there exists $\delta_1>0$ such that
\begin{equation}\label{eq-initial-singular-sets}
\left\{\begin{split}
&\min_{\alpha}\left\{\operatorname{d}_{\partial\Omega}(\mathring{\xi}_{\alpha}),\min_{\beta:\beta\ne\alpha}\vert\mathring{\xi}_{\alpha}-\mathring{\xi}_{\beta}\vert\right\}>\delta_1,\\
&\mathring{f}(x,v)\equiv 0\;\;\text{\rm if}\,\operatorname{dist}((x,v),\Gamma_1\cup\Gamma_2(0,v))\le\delta_1,\quad\forall (x,v)\in\Pi,\\
&g^{\rm pla}(t,x,v)\equiv 0\;\;\text{\rm if}\,\operatorname{dist}((x,v),\Gamma_1)\le\delta_1,\quad\forall (t,x,v)\in\bar{\mathbb{R}}_+\times\Gamma_+.
\end{split}\right.
\end{equation}

Next, we introduce the compatibility conditions. In order to prove solvability of \eqref{eq-Vlasov} with \eqref{eq-initial-data}, \eqref{eq-absorption-boundary}, we need to assume that $\mathring{f},g^{\rm pla}$ satisfy the following compatibility conditions: for $(x,v)\in\Gamma_+$
\begin{equation}\label{eq-absorption-compatibility}
\left\{\begin{split}
&\mathring{f}(x,v)=g^{\rm pla}(0,x,v),\\
&\partial_tg^{\rm pla}(0,x,v)+v\cdot\nabla_x\mathring{f}(x,v)+E_{\#}^{\rm pla}(0,x)\cdot\nabla_v\mathring{f}(x,v)=0,
\end{split}\right.
\end{equation}

In order to prove solvability of \eqref{eq-Vlasov} with \eqref{eq-initial-data}, \eqref{eq-reflection-boundary}, we need to assume that $\mathring{f}$ satisfies the following compatibility conditions: for $(x,v)\in\Gamma_+$
\begin{equation}\label{eq-reflection-compatibility}
\left\{\begin{split}
&\mathring{f}(x,v)=\mathring{f}(x,\mathcal{R}_xv),\\
&\left[v_{\perp}\nabla_{x}^{\perp}+2E_{\#,\perp}^{\rm pla}(0,x)\nabla_{v}^{\perp}\right]\mathring{f}(x,v)=-v_{\perp}\nabla_{x}^{\perp}\mathring{f}(x,\mathcal{R}_xv).
\end{split}\right.
\end{equation}
where $\nabla_{x}^{\perp},\nabla_{v}^{\perp}$ are the normal components to $\partial\Omega$ of the gradients $\nabla_x,\nabla_v$ respectively. These compatibility conditions can be deduced by the reflection condition $f(t,x,v)=f(t,x,\mathcal{R}_xv)$ and the Vlasov equation in a geometric form given in Lemma~\ref{lem-Vlasov-new-coordinates}.

In order to prove solvability of \eqref{eq-Poisson} with \eqref{eq-Neumann-boundary}, we need to assume that $h_{\mathrm{N}}$ satisfying the compatibility condition
\begin{equation}\label{eq-Neumann-compatibility}
\int_{\Pi}\mathring{f}(x,v)\,\mathrm{d}x\,\mathrm{d}v=\int_{\partial\Omega}h_{\mathrm{N}}(x)\,\mathrm{d}S_x.
\end{equation}

At last, we introduce the regularity conditions on the data. For the absorption case, we assume that $\mathring{f}$, $g^{\rm pla}$, $h_{\mathrm{N}}$ satisfy
\begin{equation}\label{eq-assum-regularity-absorption}
\mathring{f}\in C_c^{1}(\bar{\Pi};\bar{\mathbb{R}}_+),\quad g^{\rm pla}\in C_c^{1}(\bar{\mathbb{R}}_+\times\Gamma_+;\bar{\mathbb{R}}_+),\quad h_{\mathrm{N}}\in C^2(\partial\Omega;\mathbb{R}_+).
\end{equation}
For the reflection case, we assume that $\mathring{f}$, $h_{\mathrm{N}}$ satisfy
\begin{equation}\label{eq-assum-regularity-reflection}
\mathring{f}\in C_c^{1,\mu}(\bar{\Pi};\bar{\mathbb{R}}_+),\quad h_{\mathrm{N}}\in C^3(\partial\Omega;\mathbb{R}_+),
\end{equation}
for some $\mu>0$.

Now we are ready to present our main results. The first main results in the paper are the following local well-posedness.
\begin{theorem}[Local well-posedness for absorption case]\label{thm-local-wellposedness-absorption}
Let $\Omega\subset\mathbb{R}^3$ satisfy \eqref{eq-assum-Omega}. Let $\mathring{f}$, $g^{\rm pla}$, $h_{\mathrm{N}}$, $(\mathring{\xi}_{\alpha},\mathring{\eta}_{\alpha})\in\Pi$ satisfy \eqref{eq-initial-singular-sets}, \eqref{eq-absorption-compatibility}, \eqref{eq-Neumann-compatibility}, \eqref{eq-assum-regularity-absorption}. Then there exists a unique classical solution $f,\{\xi_{\alpha},\eta_{\alpha}\}$ to the plasma-charge model \eqref{eq-Vlasov}-\eqref{eq-Newton} with  \eqref{eq-initial-data}, \eqref{eq-absorption-boundary} and \eqref{eq-Dirichlet-boundary} or \eqref{eq-Neumann-boundary},  on some time interval $[0,\mathring{T})$. For any $0<T<\mathring{T}$, $f\in C_c^{1}([0,T]\times\bar{\Pi};\bar{\mathbb{R}}_+)$, $\phi\in C^{1;2}([0,T]\times\bar{\Omega})$ and $\xi_{\alpha}\in C^2([0,T];\Omega)$, $\eta_{\alpha}\in C^1([0,T];\mathbb{R}^3)$ for all $\alpha\in\mathcal{I}_M$. Assume that $\mathring{T}$ is chosen maximal, then $\mathring{T}=+\infty$ if the following hold 
\begin{align}
\sup_{0\le t<\mathring{T}}\left[\sum_{\alpha}\vert\eta_{\alpha}(t)\vert+\frac{1}{\operatorname{d}_{\partial\Omega}(\xi_{\alpha}(t))}+\sum_{\alpha\ne\beta}\frac{1}{\vert\xi_{\alpha}(t)-\xi_{\beta}(t)\vert}\right]&<\infty,\label{eq-continuation-point-charge-absorption}\\
\sup_{0\le t<\mathring{T}}\left\{\vert v\vert+\sum_{\alpha}\frac{1}{\vert x-\xi_{\alpha}(t)\vert}:(x,v)\in\operatorname{supp} f(t)\right\}&<\infty.\label{eq-continuation-plasma-absorption}
\end{align}
\end{theorem}

\begin{theorem}[Local well-posedness for reflection case]\label{thm-local-wellposedness-reflection}
Let $\Omega\subset\mathbb{R}^3$ satisfy \eqref{eq-assum-Omega}. Let $\mathring{f}$, $h_{\mathrm{N}}$, $(\mathring{\xi}_{\alpha},\mathring{\eta}_{\alpha})\in\Pi$ satisfy \eqref{eq-initial-singular-sets}, \eqref{eq-reflection-compatibility}, \eqref{eq-Neumann-compatibility}, \eqref{eq-assum-regularity-reflection}. Then there exists a unique classical solution $f,\{\xi_{\alpha},\eta_{\alpha}\}$ to the plasma-charge model \eqref{eq-Vlasov}-\eqref{eq-Newton} with  \eqref{eq-initial-data}, \eqref{eq-reflection-boundary} and \eqref{eq-Dirichlet-boundary} or \eqref{eq-Neumann-boundary},  on some time interval $[0,\mathring{T})$. For any $0<T<\mathring{T}$, $f\in C_c^{1;1,\lambda}([0,T]\times\bar{\Pi};\bar{\mathbb{R}}_+)$, $\phi\in C^{1;3,\lambda}([0,T]\times\bar{\Omega})$ for some $0<\lambda<\mu$ and $\xi_{\alpha}\in C^2([0,T];\Omega)$, $\eta_{\alpha}\in C^1([0,T];\mathbb{R}^3)$ for all $\alpha\in\mathcal{I}_M$. Assume that $\mathring{T}$ is chosen maximal, then $\mathring{T}=+\infty$ if the following hold 
\begin{align}
\sup_{0\le t<\mathring{T}}\left[\sum_{\alpha}\vert\eta_{\alpha}(t)\vert+\frac{1}{\operatorname{d}_{\partial\Omega}(\xi_{\alpha}(t))}+\sum_{\alpha\ne\beta}\frac{1}{\vert\xi_{\alpha}(t)-\xi_{\beta}(t)\vert}\right]&<\infty,\label{eq-continuation-point-charge-reflection}\\
\sup_{0\le t<\mathring{T}}\left\{\vert v\vert+\sum_{\alpha}\frac{1}{\vert x-\xi_{\alpha}(t)\vert}:(x,v)\in\operatorname{supp} f(t)\right\}&<\infty.\label{eq-continuation-plasma-reflection}
\end{align}
\end{theorem}

The second main  result is the following global well-posedness of the model with the absorption-Neumann boundary condition.
\begin{theorem}\label{thm-absorption-Neumann-case}
For Neumann boundary case, the solution given in Theorem~\ref{thm-local-wellposedness-absorption} is global, i.e., $\mathring{T}=+\infty$.
\end{theorem}

The third main result is the following global well-posedness of the model with the reflection-Neumann boundary condition.
\begin{theorem}\label{thm-reflection-Neumann-case}
For Neumann boundary case, the solution given in Theorem~\ref{thm-local-wellposedness-reflection} is global, i.e., $\mathring{T}=+\infty$.
\end{theorem}

It is worth noting that the exact representation of the charge-boundary effect in \eqref{eq-Newton} is not very obvious. The final main result is the rigorous justification of this representation in Theorem~\ref{thm-justification-boundary-effect}. Due to the substantial preliminary explanation required for this theorem, we place it in Section~\ref{sec-justification}.

\subsection{Main ideas and contributions}

The analysis of the plasma-charge model in a convex domain involves two fundamental challenges: the interplay of multiple singularities and the physical justification of boundary-induced forces on point charges.

Regarding the first challenge, the primary difficulty stems from the coexistence of four distinct types of singularities: plasma-boundary, charge-boundary, plasma-charge, and charge-charge singularities.
\begin{itemize}
\item \textit{Local well-posedness:} For Theorems~\ref{thm-local-wellposedness-absorption} and \ref{thm-local-wellposedness-reflection}, the plasma-boundary singularity is addressed by invoking a specialized velocity lemma (Lemma~\ref{lem-velocity-lemma}), which relies on a refined characterization of the distance function $\boldsymbol{\beta}$ near the grazing set. To handle the remaining singularities, we construct a carefully designed pointwise energy functional $e^n$. See Section~\ref{sec-local-well-posedness}.

\item \textit{Control of charge-boundary singualrity:} Establishing uniform control over the inverse distance $\frac{1}{\operatorname{d}_{\partial\Omega}(\xi_{\alpha}(t))}$ in \eqref{eq-continuation-point-charge-absorption}, \eqref{eq-continuation-point-charge-reflection} is pivotal for proving global well-posedness. A central component of this analysis is the derivation of pointwise estimates for the Robin functions (Proposition~\ref{prn-Robin-estimate}). These estimates provide a precise characterization of the boundary-induced field, reflecting the attractive or repulsive nature of the Dirichlet and Neumann conditions, respectively. Particularly, in the Neumann case, this allows us to exploit energy conservation to ensure that the point charges remain uniformly bounded away from the boundary. While such estimates in the Dirichlet setting can be established via the maximum principle or the asymptotic expansions in \cite[Lemma 8.3]{Flu99} (as widely utilized in vortex dynamics \cite{MP84,Tur87,Mar22}), the corresponding pointwise estimates for the Neumann case appear to be new. We prove Proposition~\ref{prn-Robin-estimate} by utilizing layer potential techniques. See Section~\ref{sec-Green-Robin}.

\item \textit{The role of convexity}: Two key agruments rely on the convexity of the domain $\Omega$. First, in the proof of the velocity lemma, this geometric assumption is a well-established necessity. Second, the convexity of $\Omega$ is utilized in the derivation of Proposition~\ref{prn-Robin-estimate}, whether this constraint is indispensable remains unclear.

\item \textit{Global continuation:} The plasma-related continuation criteria in \eqref{eq-continuation-point-charge-absorption}-\eqref{eq-continuation-plasma-reflection} are addressed via two distinct analytical frameworks, depending on the boundary conditions. For the absorption case (Theorem~\ref{thm-absorption-Neumann-case}), we generalize the Lions-Perthame's method developed in \cite{Guo94} (without point charges) by incorporating singular moment estimates established in \cite{WZ23PC} (without boundary). Due to the presence of boundary, we need to incorporate a boundary cut-off into the mehtod in \cite{WZ23PC}, see Lemma~\ref{lem-Lk-by-Hk-E-LP}. This allows us to derive the following Gr\"onwall-type inequality:
\begin{equation*}
H_k\lesssim L_{k-2}+\,\text{Lions-Perthame terms}\, \lesssim H_{k-1}+\int_0^tH_{k}(s)\,\mathrm{d}s+\,\text{Lions-Perthame terms},
\end{equation*}
which brings our to the Lions-Perthame’s argument in \cite{Guo94}. Due to the presence of point charges, we need to integrate the technology first developed by \cite{DMS15} into the method in \cite{Guo94}. See Section~\ref{sec-global-absorption}.

For the reflection case (Theorem~\ref{thm-reflection-Neumann-case}), we adapt the Pfaffelmoser's method based on the separation technique of singular sets introduced in \cite{Wu24}. The main idea is that within a small time length $\delta$, a plasma particle can only be close to either boundary or point charges. During this period, we only need to analyze either the plasma-charge effect by the method in \cite{MMP11} (without boundary) or the plasma-boundary effect by the method in \cite{HV10} (without point charges). Furthermore, We utilized different time steps $\delta$ for these two scenarios. Indeed, on the one hand, in \cite{MMP11}, when $M>1$, $\delta$ cannot be chosen too small, in order to ensure the following transitivity by Lemma~\ref{lem-estimate-a0-ne-bar-alpha}: If for some $\alpha$ there holds
\begin{equation*}
\sqrt{h_{\alpha}}(t,X,V)\le Q_{t-\delta,\delta}+CQ_{T,T}\delta,
\end{equation*}
then 
\begin{equation*}
\sqrt{h_{\beta}}(t,X,V)\le Q_{t-\delta,\delta}+CQ_{T,T}\delta,
\end{equation*}
for all $1\le \beta\le M$. On the other hand, in \cite{HV10}, to obtain stability property of the velocity of plasma particles from a non-flat boundary, we first partition $\delta$ into smaller $\tilde{\delta}$, then sum the contributions from all $\tilde{\delta}$ intervals to obtain the total contribution for the $\delta$ interval. See Section~\ref{sec-global-reflection}.
\end{itemize}

The second challenge concerns the physical consistency of the charge-boundary interaction term $H_{\#}$ in the Newtonian dynamics \eqref{eq-Newton}. While the attractive/repulsive nature of Dirichlet/Neumann boundary forces is well-understood in classical electrostatics, its rigorous integration into the Vlasov-Poisson system remains, to the author's knowledge, unexplored in the literature. We seek a justification by deriving the point-charge model from a modified Vlasov-Poisson system via a desingularization limit. This approach is inspired by the classical derivation of point vortex model from Euler equations \cite{MP83,Mar88}. However, the Vlasov-Poisson system introduces distinct analytical obstacles: the phase-mixing effect renders standard moment of inertia estimates and rearrangement inequalities, central to the fluid context, largely inapplicable. To overcome this, we focus on the interaction between point charges and the boundary rather than the self-energy of individual particles. By introducing a modified system that isolates these interactions, we rigorously deduce the point charge model incorporating boundary effects as a desingularization limit (Theorem~\ref{thm-justification-boundary-effect}). See Section~\ref{sec-justification}.

\subsection{Notation}

Let $\Omega\subset\mathbb{R}^d$, by $\vert\Omega\vert$ we denote the volume of $\Omega$, by $\operatorname{diam}\Omega$ we denote the diameter of $\Omega$, by $n_x$ we denote the outward unit normal to $\partial\Omega$ at point $x$, by $\mathrm{d}S_x$ we denote the surface measure on $\partial\Omega$. For $x\in\Omega$, we denote $\operatorname{d}_{\partial\Omega}(x):=\inf_{y\in\partial\Omega}\vert x-y\vert$. For $x\in\partial\Omega$, we denote $\partial_{n_x}:=n_x\cdot\nabla_x$. Let $B^d(x,\delta)$ be a ball in $\mathbb{R}^d$ with center $x$ radius $\delta$. In particular, a ball in $\mathbb{R}^d$ with center $0$ radius $r$ can be denoted by $B_r^d$. Let $\mathbb{R}_+^d=\{(x_1,\bar{x})\in\mathbb{R}\times\mathbb{R}^{d-1}: x_1>0\}$, and $\bar{\mathbb{R}}_+^d=\{(x_1,\bar{x})\in\mathbb{R}\times\mathbb{R}^{d-1}: x_1\ge 0\}$. For $x=(x_1,\bar{x})\in\mathbb{R}_+^d$, we define $x'=(-x_1,\bar{x})$.

The space \(C_c^{1;1,\lambda}([0,T]\times\bar{\Pi})\) is introduced in \cite{HV10}, which means that
\[
C_c^{1;1,\lambda}=C_{t,x,v}^1\cap L_t^{\infty}C_{x,v}^{1,\lambda}.
\]
\(C^{1;3,\lambda}([0,T]\times\bar{\Omega})\) has analogous definition.

Let $\chi$ be a smooth cutoff function, $\chi(t)=1$ for $t\in[0,1]$, $\chi(t)=0$ for $t\ge 2$, $-2\le\chi'(t)\le 0$, $0\le\chi(t)\le 1$ for $t\in\bar{\mathbb{R}}_+$. Denote $\tilde{\chi}=1-\chi$, $\chi_{\varsigma}(t)=\chi(\varsigma^{-1} t)$, $\tilde{\chi}_{\varsigma}(t)=\tilde{\chi}(\varsigma^{-1} t)$. The indicator function of a set $A$ is denoted as $\mathds{1}_{A}(x)$ such that $\mathds{1}_{A}(x)=1$ if $x\in A$, and $\mathds{1}_{A}(x)=0$ otherwise. When the index set $\mathcal{I}$ and the index $i\in\mathcal{I}$ are clear, we always denote $\{a_i\}$ as the set $\{a_i: i\in\mathcal{I}\}$. We denote $\mathcal{I}_M=\{1,2,\dots,M\}$ and use the following summation signs throughout the paper for short:
\begin{equation*}
\sum_{\beta:\beta\ne\alpha}=\sum_{\beta\in\mathcal{I}_M,\beta\ne\alpha},\quad\sum_{\alpha\ne\beta}=\sum_{\alpha\in\mathcal{I}_M}\sum_{\beta:\beta\ne\alpha},\quad\sum_{\alpha}=\sum_{\alpha\in\mathcal{I}_M}.
\end{equation*}
The notations $\min_{\alpha}$, $\sup_{\alpha}$ have similar definitions.

$C$ is denoted as a universal constant depending on $\|\mathring{f}\|_{C^{1}(\Pi)}$, $\|g^{\rm pla}\|_{C^1(\bar{\mathbb{R}}_+\times\Gamma_+)}$, $\|h_{\mathrm{N}}\|_{C^2(\Pi)}$ for the absorption case or $\|\mathring{f}\|_{C^{1,\mu}(\Pi)}$, $\|h_{\mathrm{N}}\|_{C^3(\Pi)}$ for the reflection case, and $\max_{\alpha}(\vert\mathring{\xi}_{\alpha}\vert,\vert\mathring{\eta}_{\alpha}\vert)$, $\delta_0$, $\delta_1$, $M$, the support of $\mathring{f}$, $\operatorname{diam}\Omega$, the $C^4$ regularity of $\partial\Omega$, which may differ on different occasions.

\section{The Green functions and Robin functions}\label{sec-Green-Robin}

Throughout this section $\Omega$ is always a bounded convex $C^4$ domain in $\mathbb{R}^d$, $d\ge 2$. The fundamental solution of the Laplacian in $\mathbb{R}^d$ is defined by
\begin{equation}\label{eq-definition-G}
G(x,y):=\left\{\begin{split}
\frac{1}{2\pi}\ln\vert x-y\vert,\quad d=2,\\
\frac{\vert x-y\vert^{2-d}}{(2-d)\omega_d},\quad d\ge3.
\end{split}\right.
\end{equation}
$\omega_d$ denotes the surface area of the unit sphere in $\mathbb{R}^d$. 

\begin{definition}
Green function $G_{\Omega,\mathrm{D}}$ of $\Omega$ with Dirichlet boundary condition is defined by $G_{\Omega,\mathrm{D}}(x,y)=0$ if $x\in\partial\Omega$ and $G_{\Omega,\mathrm{D}}(\cdot,y)-G(\cdot,y)$ is harmonic in $\Omega$ for every $y\in\Omega$.
\end{definition}

\begin{definition}
Green function $G_{\Omega,\mathrm{N}}$ of $\Omega$ with Neumann boundary condition is defined by
\begin{align*}
\Delta_x\left[G_{\Omega,\mathrm{N}}(x,y)-G(x,y)\right]=-\frac{1}{\vert\Omega\vert},&\quad (x,y)\in\Omega\times\Omega,\\
\partial_{n_x}G_{\Omega,\mathrm{N}}(x,y)=0,&\quad (x,y)\in\partial\Omega\times\Omega.
\end{align*}
\end{definition}

\begin{definition}
The Robin function associating with $G_{\Omega,\#}$ is defined by
\begin{equation*}
R_{\Omega,\#}(x):=G_{\Omega,\#}(x,x)-G(x,x).
\end{equation*}
\end{definition}

When $\Omega$ is unambiguous, we denote $G_{\Omega,\#}$, $R_{\Omega,\#}$ by $G_{\#}$, $R_{\#}$. Note $G_{\mathrm{D}}$ exists uniquely and is symmetric in the variables $x,y$. While $G_{\mathrm{N}}$ is not unique, it is available to let $G_{\mathrm{N}}$ being symmetric in the variables $x,y$. We denote the harmonic parts of the Green functions by
\begin{align*}
g_{\mathrm{D}}(x,y)&:=G_{\mathrm{D}}(x,y)-G(x,y),\\
g_{\mathrm{N}}(x,y)&:=G_{\mathrm{N}}(x,y)-G(x,y)-\tilde{g}(x,y),
\end{align*}
where $\tilde{g}$ is a smooth function on $\Omega\times\Omega$ and symmetric in the variables $x,y$ satisfying $\Delta_x\tilde{g}(x,y)\equiv-\frac{1}{\vert\Omega\vert}$. We set
\begin{equation*}
\tilde{g}(x,y)=-\frac{\vert x-y\vert^2}{2d\vert\Omega\vert}.
\end{equation*}
Denote
\begin{equation*}
\bar{g}_{\mathrm{D}}(x,y)=g_{\mathrm{D}}(x,y),\quad \bar{g}_{\mathrm{N}}=\tilde{g}(x,y)+g_{\mathrm{N}}(x,y).
\end{equation*}
Then the Robin functions are given by
\begin{equation*}
R_{\mathrm{D}}(x)=\bar{g}_{\mathrm{D}}(x,x)=g_{\mathrm{D}}(x,x),\quad R_{\mathrm{N}}(x)=\bar{g}_{\mathrm{N}}(x,x)=g_{\mathrm{N}}(x,x).
\end{equation*}

\subsection{Exact representation of Green functions for special cases}\label{subsec-special-Green}

For $\Omega=B_1^2$, where $B_1^2$ is the unit ball in $\mathbb{R}^2$, Green functions are given by
\begin{align*}
G_{B_1^2,\mathrm{D}}(x,y)&=\frac{1}{2\pi}\ln\frac{\vert x-y\vert}{\left\vert x\vert y\vert-y\vert y\vert^{-1}\right\vert},\\
G_{B_1^2,\mathrm{N}}(x,y)&=\frac{1}{2\pi}\ln\left(\vert x-y\vert\left\vert x\vert y\vert-y\vert y\vert^{-1}\right\vert\right)+\frac{\vert x\vert^2+\vert y\vert^2}{4\pi}.
\end{align*}
Robin functions are given by
\begin{align*}
R_{B_1^2,\mathrm{D}}(x)&=-\frac{1}{2\pi}\ln\left\vert x\vert x\vert-x\vert x\vert^{-1}\right\vert,\\
R_{B_1^2,\mathrm{N}}(x)&=\frac{1}{2\pi}\ln\left\vert x\vert x\vert-x\vert x\vert^{-1}\right\vert+\frac{\vert x\vert^2}{2\pi}.
\end{align*}

We can also define the Green functions for unbounded domains. For $\Omega=\mathbb{R}_+^d$, noting that $\frac{1}{\vert \Omega\vert}=0$, Green functions are given by
\begin{equation*}
G_{\mathbb{R}_+^d,\#}(x,y)=G(x,y)+(-1)^{\#}G(x,y'),
\end{equation*}
where $(-1)^\mathrm{D}=-1$, $(-1)^\mathrm{N}=1$. In this case the Robin functions only depend on $x_1$ for $x=(x_1,\bar{x})\in\mathbb{R}_+^d$, given by
\begin{equation*}
R_{\mathbb{R}_+^d,\#}(x)=(-1)^{\#}G(x,x')=(-1)^{\#}\left\{\begin{split}
\frac{1}{2\pi}\ln(2x_1),\quad &d=2,\\
\frac{(2x_1)^{2-d}}{(2-d)\omega_d},\quad &d\ge 3.
\end{split}\right.
\end{equation*}

In both cases, Robin functions satisfy 
\begin{equation*}
\lim_{\operatorname{d}_{\partial\Omega}(x)\to 0}R_{\Omega,\mathrm{D}}(x)=+\infty,\quad\lim_{\operatorname{d}_{\partial\Omega}(x)\to 0}R_{\Omega,\mathrm{N}}(x)=-\infty.
\end{equation*}
In the sequel of this section, we will prove that this fact holds also for any bounded convex domain $\Omega$.

\subsection{Layer potentials}

By the definitions of the Green functions on $\Omega$, for all $y\in\Omega$, $g_{\mathrm{D}}(\cdot,y)$ satisfies
\begin{equation}\label{eq-regular-part-D}
\left\{\begin{split}
\Delta_xg_{\mathrm{D}}=0,\quad &x\in\Omega,\\
g_{\mathrm{D}}=\bar{h}_{\mathrm{D}},\quad &x\in\partial\Omega,
\end{split}\right.
\end{equation}
where $\bar{h}_{\mathrm{D}}(x,y):=-G(x,y)$. Similarly, $g_{\mathrm{N}}(\cdot,y)$ satisfies
\begin{equation}\label{eq-regular-part-N}
\left\{\begin{split}
\Delta_xg_{\mathrm{N}}=0,\quad &x\in\Omega,\\
\partial_{n_x}g_{\mathrm{N}}=\bar{h}_{\mathrm{N}},\quad &x\in\partial\Omega,
\end{split}\right.
\end{equation}
where $\bar{h}_{\mathrm{N}}(x,y):=-\partial_{n_x}[G(x,y)+\tilde{g}(x,y)]$.

According to  \cite{FJR78} for $C^1$ bounded domain, we can solve the boundary value problems \eqref{eq-regular-part-D} and \eqref{eq-regular-part-N} by the following formula:
\begin{equation}\label{eq-layer-representation}
g_{\#}(x,y)=\mathcal{S}_{\#}[T_{\#}[\bar{h}_{\#}(\cdot,y)](\cdot)](x),\quad x,y\in\Omega,
\end{equation}
where $\mathcal{S}_{\#}$, $T_{\#}$ are operators defined on $L^p(\partial\Omega)$ as follows. The boundary-to-domain double layer potential $\mathcal{S}_{\mathrm{D}}$ is defined by
\begin{equation*}
\mathcal{S}_{\mathrm{D}}g(x):=\int_{\partial\Omega}
\partial_{n_{w}}G(x,w)g(w)\,\mathrm{d}S_w,\quad x\in\Omega.
\end{equation*}
The boundary-to-domain single layer potential $\mathcal{S}_{\mathrm{N}}$ is defined by
\begin{equation*}
\mathcal{S}_{\mathrm{N}}g(x):=\int_{\partial\Omega}
-G(x,w)g(w)\,\mathrm{d}S_w,\quad x\in\Omega.
\end{equation*}
The operator $T_{\#}$ is the inverse of a Fredholm type operator defined by
\begin{equation*}
T_{\#}:=\left(\frac{1}{2}\operatorname{Id}+\mathcal{K}_{\#}\right)^{-1},
\end{equation*}
where the boundary-to-boundary integral operator $\mathcal{K}_{\#}$ is defined by
\begin{align*}
\mathcal{K}_{\mathrm{D}}g(x)&:=\operatorname{p.v.}\int_{\partial\Omega}
\partial_{n_{w}}G(x,w)g(w)\,\mathrm{d}S_w,\quad x\in\partial\Omega,\\
\mathcal{K}_{\mathrm{N}}g(x)&:=\operatorname{p.v.}\int_{\partial\Omega}
-\partial_{n_{x}}G(x,w)g(w)\,\mathrm{d}S_w,\quad x\in\partial\Omega.
\end{align*}
For convenience, we denote the kernels of $\mathcal{S}_{\mathrm{D}},\mathcal{S}_{\mathrm{N}},\mathcal{K}_{\mathrm{D}},\mathcal{K}_{\mathrm{N}}$ by
\begin{align*}
&S_{\mathrm{D}}(x,w):=\partial_{n_{w}}G(x,w),\quad S_{\mathrm{N}}(x,w):=-G(x,w),\quad x\in\Omega,\,w\in\partial\Omega,\\
&K_{\mathrm{D}}(x,w):=\partial_{n_{w}}G(x,w),\quad K_{\mathrm{N}}(x,w):=-\partial_{n_{x}}G(x,w),\quad x,w\in\partial\Omega.
\end{align*}


First of all, to deal with the operator $\mathcal{K}_{\#}$, we need the following lemmas.
\begin{lemma}\label{lem-geometric-lemma}
There exists a constant $C>0$ such that $\forall x,w\in\partial\Omega,\,x\ne w$
\begin{align*}
\vert n_x\cdot(x-w)\vert&\le C\vert x-w\vert^2,\\
\vert K_{\#}(x,w)\vert&\le C\vert x-w\vert^{2-d}.
\end{align*}
\end{lemma}
\begin{proof}
The first inequality is a well known geometric lemma, see \cite[Chapter 1]{Gun67}. The second can be deduced immediately from the first.
\end{proof}

\begin{lemma}\label{lem-boundary-boundary-integral}
Let $-\infty<d_1,d_2<d-1$, there exists a constant $C>0$ such that $\forall w,y\in\partial\Omega,\,w\ne y$,
\begin{equation*}
\int_{\partial\Omega}\frac{1}{\vert w-z\vert^{d_1}\vert y-z\vert^{d_2}}\,\mathrm{d}S_z\le \begin{cases}
C\vert w-y\vert^{d-1-d_1-d_2},&\text{\rm if}\;\;d_1+d_2>d-1,\\
C(1+\vert\ln\vert w-y\vert\vert),&\text{\rm if}\;\;d_1+d_2=d-1,\\
C,&\text{\rm if}\;\;d_1+d_2<d-1.
\end{cases}
\end{equation*}
\end{lemma}
\begin{proof}
These estimates can be deduced by the fact that $\partial\Omega$ is a compact Lyapunov boundary or by the fact from the integration in quasi-metric space, see e.g., \cite[Lemmas 7.2.1-7.2.2]{MMM22}. We omit it for brevity.
\end{proof}

To deal with $\mathcal{S}_{\#}$, we need the following lemmas.
\begin{lemma}\label{lem-boundary-domain-integral-d-1-a}
Let $a\in\mathbb{R}$, there exists a constant $C>0$ such that $\forall x\in\Omega$,
\begin{equation*}
\int_{\partial\Omega}\frac{1}{\vert x-w\vert^a}\,\mathrm{d}S_w\le\begin{cases}
C\operatorname{d}_{\partial\Omega}(x)^{d-1-a},&\text{\rm if}\;\;a>d-1,\\
C(1+\vert\ln\operatorname{d}_{\partial\Omega}(x)\vert),&\text{\rm if}\;\;a=d-1,\\
C,&\text{\rm if}\;\;a<d-1.
\end{cases}
\end{equation*}
\end{lemma}
\begin{proof}
Note that $\Omega$ satisfies \eqref{eq-assum-Omega}. Hence for all $x\in\Omega$ satisfying $\operatorname{d}_{\partial\Omega}(x)\le \delta_0$, it holds
\begin{equation*}
n_{x_{\parallel}}\parallel x-x_{\parallel},\quad\vert x-x_{\parallel}\vert=\operatorname{d}_{\partial\Omega}(x).
\end{equation*}

Now we fix a point $x\in\Omega$. If $\operatorname{d}_{\partial\Omega}(x)\ge \delta_0/2$, note $\vert x-w\vert\ge\operatorname{d}_{\partial\Omega}(x)$, it is obvious that
\begin{align*}
\int_{\partial\Omega}\frac{1}{\vert x-w\vert^a}\,\mathrm{d}S_w&\le\frac{\vert\partial\Omega\vert}{\operatorname{d}_{\partial\Omega}(x)^{a}}\\
&\le\vert\partial\Omega\vert\min\left(\frac{2^{d-1}\delta_0^{-d+1}}{\operatorname{d}_{\partial\Omega}(x)^{a-d+1}},2^{a}\delta_0^{-a}(1+\vert\ln\operatorname{d}_{\partial\Omega}(x)\vert),2^{a}\delta_0^{-a}\right).
\end{align*}

If $\operatorname{d}_{\partial\Omega}(x)<\delta_0/2$, we split the integral into two parts:
\begin{equation*}
\int_{\partial\Omega}\frac{1}{\vert x-w\vert^a}\,\mathrm{d}S_w=\int_{A_1}\frac{1}{\vert x-w\vert^a}\,\mathrm{d}S_w+\int_{A_2}\frac{1}{\vert x-w\vert^a}\,\mathrm{d}S_w,
\end{equation*}
where
\begin{equation*}
A_1=\partial\Omega\cap B_{x_{\parallel},2\operatorname{d}_{\partial\Omega}(x)},\quad A_2=\partial\Omega\setminus A_1.
\end{equation*}
Since $\partial\Omega$ is $C^4$, there exists a constant $C>0$ such that $\vert A_1\vert\le C\operatorname{d}_{\partial\Omega}(x)^{d-1}$. Note $\vert x-w\vert\ge\operatorname{d}_{\partial\Omega}(x)$, 
\begin{equation*}
\int_{A_1}\frac{1}{\vert x-w\vert^a}\,\mathrm{d}S_w\le \frac{\vert A_1\vert}{\operatorname{d}_{\partial\Omega}(x)^{a}}\le \frac{C}{\operatorname{d}_{\partial\Omega}(x)^{a-d+1}}.
\end{equation*}
For $w\in A_2$, we have $\vert x-w\vert\ge \vert x_{\parallel}-w\vert-\vert x-x_{\parallel}\vert\ge\frac{1}{2}\vert x_{\parallel}-w\vert$. By the integration in quasi-metric space, \cite[Lemma 7.2.1]{MMM22}, we have
\begin{align*}
\int_{A_2}\frac{1}{\vert x-w\vert^a}\,\mathrm{d}S_w&\le \int_{\partial\Omega\setminus B_{x_{\parallel},2\operatorname{d}_{\partial\Omega}(x)}}\frac{C}{\vert x_{\parallel}-w\vert^a}\,\mathrm{d}S_w\\
&\le\begin{cases}
C\operatorname{d}_{\partial\Omega}(x)^{d-1-a},&\text{\rm if}\;\;a>d-1,\\
C(1+\vert\ln\operatorname{d}_{\partial\Omega}(x)\vert),&\text{\rm if}\;\;a=d-1,\\
C,&\text{\rm if}\;\;a<d-1.
\end{cases}
\end{align*}
This ends the proof.
\end{proof}

\begin{lemma}\label{lem-boundary-domain-integral}
There exists a constant $C>0$ such that 
\begin{equation*}
\sup_{x\in\Omega}\int_{\partial\Omega}\frac{\vert n_w\cdot(x-w)\vert}{\vert x-w\vert^d}\,\mathrm{d}S_w\le C.
\end{equation*}
\end{lemma}
\begin{proof}
The proof is similar to that of Lemma~\ref{lem-boundary-domain-integral-d-1-a}. We fix a point $x\in\Omega$. If $\operatorname{d}_{\partial\Omega}(x)\ge \delta_0/2$, it is obvious that
\begin{equation*}
\int_{\partial\Omega}\frac{\vert n_w\cdot(x-w)\vert}{\vert x-w\vert^d}\,\mathrm{d}S_w\le 2^{d-1}\delta_0^{1-d}\vert\partial\Omega\vert.
\end{equation*}

If $\operatorname{d}_{\partial\Omega}(x)<\delta_0/2$, we split the integral into three parts:
\begin{equation*}
\int_{\partial\Omega}\frac{\vert n_w\cdot(x-w)\vert}{\vert x-w\vert^d}\,\mathrm{d}S_w=\sum_{i=1}^{3}\int_{A_i}\frac{\vert n_w\cdot(x-w)\vert}{\vert x-w\vert^d}\,\mathrm{d}S_w=:\sum_{i=1}^{3}I_i,
\end{equation*}
where
\begin{equation*}
A_1=\partial\Omega\cap B_{x_{\parallel},2\operatorname{d}_{\partial\Omega}(x)},\quad A_2=\partial\Omega\cap B_{x_{\parallel},\delta_0}\setminus A_1,\quad A_3=\partial\Omega\setminus(A_1\cup A_2).
\end{equation*}
For $w\in A_1$, we have $\vert A_1\vert\le C\operatorname{d}_{\partial\Omega}(x)^{d-1}$ and $\vert x-w\vert\ge\operatorname{d}_{\partial\Omega}(x)$, hence
\begin{equation*}
I_1\le\operatorname{d}_{\partial\Omega}(x)^{1-d}\vert A_1\vert\le C.
\end{equation*}
For $w\in A_3$, we have $\vert x-w\vert\ge\vert x_{\parallel}-w\vert-\vert x-x_{\parallel}\vert\ge \delta_0$, hence
\begin{equation*}
I_3\le 2^{d-1}\delta_0^{1-d}\vert A_3\vert\le C.
\end{equation*}
For $w\in A_2$, by Lemma~\ref{lem-geometric-lemma}, we have
\begin{equation*}
\vert n_w\cdot(x-w)\vert\le\vert n_w\cdot(x_{\parallel}-w)\vert+\vert n_w\cdot(x_{\parallel}-x)\vert\le C\vert x_{\parallel}-w\vert^{2}+\operatorname{d}_{\partial\Omega}(x).
\end{equation*}
Hence by Lemma~\ref{lem-boundary-domain-integral-d-1-a},
\begin{equation*}
I_2\le\int_{A_2}\frac{C}{\vert x-w\vert^{d-2}}\,\mathrm{d}S_w+\int_{A_2}\frac{\operatorname{d}_{\partial\Omega}(x)}{\vert x-w\vert^{d}}\,\mathrm{d}S_w\le C.
\end{equation*}
Summing up the estimates of $I_1,I_2,I_3$, we get the conclusion.
\end{proof}

By the above lemmas, we get the following properties of $\mathcal{K}_{\#},\mathcal{S}_{\#}$.
\begin{lemma}\label{lem-Km-L1-Linfty}
Let $\mathcal{K}_{\#}^i$ denote the $i$-th composition of $\mathcal{K}_{\#}$, and $K_{i,\#}(x,y)$ the corresponding integral kernel. Then $K_{1,\#}(x,y)$ is bounded for $d=2$. For $d\ge 3$, then $\forall x,y\in\partial\Omega,\,x\ne y$,
\begin{equation*}
\vert K_{i,\#}(x,y)\vert\le \begin{cases}
C\vert x-y\vert^{i+1-d},&\text{\rm if}\;\;1\le i\le d-2,\\
C(1+\vert \ln\vert x-y\vert\vert),&\text{\rm if}\;\;i=d-1,\\
C,&\text{\rm if}\;\;i=d.
\end{cases}
\end{equation*}

As a consequence, $\mathcal{K}_{\#}$ is bounded from $L^p(\partial\Omega)$ to $L^q(\partial\Omega)$ for $q\in[1,\frac{(d-1)p}{d-1-p})$ if $p\in[1,d-1)$, for $q=\infty$ if $p>d-1$ and for $q=\infty$ if $p\ge 1$, $d=2$. $\mathcal{K}_{\#}^m$ is bounded from $L^{1}(\partial\Omega)$ to $L^{\infty}(\partial\Omega)$ for $m=d$ if $d\ge 3$ and for $m=1$ if $d=2$.
\end{lemma}
\begin{proof}
The case $d=2$ is trivial by Lemma~\ref{lem-geometric-lemma}. For $d\ge 3$, by the definition of $\mathcal{K}_{\#}$,
\begin{equation*}
K_{1,\mathrm{D}}(x,y)=\partial_{n_y}G(x,y),\quad K_{1,\mathrm{N}}(x,y)=-\partial_{n_x}G(x,y).
\end{equation*}
It is no difficult to deduce by Fubini-Tonelli theorem that
\begin{equation*}
K_{i+1,\#}(x,y)=\int_{\partial\Omega}K_{1,\#}(x,w)K_{i,\#}(w,y)\,\mathrm{d}S_w.
\end{equation*}
Note by Lemma~\ref{lem-geometric-lemma}, we have $\vert K_{1,\#}(x,y)\vert\le C\vert x-y\vert^{2-d}$. It can be deduced by induction and Lemma~\ref{lem-boundary-boundary-integral} that for $1\le i\le d-2$
\begin{equation*}
\vert K_{i,\#}(x,y)\vert\le C\vert x-y\vert^{i+1-d}.
\end{equation*}
Applying Lemma~\ref{lem-boundary-boundary-integral} again for $i=d-1,d$, we obtain
\begin{equation*}
\vert K_{d-1,\#}(x,y)\vert\le C(1+\vert \ln\vert x-y\vert\vert),\quad\vert K_{d,\#}(x,y)\vert\le C.
\end{equation*}

The consequence can be deduced by the argument of proving Young's inequality and the estimates of the kernel $K_{i,\#}(x,y)$.
\end{proof}

\begin{lemma}\label{lem-S-Linfty-Linfty}
$\mathcal{S}_{\#}$ is a bounded operator from $L^{\infty}(\partial\Omega)$ to $L^{\infty}(\Omega)$. 
\end{lemma}
\begin{proof}
For the single layer potential $\mathcal{S}_{\mathrm{N}}$, by Lemma~\ref{lem-boundary-domain-integral-d-1-a},
\begin{equation*}
\| \mathcal{S}_{\mathrm{N}}g\|_{L^{\infty}(\Omega)}\le C\| g\|_{L^{\infty}(\partial\Omega)}\sup_{x\in\Omega}\int_{\partial\Omega}\frac{1}{\vert x-w\vert^{d-2}}\,\mathrm{d}S_w\le C\| g\|_{L^{\infty}(\partial\Omega)}.
\end{equation*}

For the double layer potential $\mathcal{S}_{\mathrm{D}}$, by Lemma~\ref{lem-boundary-domain-integral},
\begin{equation*}
\| \mathcal{S}_{\mathrm{D}}g\|_{L^{\infty}(\Omega)}\le C\| g\|_{L^{\infty}(\partial\Omega)}\sup_{x\in\Omega}\int_{\partial\Omega}\frac{\vert n_{w}\cdot(x-w)\vert}{\vert x-w\vert^{d}}\,\mathrm{d}S_w\le C\| g\|_{L^{\infty}(\partial\Omega)}.
\end{equation*}
\end{proof}

Finally, we need a conclusion for the $L^1$ boundedness of $T_{\#}[\bar{h}_{\#}]$, which has been established in \cite[p227, (21)]{Gun67} for $d=3$.
\begin{lemma}\label{lem-T-L1-L1}
Let $g$ is a continuous function on $\partial\Omega$, then
\begin{equation*}
\|T_{\#}g\|_{L^1(\partial\Omega)}\le C\|g\|_{L^1(\partial\Omega)}.
\end{equation*}
\end{lemma}
\begin{proof}
We can imitate the proof of \cite[p227, (21)]{Gun67} line by line, which is based on the expansion of the solutions to the Fredholm integral equation of the second kind in \cite[Chapter 3]{Gun67}. We omit the details for brevity.
\end{proof}

\subsection{Estimates on the Green functions}

\begin{lemma}\label{lem-estimate-Green}
Let $G_{\#}$ be a Green function on $\Omega$, then there exists a constant $C>0$ such that $\forall x,y\in\Omega$, $x\ne y$,
\begin{equation}\label{eq-Green-elementary-estimate-1}
\vert G_{\#}(x,y)\vert\le\begin{cases}
C(1+\vert\ln\vert x-y\vert\vert),&\text{\rm if}\;\; d=2,\\
C\vert x-y\vert^{2-d},&\text{\rm if}\;\; d\ge3,
\end{cases}
\end{equation}
and
\begin{equation}\label{eq-Green-elementary-estimate-2}
\vert \nabla_xG_{\#}(x,y)\vert\le C\vert x-y\vert^{1-d},\quad\vert \nabla_x^2G_{\#}(x,y)\vert\le C\vert x-y\vert^{-d}.
\end{equation}
In particular, there exists $A_{\Omega}\ge 0$ such that
\begin{equation}\label{eq-Green-upper-bound}
\sup_{x,y\in\Omega}G_{\mathrm{D}}(x,y)\le 0,\quad\sup_{x,y\in\Omega}G_{\mathrm{N}}(x,y)\le A_{\Omega}.
\end{equation}
\end{lemma}
\begin{proof}
The pointwise estimates \eqref{eq-Green-elementary-estimate-1}-\eqref{eq-Green-elementary-estimate-2} are well known, which can be proved by the layer potential technique for $d\ge 3$, see e.g., \cite{Ei58} for details. For $\#=\mathrm{D}$, $d\ge 3$, \eqref{eq-Green-upper-bound} can be deduced immediately by the maximum principle. To prove \eqref{eq-Green-upper-bound} for $\#=\mathrm{N}$, $d\ge 3$, we note in \cite{Ei58}, it has been proved that there exists a small constant $\gamma>0$ such that
\begin{equation*}
\vert G_{\mathrm{N}}(x,y)-\bar{G}_{\mathrm{N}}(x,y)\vert\le C\vert x-y\vert^{2-d+\gamma},
\end{equation*}
where
\begin{equation*}
\bar{G}_{\mathrm{N}}(x,y):=G(x,y)-2\int_{\partial\Omega}G(x,z)\bar{h}_{\mathrm{N}}(z,y)\,\mathrm{d}S_z.
\end{equation*}
Note there exists a constant $\varepsilon>0$ such that for $\vert x-y\vert\le \varepsilon$
\begin{equation*}
G(x,y)+C\vert x-y\vert^{2-d+\gamma}\le 0.
\end{equation*}
Hence we have
\begin{align*}
G_{\mathrm{N}}(x,y)&\le\bar{G}_{\mathrm{N}}(x,y)+C\vert x-y\vert^{2-d+\gamma}\\
&\le C\varepsilon^{2-d+\gamma}-2\int_{\partial\Omega}G(x,z)\bar{h}_{\mathrm{N}}(z,y)\,\mathrm{d}S_z\\
&=C\varepsilon^{2-d+\gamma}+2\int_{\partial\Omega}G(x,z)\partial_{n_z}[G(z,y)+\tilde{g}(z,y)]\,\mathrm{d}S_z
\end{align*}  
By Lemma~\ref{lem-boundary-domain-integral-d-1-a}
\begin{align*}
\tilde{A}_{\Omega}:=&\sup_{x,y\in\Omega}\left\vert 2\int_{\partial\Omega}G(x,z)\partial_{n_z}\tilde{g}(z,y)\,\mathrm{d}S_z\right\vert\\
=&\sup_{x,y\in\Omega}\left\vert-2\int_{\partial\Omega}G(x,z)\frac{(z-y)\cdot n_z}{d\vert\Omega\vert}\,\mathrm{d}S_z\right\vert<\infty.
\end{align*}
By the convexity of $\Omega$,
\begin{equation*}
\partial_{n_z}G(z,y)=\frac{(z-y)\cdot n_z}{\omega_d\vert z-y\vert^d}\ge 0.
\end{equation*}
Hence
\begin{align*}
G_{\mathrm{N}}(x,y)&\le \tilde{A}_{\Omega}+C\varepsilon^{2-d+\gamma}+2\int_{\partial\Omega}G(x,z)\partial_{n_z}G(z,y)\,\mathrm{d}S_z\\
&\le \tilde{A}_{\Omega}+C\varepsilon^{2-d+\gamma}=:A_{\Omega}.
\end{align*} 

For $d=2$, these estimates can be proved by the conformal mapping method. Indeed, according to \cite[Appendix]{BV07}, there exists a conformal mapping $\psi\in C^{3,a}(\bar{\Omega};B_1^2)$ and a mapping $w\in C^{4,a}(\bar{\Omega})$ with $0<a<1$ such that
\begin{align*}
G_{\Omega,\mathrm{D}}(x,y)&=G_{B_1^2,\mathrm{D}}(\psi(x),\psi(y)),\\
G_{\Omega,\mathrm{N}}(x,y)&=G_{B_1^2,\mathrm{N}}(\psi(x),\psi(y))+w(x)+w(y).
\end{align*}
According to the representations given in Subsection~\ref{subsec-special-Green}, it is no difficult to check that \eqref{eq-Green-elementary-estimate-1}-\eqref{eq-Green-upper-bound} hold.

\end{proof}

\subsection{Estimates on the Robin functions}

Now we deduce the crucial proposition in this section.

\begin{proposition}\label{prn-Robin-estimate}
There exist constants $c_1,c_2>0$ such that $\forall x\in\Omega$, there hold
\begin{equation*}
R_{\mathrm{D}}(x)\ge\begin{cases}
c_2\operatorname{d}_{\partial\Omega}(x)^{2-d}-c_1,&d\ge 3,\\
-c_2\ln\operatorname{d}_{\partial\Omega}(x)-c_1,&d=2,
\end{cases}
\qquad R_{\mathrm{N}}(x)\le\begin{cases}
c_1-c_2\operatorname{d}_{\partial\Omega}(x)^{2-d},&d\ge 3,\\
c_1+c_2\ln\operatorname{d}_{\partial\Omega}(x),&d=2.
\end{cases}
\end{equation*}
\end{proposition}

\begin{proof}
In the following, we will denote $\bar{c}_1>0$ as an inessential big constant, $\bar{c}_2>0$ as an inessential small constant and they may differ on different occasions. Then $c_1$ is the biggest of those $\bar{c}_1$ and $c_2$ is the smallest of those $\bar{c}_2$.

Denote 
\begin{equation*}
k_{\#}(x,y):=T_{\#}[\bar{h}_{\#}(\cdot,y)](x),\quad x\in\partial\Omega,\,y\in\Omega.
\end{equation*}
By the definition of $T_{\#}$, we have
\begin{equation*}
\frac{1}{2}k_{\#}(x,y)+\mathcal{K}_{\#}[k_{\#}(\cdot,y)](x)=\bar{h}_{\#}(x,y),\quad x\in\partial\Omega,\,y\in\Omega.
\end{equation*}
Applying the operator $\mathcal{K}_{\#}$ on the above equality $d-1$ times to obtain
\begin{align*}
&\frac{1}{2}\mathcal{K}_{\#}[k_{\#}(\cdot,y)](x)+\mathcal{K}_{\#}^2[k_{\#}(\cdot,y)](x)=\mathcal{K}_{\#}[\bar{h}_{\#}(\cdot,y)](x),\\
&\cdots\\
&\frac{1}{2}\mathcal{K}_{\#}^{d-1}[k_{\#}(\cdot,y)](x)+\mathcal{K}_{\#}^d[k_{\#}(\cdot,y)](x)=\mathcal{K}_{\#}^{d-1}[\bar{h}_{\#}(\cdot,y)](x).
\end{align*}
Then we can cancel the terms $\mathcal{K}_{\#}^i[k_{\#}(\cdot,y)](x)$, $i=1,2,\dots,d-1$ by the above equations to obtain
\begin{align*}
\frac{1}{2}k_{\#}(x,y)-\bar{h}_{\#}(x,y)&=-(-2)^{d-1}\mathcal{K}_{\#}^d[k_{\#}(\cdot,y)](x)+\sum_{i=1}^{d-1}(-2)^i\mathcal{K}_{\#}^i[\bar{h}_{\#}(\cdot,y)](x)\\
&=:k_{\#,d}(x,y).
\end{align*}
Applying the operator $\mathcal{S}_{\#}$ on the above equality to obtain
\begin{equation*}
\frac{1}{2}\mathcal{S}_{\#}[k_{\#}(\cdot,y)](x)-\mathcal{S}_{\#}[\bar{h}_{\#}(\cdot,y)](x)=\mathcal{S}_{\#}[k_{\#,d}(\cdot,y)](x),\quad x,y\in\Omega.
\end{equation*}
Hence by the representation \eqref{eq-layer-representation} of $g_{\#}$, we have
\begin{equation*}
g_{\#}(x,y)=2\mathcal{S}_{\#}[\bar{h}_{\#}(\cdot,y)](x)+2\mathcal{S}_{\#}[k_{\#,d}(\cdot,y)](x),
\end{equation*}
and
\begin{equation}\label{eq-rewrite-R}
R_{\#}(x)=2\mathcal{S}_{\#}[\bar{h}_{\#}(\cdot,x)](x)+2\mathcal{S}_{\#}[k_{\#,d}(\cdot,x)](x).
\end{equation}
Now we focus on the analysis on the right hand side. 

\noindent {\bf Step 1. The estimates on $\mathcal{S}_{\#}[k_{\#,d}]$}

By the definition of $k_{\#,d}$,
\begin{align*}
\mathcal{S}_{\#}[k_{\#,d}(\cdot,x)](x)=&-(-2)^{d-1}\mathcal{S}_{\#}[\mathcal{K}_{\#}^d[T_{\#}[\bar{h}_{\#}(\cdot,x)](\cdot)](\cdot)](x)\\
&+\sum_{i=1}^{d-1}(-2)^i\mathcal{S}_{\#}[\mathcal{K}_{\#}^i[\bar{h}_{\#}(\cdot,x)](\cdot)](x).
\end{align*}
For the first term on the right hand side, by Lemmas~\ref{lem-Km-L1-Linfty}-\ref{lem-S-Linfty-Linfty}, we have
\begin{align*}
\|\mathcal{S}_{\#}[\mathcal{K}_{\#}^d[T_{\#}[\bar{h}_{\#}(\cdot,x)](\cdot)](\cdot)](x)\|_{L_x^{\infty}(\Omega)}&\le \bar{c}_1\|\mathcal{K}_{\#}^d[T_{\#}[\bar{h}_{\#}(\cdot,y)](\cdot)](x)\|_{L_{y}^{\infty}(\Omega;L_{x}^{\infty}(\Omega))}\\
&\le \bar{c}_1\|T_{\#}[\bar{h}_{\#}(\cdot,y)](x)\|_{L_{y}^{\infty}(\Omega;L_x^{1}(\partial\Omega))}.
\end{align*}
By Lemma~\ref{lem-T-L1-L1}, the definition of $\bar{h}_{\#}$ and Lemmas~\ref{lem-boundary-domain-integral-d-1-a}-\ref{lem-boundary-domain-integral}, we have
\begin{equation*}
\|T_{\#}[\bar{h}_{\#}(\cdot,y)](x)\|_{L_{y}^{\infty}(\Omega;L_x^{1}(\partial\Omega))}\le \bar{c}_1\|\bar{h}_{\#}(x,y)\|_{L_{y}^{\infty}(\Omega;L_x^{1}(\partial\Omega))}\le \bar{c}_1.
\end{equation*}
Hence, we have
\begin{equation*}
\|\mathcal{S}_{\#}[\mathcal{K}_{\#}^d[T_{\#}[\bar{h}_{\#}(\cdot,x)](\cdot)](\cdot)](x)\|_{L_x^{\infty}(\Omega)}\le \bar{c}_1.
\end{equation*}

Now we estimate the summation term on the right hand side. By Lemma~\ref{lem-Km-L1-Linfty}, for $1\le i\le d-1$,
\begin{align*}
\left\vert\mathcal{S}_{\#}[\mathcal{K}_{\#}^i[\bar{h}_{\#}(\cdot,x)](\cdot)](x)\right\vert&\le\int_{\partial\Omega}\int_{\partial\Omega}\vert S_{\#}(x,z)\vert\vert K_{i,\#}(z,w)\vert\vert\bar{h}_{\#}(w,x)\vert \,\mathrm{d}S_w\,\mathrm{d}S_z\\
&\le \bar{c}_1\int_{\partial\Omega}\int_{\partial\Omega}\frac{\vert S_{\#}(x,z)\vert\vert\bar{h}_{\#}(w,x)\vert}{\vert z-w\vert^{d-2}}\,\mathrm{d}S_w\,\mathrm{d}S_z.
\end{align*}
By the definitions of $S_{\#},\bar{h}_{\#}$, we have
\begin{align*}
\vert S_{\mathcal{D}}(x,z)\vert\vert\bar{h}_{\mathcal{D}}(w,x)\vert&\le\vert\partial_{n_z}G(x,z)\vert\vert G(w,x)\vert,\\
\vert S_{\mathcal{N}}(x,z)\vert\vert\bar{h}_{\mathcal{N}}(w,x)\vert&\le\vert G(x,z)\vert(\vert \partial_{n_w}G(w,x)\vert+\vert\partial_{n_w}\tilde{g}(w,x)\vert)
\end{align*}
Note $\vert\partial_{n_w}\tilde{g}(w,x)\vert\le \bar{c}_1$. Hence by the symmetry of $w,z$, we have for $d\ge 3$
\begin{align*}
&\left\vert\mathcal{S}_{\#}[\mathcal{K}_{\#}^i[\bar{h}_{\#}(\cdot,x)](\cdot)](x)\right\vert\\
&\le \bar{c}_1\int_{\partial\Omega}\int_{\partial\Omega}\frac{1}{\vert x-z\vert^{d-2}\vert z-w\vert^{d-2}}\left(\frac{\vert n_w\cdot(w-x)\vert}{\vert w-x\vert^{d}}+1\right)\,\mathrm{d}S_w\,\mathrm{d}S_z.
\end{align*}
By Fubini-Tonelli theorem and Lemma~\ref{lem-boundary-domain-integral}
\begin{align*}
&\int_{\partial\Omega}\int_{\partial\Omega}\frac{1}{\vert x-z\vert^{d-2}\vert z-w\vert^{d-2}}\left(\frac{\vert n_w\cdot(w-x)\vert}{\vert w-x\vert^{d}}+1\right)\,\mathrm{d}S_w\,\mathrm{d}S_z\\
&\le\sup_{w\in\partial\Omega}\int_{\partial\Omega}\frac{1}{\vert x-z\vert^{d-2}\vert z-w\vert^{d-2}}\,\mathrm{d}S_z\int_{\partial\Omega}\frac{\vert n_w\cdot(w-x)\vert}{\vert w-x\vert^{d}}+1\,\mathrm{d}S_w\\
&\le \bar{c}_1\sup_{w\in\partial\Omega}\int_{\partial\Omega}\frac{1}{\vert x-z\vert^{d-2}\vert z-w\vert^{d-2}}\,\mathrm{d}S_z.
\end{align*}
Note $\vert x-z\vert\ge \operatorname{d}_{\partial\Omega}(x)$, by H\"older's inequality and Lemmas~\ref{lem-boundary-boundary-integral}-\ref{lem-boundary-domain-integral-d-1-a}, we have
\begin{align*}
&\int_{\partial\Omega}\frac{1}{\vert x-z\vert^{d-2}\vert z-w\vert^{d-2}}\,\mathrm{d}S_z\\
&\le\operatorname{d}_{\partial\Omega}(x)^{\frac{5}{2}-d}\int_{\partial\Omega}\frac{1}{\vert x-z\vert^{\frac{1}{2}}\vert z-w\vert^{d-2}}\,\mathrm{d}S_z\\
&\le\operatorname{d}_{\partial\Omega}(x)^{\frac{5}{2}-d}\left(\int_{\partial\Omega}\frac{1}{\vert x-z\vert^{\frac{3(d-1)}{4}}}\,\mathrm{d}S_z\right)^{\frac{2}{3(d-1)}}\left(\int_{\partial\Omega}\frac{1}{\vert z-w\vert^{\frac{(d-2)(d-1)}{d-\frac{8}{3}}}}\,\mathrm{d}S_z\right)^{\frac{d-\frac{8}{3}}{d-1}}\\
&\le \bar{c}_1\operatorname{d}_{\partial\Omega}(x)^{\frac{5}{2}-d}.
\end{align*}

For $d=2$, it is obvious by Lemmas~\ref{lem-boundary-domain-integral-d-1-a}-\ref{lem-boundary-domain-integral} that
\begin{align*}
&\left\vert\mathcal{S}_{\#}[\mathcal{K}_{\#}^i[\bar{h}_{\#}(\cdot,x)](\cdot)](x)\right\vert\\
&\le \bar{c}_1\int_{\partial\Omega}\int_{\partial\Omega}\vert\ln\vert x-z\vert\vert\left(\frac{\vert n_w\cdot(w-x)\vert}{\vert w-x\vert^{2}}+1\right)\,\mathrm{d}S_w\,\mathrm{d}S_z\le \bar{c}_1.
\end{align*}

In summary, we have
\begin{equation}\label{eq-Skd-estimate}
\vert\mathcal{S}_{\#}[k_{\#,d}(\cdot,x)](x)\vert\le\begin{cases}
\bar{c}_1\left[1+\operatorname{d}_{\partial\Omega}(x)^{\frac{5}{2}-d}\right],& d\ge 3,\\
\bar{c}_1,& d=2.
\end{cases}
\end{equation}

\noindent {\bf Step 2. The estimates on $\mathcal{S}_{\#}[\bar{h}_{\#}]$}

By the definitions of $\mathcal{S}_{\#}$ and $\bar{h}_{\#}$, we have
\begin{align*}
\mathcal{S}_{\mathrm{D}}[\bar{h}_{\mathrm{D}}(\cdot,x)](x)&=-\int_{\partial\Omega}\partial_{n_w}G(x,w)G(x,w)\,\mathrm{d}S_w\\
&=-\frac{1}{2}\int_{\partial\Omega}\partial_{n_w}\vert G(x,w)\vert^2\,\mathrm{d}S_w,\\
\mathcal{S}_{\mathrm{N}}[\bar{h}_{\mathrm{N}}(\cdot,x)](x)&=\int_{\partial\Omega}G(x,w)\partial_{n_w}[G(x,w)+\tilde{g}(x,w)]\,\mathrm{d}S_w\\
&=\frac{1}{2}\int_{\partial\Omega}\partial_{n_w}\vert G(x,w)\vert^2\,\mathrm{d}S_w+\int_{\partial\Omega}G(x,w)\partial_{n_w}\tilde{g}(x,w)\,\mathrm{d}S_w.
\end{align*}
They contain the same boundary integral
\begin{equation*}
\bar{R}(x):=\int_{\partial\Omega}\partial_{n_w}\vert G(x,w)\vert^2\,\mathrm{d}S_w.
\end{equation*}
Note it is no difficult to bound the integral including $\tilde{g}$ as follows:
\begin{align*}
\int_{\partial\Omega}\vert G(x,w)\partial_{n_w}\tilde{g}(x,w)\vert\,\mathrm{d}S_w\le \bar{c}_1\int_{\partial\Omega}\vert G(x,w)\vert\,\mathrm{d}S_w\le \bar{c}_1.
\end{align*}

The estimates on $\bar{R}$ are more involved. We claim that
\begin{equation*}
\bar{R}(x)\le\begin{cases}
\frac{2^{1-d}}{(2-d)\omega_d}[\operatorname{d}_{\partial\Omega}(x)]^{2-d},&d\ge 3,\\
1+\vert\ln\operatorname{diam}\Omega\vert+\frac{1}{3\pi}\ln\operatorname{d}_{\partial\Omega}(x),&d=2.
\end{cases}
\end{equation*}
By the claim, we have for $\#=\mathrm{D}$,
\begin{equation}\label{eq-SbarhD-estimate}
\mathcal{S}_{\mathrm{D}}[\bar{h}_{\mathrm{D}}(\cdot,x)](x)\ge\begin{cases}
\bar{c}_2\operatorname{d}_{\partial\Omega}(x)^{2-d},& d\ge 3,\\
-\bar{c}_1-\bar{c}_2\ln\operatorname{d}_{\partial\Omega}(x),& d=2.
\end{cases}
\end{equation}
For $\#=\mathrm{N}$,
\begin{equation}\label{eq-SbarhN-estimate}
\mathcal{S}_{\mathrm{N}}[\bar{h}_{\mathrm{N}}(\cdot,x)](x)\le\begin{cases}
\bar{c}_1-\bar{c}_2\operatorname{d}_{\partial\Omega}(x)^{2-d},&d\ge 3,\\
\bar{c}_1+\bar{c}_2\ln\operatorname{d}_{\partial\Omega}(x),&d=2.
\end{cases}
\end{equation}

\noindent {\bf Step 3. Proof of the claim in Step 2 for $d\ge 3$}

We split $\bar{R}(x)$ into two parts,
\begin{align*}
\bar{R}(x)=&\int_{\partial\Omega\cup\partial\{\vert x-w\vert>\operatorname{d}_{\partial\Omega}(x)/2\}}\partial_{n_w}\vert G(x,w)\vert^2\,\mathrm{d}S_w\\
&-\int_{\partial\{\vert x-w\vert>\operatorname{d}_{\partial\Omega}(x)/2\}}\partial_{n_w}\vert G(x,w)\vert^2\,\mathrm{d}S_w.
\end{align*}
By the Gauss-Green formula
\begin{align*}
\bar{R}(x)=&\int_{\Omega\cap\{\vert x-w\vert>\operatorname{d}_{\partial\Omega}(x)/2\}}\Delta_w\vert G(x,w)\vert^2\,\mathrm{d}w\\
&-\int_{\vert x-w\vert=\operatorname{d}_{\partial\Omega}(x)/2}\frac{x-w}{\vert x-w\vert}\cdot\nabla_w\vert G(x,w)\vert^2\,\mathrm{d}S_w\\
=&2\int_{\Omega\cap\{\vert x-w\vert>\operatorname{d}_{\partial\Omega}(x)/2\}}\vert \nabla_wG(x,w)\vert^2\,\mathrm{d}w\\
&-\int_{\vert x-w\vert=\operatorname{d}_{\partial\Omega}(x)/2}\frac{x-w}{\vert x-w\vert}\cdot\nabla_w\vert G(x,w)\vert^2\,\mathrm{d}S_w.
\end{align*}
By the definition of $G$ in \eqref{eq-definition-G}, we have
\begin{align*}
\bar{R}(x)&=\frac{2}{\omega_d^2}\int_{\Omega\cap\{\vert x-w\vert>\operatorname{d}_{\partial\Omega}(x)/2\}}\frac{1}{\vert x-w\vert^{2d-2}}\,\mathrm{d}w-\frac{2^{d-1}}{(d-2)\omega_d[\operatorname{d}_{\partial\Omega}(x)]^{d-2}}\\
&=:I_1(x)-\frac{2^{d-1}}{(d-2)\omega_d[\operatorname{d}_{\partial\Omega}(x)]^{d-2}}.
\end{align*}
Note
\begin{align*}
I_1(x)=&\frac{2}{\omega_d^2}\int_{\Omega\cap\{\vert x-w\vert>2\operatorname{d}_{\partial\Omega}(x)\}}\frac{1}{\vert x-w\vert^{2d-2}}\,\mathrm{d}w\\
&+\frac{2}{\omega_d^2}\int_{\Omega\cap\{2\operatorname{d}_{\partial\Omega}(x)\ge\vert x-w\vert>\operatorname{d}_{\partial\Omega}(x)/2\}}\frac{1}{\vert x-w\vert^{2d-2}}\,\mathrm{d}w\\
\le&\frac{2}{\omega_d^2}\int_{\Omega\cap\{\vert x-w\vert>2\operatorname{d}_{\partial\Omega}(x)\}}\frac{1}{\vert x-w\vert^{2d-2}}\,\mathrm{d}w+\frac{2^{d-1}-2^{3-d}}{(d-2)\omega_d[\operatorname{d}_{\partial\Omega}(x)]^{d-2}}\\
=:&I_2(x)+\frac{2^{d-1}-2^{3-d}}{(d-2)\omega_d[\operatorname{d}_{\partial\Omega}(x)]^{d-2}}.
\end{align*}
By the convexity of the domain $\Omega$, we have
\begin{align*}
\Omega\cap\{\vert x-w\vert>2\operatorname{d}_{\partial\Omega}(x)\}\subset&\{(r,\theta_1,\dots,\theta_{d-1}): \operatorname{d}_{\partial\Omega}(x)<r/2\le \operatorname{diam}\Omega,\\
&\quad 0\le\theta_1/2,\theta_2,\dots,\theta_{d-2}\le \pi,\,\pi/3\le\theta_{d-1}\le\pi\}.
\end{align*}
Hence by the coordinate transformation,
\begin{equation*}
I_2(x)\le\frac{1}{\omega_d}\int_{\pi/3}^{\pi}\int_{2\operatorname{d}_{\partial\Omega}(x)}^{2\operatorname{diam}\Omega}r^{1-d}\sin\theta_{d-1}\,\mathrm{d}r\,\mathrm{d}\theta_{d-1}\le\frac{3\cdot 2^{1-d}}{(d-2)\omega_d[\operatorname{d}_{\partial\Omega}(x)]^{d-2}}.
\end{equation*}

In summary, we have
\begin{align*}
\bar{R}(x)&\le-\frac{2^{d-1}}{(d-2)\omega_d[\operatorname{d}_{\partial\Omega}(x)]^{d-2}}+\frac{2^{d-1}-2^{3-d}}{(d-2)\omega_d[\operatorname{d}_{\partial\Omega}(x)]^{d-2}}+\frac{3\cdot 2^{1-d}}{(d-2)\omega_d[\operatorname{d}_{\partial\Omega}(x)]^{d-2}}\\
&\le-\frac{2^{1-d}}{(d-2)\omega_d[\operatorname{d}_{\partial\Omega}(x)]^{d-2}}.
\end{align*}

\noindent {\bf Step 4. Proof of the claim in Step 2 for $d=2$}

The proof is similar to the case $d\ge 3$. By the Gauss-Green formula and the representation \eqref{eq-definition-G} of $G$, we have
\begin{align*}
\bar{R}(x)&=\frac{1}{2\pi^2}\int_{\Omega\cap\{\vert x-w\vert>\operatorname{d}_{\partial\Omega}(x)/2\}}\frac{1}{\vert x-w\vert^2}\,\mathrm{d}w+\frac{1}{\pi}\ln\frac{\operatorname{d}_{\partial\Omega}(x)}{2}\\
&=:J_1(x)+\frac{1}{\pi}\ln\frac{\operatorname{d}_{\partial\Omega}(x)}{2}.
\end{align*}
Note
\begin{align*}
J_1(x)=&\frac{1}{2\pi^2}\int_{\Omega\cap\{\vert x-w\vert>2\operatorname{d}_{\partial\Omega}(x)\}}\frac{1}{\vert x-w\vert^2}\,\mathrm{d}w\\
&+\frac{1}{2\pi^2}\int_{\Omega\cap\{2\operatorname{d}_{\partial\Omega}(x)\ge\vert x-w\vert>\operatorname{d}_{\partial\Omega}(x)/2\}}\frac{1}{\vert x-w\vert^2}\,\mathrm{d}w\\
\le&\frac{1}{2\pi^2}\int_{\Omega\cap\{\vert x-w\vert>2\operatorname{d}_{\partial\Omega}(x)\}}\frac{1}{\vert x-w\vert^2}\,\mathrm{d}w+\frac{1}{\pi}\ln 4\\
=:&J_2(x)+\frac{1}{\pi}\ln 4.
\end{align*}
By the convexity of the domain $\Omega$, we have
\begin{equation*}
\Omega\cap\{\vert x-w\vert>2\operatorname{d}_{\partial\Omega}(x)\}\subset\{(r,\theta_1): \operatorname{d}_{\partial\Omega}(x)<r/2\le \operatorname{diam}\Omega,\,\pi/3\le\theta_1\le 5\pi/3\}.
\end{equation*}
Hence by the coordinate transformation,
\begin{equation*}
J_2(x)\le\frac{1}{2\pi^2}\int_{\pi/3}^{5\pi/3}\int_{2\operatorname{d}_{\partial\Omega}(x)}^{2\operatorname{diam}\Omega}r^{-1}\,\mathrm{d}r\,\mathrm{d}\theta_1\\
\le\frac{2}{3\pi}\ln\frac{\operatorname{diam}\Omega}{\operatorname{d}_{\partial\Omega}(x)}.
\end{equation*}

In summary, we have
\begin{equation*}
\bar{R}(x)\le\frac{1}{\pi}\ln\frac{\operatorname{d}_{\partial\Omega}(x)}{2}+\frac{1}{\pi}\ln 4+\frac{2}{3\pi}\ln\frac{\operatorname{diam}\Omega}{\operatorname{d}_{\partial\Omega}(x)}\le 1+\vert\ln\operatorname{diam}\Omega\vert+\frac{1}{3\pi}\ln\operatorname{d}_{\partial\Omega}(x).
\end{equation*}

\noindent {\bf Step 5. End of the proof}

By \eqref{eq-rewrite-R}, \eqref{eq-Skd-estimate} and \eqref{eq-SbarhD-estimate}, we have
\begin{equation*}
R_{\mathrm{D}}(x)\ge \frac{\bar{c}_2}{\operatorname{d}_{\partial\Omega}(x)^{d-2}}-\bar{c}_1\left[1+\operatorname{d}_{\partial\Omega}(x)^{\frac{5}{2}-d}\right]\ge \frac{c_2}{\operatorname{d}_{\partial\Omega}(x)^{d-2}}-c_1,
\end{equation*}
for $d\ge 3$, and $R_{\mathrm{D}}(x)\ge -c_1-c_2\ln\operatorname{d}_{\partial\Omega}(x)$ for $d=2$.

By \eqref{eq-rewrite-R}, \eqref{eq-Skd-estimate} and \eqref{eq-SbarhN-estimate}, we have
\begin{equation*}
R_{\mathrm{N}}(x)\le \bar{c}_1-\frac{\bar{c}_2}{\operatorname{d}_{\partial\Omega}(x)^{d-2}}+\bar{c}_1\left[1+\operatorname{d}_{\partial\Omega}(x)^{\frac{5}{2}-d}\right]\le c_1-\frac{c_2}{\operatorname{d}_{\partial\Omega}(x)^{d-2}},
\end{equation*}
for $d\ge 3$, and $R_{\mathrm{N}}(x)\le c_1+c_2\ln\operatorname{d}_{\partial\Omega}(x)$ for $d=2$.

\end{proof}

\section{Local Well-posedness}\label{sec-local-well-posedness}

\subsection{elementary estimates on the electric field}

Note $h_{\mathrm{D}}\equiv 0$. The solution $\phi_{\#}$ of the Poisson equation \eqref{eq-Poisson} with boundary condition either \eqref{eq-Dirichlet-boundary} or \eqref{eq-Neumann-boundary} can be represented uniquely (up to a constant if $\#=\mathrm{N}$) using Green function:
\begin{equation*}
\phi_{\#}(t,x)=\int_{\Omega} G_{\#}(x,y)\rho(t,y)\,\mathrm{d}y-\int_{\partial\Omega} G_{\#}(x,y)h_{\#}(y)\,\mathrm{d}S_y,
\end{equation*}
and $E_{\#}=\nabla_x\phi_{\#}$ is given by
\begin{equation*}
E_{\#}(t,x)=\int_{\Omega}\nabla_xG_{\#}(x,y)\rho(t,y)\,\mathrm{d}y-\int_{\partial\Omega} \nabla_xG_{\#}(x,y)h_{\#}(y)\,\mathrm{d}S_y.
\end{equation*}

By Lemma~\ref{lem-estimate-Green}, we have
\begin{align*}
\left\vert \phi_{\#}(t,x)\right\vert&\le C\int_{\Omega} \frac{\rho(t,y)}{\vert x-y\vert}\,\mathrm{d}y+C\int_{\partial\Omega} \frac{h_{\#}(y)}{\vert x-y\vert}\,\mathrm{d}S_y,\\
\left\vert E_{\#}(t,x)\right\vert&\le C\int_{\Omega} \frac{\rho(t,y)}{\vert x-y\vert^2}\,\mathrm{d}y+C\int_{\partial\Omega} \frac{h_{\#}(y)}{\vert x-y\vert^2}\,\mathrm{d}S_y.
\end{align*}

The following estimates will be frequently used.
\begin{lemma}\label{lem-elliptic-estimate-electric-field}
Assume $\rho\in L^1\cap L^p(\Omega)$. Then 
\begin{align*}
\|E_{\#}(t)\|_{\frac{3p}{3-p}}&\le C(\|\rho(t)\|_{p}+1)\quad\text{\rm if}\;1<p<3,\\
\|E_{\#}(t)\|_{\infty}&\le C(\|\rho(t)\|_{p}+1)\quad\text{\rm if}\;3<p<\infty,
\end{align*}
where $C$ depends on $p$ and $\|h_{\#}\|_{\infty}$.
\end{lemma}
\begin{proof}
It can be proved by the Hardy-Littlewood-Sobolev inequality.
\end{proof}

\begin{lemma}\label{lem-pre-estimates-electric-field}
Assume $\rho(x)=\int_{\vert v-v_0\vert<R}f(x,v)\,\mathrm{d}v$ for $R>0$, $v_0\in\mathbb{R}^3$ and $f\in L^{\infty}(\Omega\times\mathbb{R}^3)$. Assume also $\rho\in L^{5/3}(\Omega)$, then
\begin{equation}\label{eq-pre-estimates-electric-field}
\int_{\Omega}\frac{\vert\rho(y)\vert}{\vert x-y\vert^2}\,\mathrm{d}y\le C\|f\|_{\infty}^{4/9}\|\rho\|_{5/3}^{5/9}R^{4/3}.
\end{equation}
\end{lemma}

\subsection{The conservation laws}

The energy of the plasma-charge model \eqref{eq-Vlasov}-\eqref{eq-Newton} is defined by
\begin{align*}
\mathcal{E}(t)=&\frac{1}{2}\iint_{\Omega\times\mathbb{R}^3}\left(\vert v\vert^2-\phi_{\#}\right)f(t,x,v)\,\mathrm{d}x\,\mathrm{d}v+\frac{1}{2}\int_{\partial\Omega}h_{\#}(x)\phi_{\#}(t,x)\,\mathrm{d}x\\
&+\sum_{\alpha}\left[\frac{1}{2}\vert \eta_{\alpha}(t)\vert^2-\phi_{\#}(t,\xi_{\alpha}(t))-H_{\#}(\xi_{\alpha}(t))\right]-\frac{1}{2}\sum_{\alpha\ne\beta}G_{\#}(\xi_{\alpha}(t),\xi_{\beta}(t))\\
&-\frac{\gamma}{2}\int_0^t\iint_{\partial\Omega\times\mathbb{R}^3}\vert v\vert^2v_{\perp}f(s,x,v)\,\mathrm{d}S_{x}\,\mathrm{d}v\,\mathrm{d}s,
\end{align*}
where $\gamma=1$ for the absorption case, $\gamma=0$ for the reflection case. Via the Poisson equation and integration by parts, it can be rewritten as
\begin{align*}
\mathcal{E}(t)=&\frac{1}{2}\iint_{\Omega\times\mathbb{R}^3}\vert v\vert^2f(t,x,v)\,\mathrm{d}x\,\mathrm{d}v+\frac{1}{2}\iint_{\Omega}\vert E_{\#}(t,x)\vert^2\,\mathrm{d}x\\
&+\sum_{\alpha}\left[\frac{1}{2}\vert \eta_{\alpha}(t)\vert^2-\phi_{\#}(t,\xi_{\alpha}(t))-H_{\#}(\xi_{\alpha}(t))\right]-\frac{1}{2}\sum_{\alpha\ne\beta}G_{\#}(\xi_{\alpha}(t),\xi_{\beta}(t))\\
&-\frac{\gamma}{2}\int_0^t\iint_{\partial\Omega\times\mathbb{R}^3}\vert v\vert^2v_{\perp}f(s,x,v)\,\mathrm{d}S_{x}\,\mathrm{d}v\,\mathrm{d}s.
\end{align*}

We collect the well-known a priori estimates on the classical solutions of \eqref{eq-Vlasov}-\eqref{eq-Newton} in the following proposition.
\begin{proposition}\label{prn-conservation-laws}
Let $f,\{\xi_{\alpha},\eta_{\alpha}\}$ be a classical solution given in Theorem~\ref{thm-local-wellposedness-absorption} or Theorem~\ref{thm-local-wellposedness-reflection}, then $\forall t\in[0,\mathring{T})$, $1\le p\le\infty$
\begin{equation}\label{eq-conservation-laws}
\|f(t)\|_p-\gamma\int_0^t\iint_{\partial\Omega\times\mathbb{R}^3}v_{\perp}\vert f(s,x,v)\vert^p\,\mathrm{d}S_{x}\,\mathrm{d}v\,\mathrm{d}s=\|\mathring{f}\|_p,\quad\mathcal{E}(t)=\mathcal{E}(0).
\end{equation}
For the Neumann case, there exist $C,\delta_2>0$ depending only on $\mathcal{E}(0)$, $\|\mathring{f}\|_{1}$, $\|\mathring{f}\|_{\infty}$, $\|h_{\mathrm{N}}\|_{\infty}$, $M$, $A_{\Omega}$, $\operatorname{diam}\Omega$, $s,r$ such that
\begin{gather}
\max_{\alpha}\sup_{t\in[0,\mathring{T})}\vert \eta_{\alpha}(t)\vert\le C,\quad\min_{\alpha}\inf_{t\in[0,\mathring{T})}\operatorname{d}_{\partial\Omega}(\xi_{\alpha}(t))\ge\delta_2,\label{eq-uniform-bound-point-charge-1}\\
\min_{\alpha\ne\beta}\inf_{t\in[0,\mathring{T})}\vert\xi_{\alpha}(t)-\xi_{\beta}(t)\vert\ge\delta_2,\label{eq-uniform-bound-point-charge-2}\\
\|\rho(t)\|_{s}\le C\quad\text{\rm if}\;1\le s\le \frac{5}{3},\label{eq-Lp-uniform-bound-rho}\\
\|E_{\mathrm{N}}(t)\|_{r}\le C\quad\text{\rm if}\;1\le r\le \frac{15}{4}.\label{eq-Lp-uniform-bound-electric-field}
\end{gather}
\end{proposition}

\begin{remark}\label{rem-Neumann-distance-lower-bound}
Note
\begin{equation*}
\mathcal{E}(t)=-\frac{1}{2}\sum_{\alpha}R_{\#}(\xi_{\alpha}(t))+\text{other terms}.
\end{equation*}
By Proposition~\ref{prn-Robin-estimate}, when $\operatorname{d}_{\partial\Omega}(\xi_{\alpha}(t))\to 0^+$, we have
\begin{equation*}
-\frac{1}{2}R_{\mathrm{D}}(\xi_{\alpha}(t))\to-\infty,\quad -\frac{1}{2}R_{\mathrm{N}}(\xi_{\alpha}(t))\to+\infty,
\end{equation*}
which implies that as a point charge approaches the boundary, the Dirichlet boundary has an attractive effect on the point charge, while the Neumann boundary has a repulsive effect on the point charge. 
\end{remark}

\begin{remark}
We emphasize that these uniform bounds \eqref{eq-uniform-bound-point-charge-1}-\eqref{eq-Lp-uniform-bound-electric-field} are independent of $t$. For Dirichlet boundary, however, such uniform bounds are unknown.
\end{remark}

\begin{proof}
The conservation laws \eqref{eq-conservation-laws} are standard. Note
\begin{equation*}
\int_{\partial\Omega}\vert G_{\#}(x,y)h_{\#}(y)\vert\,\mathrm{d}S_y\le C\|h_{\#}\|_{L^{\infty}(\partial\Omega)}\int_{\partial\Omega}\vert G_{\#}(x,y)\vert\,\mathrm{d}S_y\le C\|h_{\#}\|_{L^{\infty}(\partial\Omega)}.
\end{equation*}
Combining it with \eqref{eq-Green-upper-bound} and the representation of $\phi_{\#}$ using Green function, we have
\begin{equation*}
\phi_{\#}(t,x)\le A_{\Omega}\|f(t)\|_{L_1(\Omega\times\mathbb{R}^d)}+C\|h_{\#}\|_{L^{\infty}(\partial\Omega)}.
\end{equation*}

Hence it is no difficult to deduce that there exists a constant $C>0$ such that
\begin{equation*}
-C+\frac{1}{2}\iint_{\Omega\times\mathbb{R}^3}\vert v\vert^2f(t)\,\mathrm{d}x\,\mathrm{d}v+\frac{1}{2}\sum_{\alpha}\vert \eta_{\alpha}(t)\vert^2-\frac{1}{2}\sum_{\alpha}R_{\#}(\xi_{\alpha}(t))\le \mathcal{E}(t)=\mathcal{E}(0),
\end{equation*}
By Proposition~\ref{prn-Robin-estimate}, for $\#=\mathrm{N}$, there exist $C,\delta_2>0$ depending only on $\mathcal{E}(0)$, $\|\mathring{f}\|_{L^1}$, $M$, $A_{\Omega}$, $\operatorname{diam}\Omega$ such that
\begin{equation*}
\max_{\alpha}\sup_{t\in[0,T]}\vert \eta_{\alpha}(t)\vert\le C,\quad\min_{\alpha}\inf_{t\in[0,T]}\operatorname{d}_{\partial\Omega}(\xi_{\alpha}(t))\ge\delta_2.
\end{equation*}
Recall that
\begin{equation*}
G_{\mathrm{N}}(x,y)=G(x,y)+\bar{g}_{\mathrm{N}}(x,y).
\end{equation*}
Since $\bar{g}_{\mathrm{N}}$ is smooth in $\Omega\times\Omega$, by \eqref{eq-uniform-bound-point-charge-1}, we have
\begin{equation*}
\vert\bar{g}_{\mathrm{N}}(\xi_{\alpha}(t),\xi_{\beta}(t))\vert\le C
\end{equation*}
where $C>0$ depending only on $\delta_2$. Hence it is no difficult to deduce that
\begin{equation*}
-C-\frac{1}{2}\sum_{\alpha\ne\beta}G(\xi_{\alpha}(t),\xi_{\beta}(t))\le \mathcal{E}(t)=\mathcal{E}(0),
\end{equation*}
where $C$ depends only on $\mathcal{E}(0)$, $\|\mathring{f}\|_{1}$, $\|h_{\#}\|_{\infty}$, $M$, $A_{\Omega}$, $\operatorname{diam}\Omega$. Therefore, by the definition of $G$,  there exists a constant, still denoted as $\delta_2>0$ such that
\begin{equation*}
\min_{\alpha\ne\beta}\inf_{t\in[0,T]}\vert\xi_{\alpha}(t)-\xi_{\beta}(t)\vert\ge\delta_2.
\end{equation*}

\eqref{eq-Lp-uniform-bound-rho} can be obtained by the kinetic interpolation inequality. \eqref{eq-Lp-uniform-bound-electric-field} can be obtained by Lemma~\ref{lem-elliptic-estimate-electric-field} and \eqref{eq-Lp-uniform-bound-rho}.
\end{proof}

\subsection{Velocity lemmas}

The characteristics associated with \eqref{eq-Vlasov-short}, i.e., the trajectories in phase space of the plasma particles, are the solutions $(X(s,t,x,v),V(s,t,x,v))$ of the following differential equations:
\begin{equation}\label{eq-character-ODE}
\left\{\begin{split}
&\partial_sX=V,\quad\partial_sV=E_{\#}^{\rm pla}(s,X),\\
&(X(t,t,x,v),V(t,t,x,v))=(x,v)\in\operatorname{supp}f(t),
\end{split}\right.
\end{equation}
as long as $X(s,t,x,v)$ remains in $\Omega$. Assume at time $s=\bar{s}$, $X(\bar{s},t,x,v)\in\partial\Omega$, we define
\begin{equation}\label{eq-character-V-reflection}
\left\{\begin{split}
&V(\bar{s}+0,t,x,v)=\mathcal{R}_{X(\bar{s})}V(\bar{s}-0,t,x,v),\\
&\vert V(\bar{s},t,x,v)\vert=\vert V(\bar{s}\pm 0,t,x,v)\vert,
\end{split}\right.
\end{equation}
where $V(\bar{s}\pm0,t,x,v)=\lim_{\Delta\to0^+}V(\bar{s}\pm\Delta,t,x,v)$ and $X(\bar{s})=X(\bar{s},t,x,v)$ for short.
Then \eqref{eq-character-ODE}, \eqref{eq-character-V-reflection} together define the generalized characteristics along which $f$ is constant if $f$ satisfies \eqref{eq-Vlasov-short} and the reflection condition in \eqref{eq-Dirichlet-boundary}, \eqref{eq-Neumann-boundary}. 


Let $\#=\mathrm{N}$, we have the following so called velocity lemma.
\begin{lemma}\label{lem-velocity-lemma}
Let $\bar{T},L,\varepsilon>0$. Assume that $\rho\in C^{0;\mu}([0,\bar{T}]\times\bar{\Omega})$ for some $\mu\in(0,1)$, $h_{\mathrm{N}}\in C^2(\partial\Omega;\mathbb{R}_+)$, $\xi_{\alpha}\in C([0,\bar{T}];\Omega)$ are given. Let $X(s)=X(s,0,x,v),V(s)=V(s,0,x,v)$ be the solution of the equations \eqref{eq-character-ODE}, \eqref{eq-character-V-reflection} satisfying $X_{\perp}(s)\in[\partial\Omega+B_{\delta_0}(0)]\cap\Omega$, $\vert V(s)\vert\le L$, $\min_{\alpha}\{\vert X(s)-\xi_{\alpha}(s)\vert,\operatorname{d}_{\partial\Omega}(\xi_{\alpha}(s))\}\ge\varepsilon$ for all $s\in[0,\bar{T}]$. Then there exists a constant $C>1$ depending only on $\bar{T}$, $\varepsilon$, $L$, $\|h_{\mathrm{N}}\|_{C^2(\partial\Omega)}$, $\|\rho\|_{C^{0;\mu}([0,\bar{T}]\times\bar{\Omega})}$ and the $C^4$ regularity of $\partial\Omega$ such that
\begin{equation}\label{eq-velocity-lem}
C^{-1}(X_{\perp}(0)+V_{\perp}^2(0))\le (X_{\perp}(s)+V_{\perp}^2(s))\le C(X_{\perp}(0)+V_{\perp}^2(0)),\quad\forall s\in[0,\bar{T}].
\end{equation}

\end{lemma}

\begin{remark}\label{rem-velocity-lemma-Dirichlet}
Analogous conclusion holds also for $\#=\mathrm{D}$, see \cite[Lemma 2.4]{Wu24}. However, the assumption $\min_{\alpha}\{\operatorname{d}_{\partial\Omega}(\xi_{\alpha}(s))\}\ge\varepsilon$ can not be verified for all $\bar{T}>0$, which invalidates the argument in the next sections.
\end{remark}

\begin{proof}
In view of Lemma~\ref{lem-Vlasov-new-coordinates}, the equations describing the evolution of the characteristics associated with \eqref{eq-Vlasov-short} in a geometrical form are:
\begin{equation}\label{eq-character-ODE-new-coordinates}
\left\{\begin{split}
\partial_s\boldsymbol{\mu}_i&=\frac{\boldsymbol{w}_i}{1+\boldsymbol{\kappa}_i\boldsymbol{x}_{\perp}},\quad\partial_s\boldsymbol{w}_i=\bar{\boldsymbol{E}}_{\mathrm{N},i}^{\rm pla},\\
\partial_s\boldsymbol{x}_{\perp}&=\boldsymbol{v}_{\perp},\quad\partial_s\boldsymbol{v}_{\perp}=\bar{\boldsymbol{E}}_{\mathrm{N},\perp}^{\rm pla},
\end{split}\right.
\end{equation}
where 
\begin{align*}
\boldsymbol{\mu}_i&=\boldsymbol{\mu}_i(s,t,\mu_1,\mu_2,x_{\perp},w_1,w_2,v_{\perp}),\\
\boldsymbol{\kappa}_i&=\kappa_i(\boldsymbol{\mu}_1,\boldsymbol{\mu}_2),\\
\boldsymbol{\sigma}_i&=\sigma_i(s,\boldsymbol{\mu}_1,\boldsymbol{\mu}_2,\boldsymbol{x}_{\perp},\boldsymbol{w}_1,\boldsymbol{w}_2,\boldsymbol{v}_{\perp}),\\
\bar{\boldsymbol{E}}_{\mathrm{N},i}^{\rm pla}&=\bar{E}_{\mathrm{N},i}^{\rm pla}(s,\boldsymbol{\mu}_1,\boldsymbol{\mu}_2,\boldsymbol{x}_{\perp},\boldsymbol{w}_1,\boldsymbol{w}_2,\boldsymbol{v}_{\perp}),\\
\bar{\boldsymbol{E}}_{\mathrm{N},\perp}^{\rm pla}&=\bar{E}_{\mathrm{N},\perp}^{\rm pla}(s,\boldsymbol{\mu}_1,\boldsymbol{\mu}_2,\boldsymbol{x}_{\perp},\boldsymbol{w}_1,\boldsymbol{w}_2),
\end{align*} 
and $\boldsymbol{w}_i$, $\boldsymbol{x}_{\perp}$, $\boldsymbol{v}_{\perp}$ are defined similarly to $\boldsymbol{\mu}_i$.

We define the distance function toward the grazing set as
\begin{equation*}
\boldsymbol{\beta}=\beta(s,\boldsymbol{\mu}_1,\boldsymbol{\mu}_2,\boldsymbol{x}_{\perp},\boldsymbol{w}_1,\boldsymbol{w}_2,\boldsymbol{v}_{\perp})=\frac{\vert \boldsymbol{v}_{\perp}\vert^2}{2}-\bar{\boldsymbol{E}}_{\mathrm{N},\perp,0}^{\rm pla}\boldsymbol{x}_{\perp},
\end{equation*}
where 
\begin{equation*}
\bar{\boldsymbol{E}}_{\mathrm{N},\perp,0}^{\rm pla}=\bar{E}_{\mathrm{N},\perp}^{\rm pla}(s,\boldsymbol{\mu}_1,\boldsymbol{\mu}_2,0,\boldsymbol{w}_1,\boldsymbol{w}_2)=-\boldsymbol{h}_{\mathrm{N}}+\sum_{j=1,2}\boldsymbol{w}_j^2\boldsymbol{b}_j.
\end{equation*}
We have used the facts that $\partial_{n_x}\phi_{\mathrm{N}}(x)=h_{\mathrm{N}}$, $\partial_{n_x}G_{\mathrm{N}}(x,y)=0$ for $(x,y)\in\partial\Omega\times\Omega$. Since $\min_{x\in\partial\Omega}h_{\mathrm{N}}(x)>0$, $b_j\le 0$ due to the convexity of $\Omega$, there exists a constant $c>0$ depending only on $\|h_{\mathrm{N}}\|_{L^{\infty}(\partial\Omega)}$, $\sum_{j=1,2}\|b_j\|_{L^{\infty}(\partial\Omega)}$, $L$ such that
\begin{equation*}
c(\vert \boldsymbol{v}_{\perp}\vert^2+\boldsymbol{x}_{\perp})\le\boldsymbol{\beta}\le c^{-1}(\vert \boldsymbol{v}_{\perp}\vert^2+\boldsymbol{x}_{\perp}).
\end{equation*}

According to \eqref{eq-character-ODE-new-coordinates}, differentiating $\boldsymbol{\beta}$ with respect to $s$ we obtain:
\begin{align*}
\partial_s\boldsymbol{\beta}=&\boldsymbol{v}_{\perp}\left(\bar{\boldsymbol{E}}_{\mathrm{N},\perp}^{\rm pla}-\bar{\boldsymbol{E}}_{\mathrm{N},\perp,0}^{\rm pla}\right)-\boldsymbol{x}_{\perp}\sum_{i=1,2}\left(\partial_{\boldsymbol{\mu}_i}\bar{\boldsymbol{E}}_{\mathrm{N},\perp,0}^{\rm pla}\frac{\boldsymbol{w}_i}{1+\boldsymbol{\kappa}_i\boldsymbol{x}_{\perp}}+\partial_{\boldsymbol{w}_i}\bar{\boldsymbol{E}}_{\mathrm{N},\perp,0}^{\rm pla}\bar{\boldsymbol{E}}_{\mathrm{N},i}^{\rm pla}\right),
\end{align*}
where
\begin{equation*}
\partial_{\boldsymbol{\mu}_i}\bar{\boldsymbol{E}}_{\mathrm{N},\perp,0}^{\rm pla}=-\partial_{\boldsymbol{\mu}_i}\boldsymbol{h}_{\mathrm{N}}+\sum_{j=1,2}\boldsymbol{w}_j^2\partial_{\boldsymbol{\mu}_i}\boldsymbol{b}_j,\quad \partial_{\boldsymbol{w}_i}\bar{\boldsymbol{E}}_{\mathrm{N},\perp,0}^{\rm pla}=2\boldsymbol{w}_i\boldsymbol{b}_i.
\end{equation*}
Hence we have
\begin{equation*}
\vert\partial_s\boldsymbol{\beta}\vert\le C\vert\boldsymbol{v}_{\perp}\vert\sqrt{\boldsymbol{x}_{\perp}}+C\boldsymbol{x}_{\perp}\le C\boldsymbol{\beta},
\end{equation*}
where $C>0$ depending only on $L$, $\|h_{\mathrm{N}}\|_{C^1(\partial\Omega)}$, $\sum_{i=1,2}\|\kappa_i\|_{L^{\infty}(\partial\Omega)}+\|b_i\|_{C^1(\partial\Omega)}$, $\|\bar{E}_{\mathrm{N}}^{\rm pla}\|_{C^{0;1/2}([0,\bar{T}]\times\bar{\Omega})}$. Recall the definitions of $\bar{E}_{\mathrm{N},i}^{\rm pla}$, $E_{\mathrm{N},i}^{\rm pla}$, $\bar{E}_{\mathrm{N},\perp}^{\rm pla}$, $E_{\mathrm{N},\perp}^{\rm pla}$ in Lemma~\ref{lem-Vlasov-new-coordinates}, by the elliptic estimates, we have
\begin{equation*}
\|\bar{E}_{\mathrm{N}}^{\rm pla}\|_{C([0,\bar{T}];C^{1/2}(\bar{\Omega}))}\le C,
\end{equation*}
where $C>0$ depending only on $\varepsilon$, $\|h_{\mathrm{N}}\|_{C^2(\partial\Omega)}$, $\|\rho\|_{C^{0;\mu}([0,\bar{T}]\times\bar{\Omega})}$ and $C^4$ regularity of $\partial\Omega$. According to the Gr\"onwall inequality we obtain \eqref{eq-velocity-lem}.

\end{proof}
The following lemma proved by contradiction with the help of velocity lemmas, see \cite[Lemma 13]{HV10} for details of the proof.
\begin{lemma}\label{lem-min-dichotomy}
Let $(X,V)$ and $(Y,W)$ be two generalized characteristics on some interval $[t_1,t_2]$. Suppose that
\begin{equation*}
\vert Y_{\perp}(s_0)-X_{\perp}(s_0)\vert=\min_{s\in[t_1,t_2]}\vert Y_{\perp}(s)-X_{\perp}(s)\vert
\end{equation*}
with $s_0\in(t_1,t_2)$. Then either both $Y_{\perp}(s_0)>0$, $X_{\perp}(s_0)>0$ or both $Y_{\perp}(s_0)=X_{\perp}(s_0)=0$.
\end{lemma}

\subsection{Proof of Theorems~\ref{thm-local-wellposedness-absorption} and \ref{thm-local-wellposedness-reflection}}\label{subsec-local-well-posedness}

The proof is somewhat standard, which has been established exhaustively, e.g., in \cite{Guo94,HV09,HV10} for the boundary case and \cite{Gla96,Rei07} for the whole space case. Hence we only sketch the proof of existence for the reflection-Dirichlet boundary case.

We define iteratively a sequence of functions $f^n,\{\xi_{\alpha}^n,\eta_{\alpha}^n\}$ as follows. Thanks to Lemma~\ref{lem-velocity-lemma} and Remark~\ref{rem-velocity-lemma-Dirichlet}, such sequence is well defined (see \cite[Proposition 2]{HV10}). Let $(f^0,\xi_{\alpha}^0,\eta_{\alpha}^0)=(\mathring{f},\mathring{\xi}_{\alpha},\mathring{\eta}_{\alpha})$. For $n\ge 1$, let $\rho^{n-1}=\int f^{n-1}\,\mathrm{d}v$ and
\begin{align*}
&E_{\#}^{n-1}(t,x)=\int_{\Omega}\nabla_xG_{\#}(x,y)\rho^{n-1}(t,y)\,\mathrm{d}y-\int_{\partial\Omega} \nabla_xG_{\#}(x,y)h_{\#}(y)\,\mathrm{d}S_y,\\
&E_{\#}^{{\rm pla},n-1}(t,x)=E_{\#}^{n-1}(t,x)+\sum_{\alpha}\nabla_xG(x,\xi_{\alpha}^n(t)),\\
&E_{\#,\alpha}^{{\rm cha},n-1}(t)=E_{\#}^{n-1}(t,\xi_{\alpha}^{n}(t))+\sum_{\beta:\beta\ne\alpha}\nabla_xG_{\#}(\xi_{\alpha}^{n}(t),\xi_{\beta}^n(t))+\nabla_xH_{\#}(\xi_{\alpha}^n(t)),\\
&\partial_tf^{n}+v\cdot\nabla_xf^{n}+E_{\#}^{{\rm pla},n-1}\cdot\nabla_vf^{n}=0,\\
&\dot{\xi}_{\alpha}^{n}=\eta_{\alpha}^{n},\quad\dot{\eta}_{\alpha}^{n}=E_{\#,\alpha}^{{\rm cha},n-1},
\end{align*}
with initial-boundary value conditions
\begin{align*}
&(f^n,{\xi}_{\alpha}^{n},{\eta}_{\alpha}^{n})\vert_{t=0}=(\mathring{f},\mathring{{\xi}}_{\alpha},\mathring{{\eta}}_{\alpha}),\\
&f^n(t,x,v)=f^n(t,x,\mathcal{R}_xv),\quad x\in\partial\Omega.
\end{align*}
For $n\ge 1$, the characteristics $X^n(t)=X^n(t,0,x,v)$, $V^n(t)=V^n(t,0,x,v)$ are given by
\begin{equation*}
\left\{\begin{split}
&\partial_tX^n=V^n,\quad \partial_tV^n=E_{\#}^{{\rm pla},n-1}(t,X^n),\\
&(X^n(0),V^n(0))=(x,v)\in\operatorname{supp}\mathring{f}.
\end{split}\right.
\end{equation*}
It can be deduced that $\|f^{n}(t)\|_{p}=\|\mathring{f}\|_{p}$ for all $n\ge 1$, $1\le p\le\infty$.

For $n\ge 0$, we introduce a function which contains all the information we need:
\begin{align*}
e^n(t,x,v)=&\frac{\vert v\vert^2}{2}+\sum_{\alpha}\left[\frac{1}{\vert x-\xi_{\alpha}^n(t)\vert}+\frac{1}{2}\vert\eta_{\alpha}^n(t)\vert^2+\frac{1}{\operatorname{d}_{\partial\Omega}(\xi_{\alpha}(t))}\right]\\
&+\frac{1}{2}\sum_{a\ne\beta}\frac{1}{\vert\xi_{\alpha}^n(t)-\xi_{\beta}^n(t)\vert}+1
\end{align*}
and the associated quantity for $n\ge 0$
\begin{align*}
Q^n(t):=&\sup\{\sqrt{e^n}(s,x,v):(x,v)\in\operatorname{supp} f^n(s),\,0\le s\le t\},\\
\text{(for\;$n\ge 1$)}=&\sup\{\sqrt{e^n}(s,X^n(s),V^n(s)):(x,v)\in\operatorname{supp}\mathring{f},\,0\le s\le t\}.
\end{align*}
Notice for $n\ge 1$
\begin{equation*}
Q^0(t)=Q^n(0)=\sup\{\sqrt{e^n}(0,x,v):(x,v)\in\operatorname{supp}\mathring{f}\}:=\mathring{Q}.
\end{equation*}
Denote $\mathsf{e}^n(t)=e^n(t,X^n(t),V^n(t))$, differentiating it to obtain
\begin{align*}
\frac{\mathrm{d}}{\mathrm{d}t}\mathsf{e}^n(t)&=V^n(t)\cdot E_{\#}^{{\rm pla},n-1}(t,X^n(t))+\sum_{\alpha}\frac{(\eta_{\alpha}^n(t)-V^n(t))\cdot(X^n(t)-\xi_{\alpha}^n(t))}{\vert X^n(t)-\xi_{\alpha}^n(t)\vert^3}\\
&\quad+\sum_{\alpha}\eta_{\alpha}^n(t)\cdot E_{\#,\alpha}^{{\rm cha},n-1}(t)+\sum_{a\ne\beta}\frac{\eta_{\beta}^n(t)\cdot(\xi_{\alpha}^n(t)-\xi_{\beta}^n(t))}{\vert\xi_{\alpha}^n(t)-\xi_{\beta}^n(t)\vert^3}\\
&\quad-\sum_{\alpha}\frac{1}{[\operatorname{d}_{\partial\Omega}(\xi_{\alpha}(t))]^2}\nabla\operatorname{d}_{\partial\Omega}(\xi_{\alpha}(t))\cdot\eta_{\alpha}(t)\\
&\le C\sqrt{\mathsf{e}^n(t)}\|E_{\#}^{n-1}\|_{{\infty}}+C(\mathsf{e}^n(t))^{5/2}\\
&\le C\sqrt{\mathsf{e}^n(t)}\left(Q^{n-1}(t)\right)^{2}+C(\mathsf{e}^n(t))^{5/2},
\end{align*}
where we have used \eqref{eq-pre-estimates-electric-field} for $\rho^{n-1}(t)=\int_{\vert v\vert\le2Q^{n-1}(t)}f^{n-1}(t)\,\mathrm{d}v$. Integrating the inequality from $0$ to $t$ we have
\begin{equation*}
\sqrt{\mathsf{e}^n(t)}\le \sqrt{\mathsf{e}^n(0)}+C\int_0^t\left(Q^{n-1}(s)\right)^{2}+(\mathsf{e}^n(s))^{2}\,\mathrm{d}s.
\end{equation*}
Then
\begin{equation*}
Q^n(t)\le \mathring{Q}+\mathring{C}\int_0^t\left(Q^{n-1}(s)\right)^{2}+\left(Q^{n}(s)\right)^{4}\,\mathrm{d}s.
\end{equation*}
where $\mathring{C}$ depends on $\|\mathring{f}\|_{1},\|\mathring{f}\|_{\infty}$. We claim that there exist constants $\mathring{T},\mathring{K}$ depending on $\mathring{Q},\mathring{C}$ such that
\begin{equation}\label{eq-pf-thm-local-wellposedness-1}
Q^n(t)\le\mathring{K},\quad\forall n\ge 0,\,t\in[0,\mathring{T}].
\end{equation}
Indeed, $\mathring{T},\mathring{K}$ can be taken as
\begin{equation*}
\mathring{K}=2\mathring{Q}+2,\quad\mathring{T}=\frac{1}{48\mathring{C}(\mathring{K}^{2}+\mathring{Q}^3)},
\end{equation*}
then it is no difficulty to deduce the claim by induction.

With the aid of the uniform estimate \eqref{eq-pf-thm-local-wellposedness-1}, Lemma~\ref{lem-velocity-lemma} and Remark~\ref{rem-velocity-lemma-Dirichlet}, it is sufficient to obtain a solution $f,\{\xi_{\alpha},\eta_{\alpha}\}$ on the interval $[0,\mathring{T}]$ by adapting the argument in \cite[Section 5]{HV10}. Notice that \(\mathring{T}\) depends only on $\|\mathring{f}\|_{1},\|\mathring{f}\|_{\infty}$ and \(\mathring{Q}\), hence the continuation criterion \eqref{eq-continuation-point-charge-reflection}, \eqref{eq-continuation-plasma-reflection} follows.

\section{The global existence for absorption-Neumann boundary}\label{sec-global-absorption}

We define a pointwise energy function as
\begin{equation*}
h(s,x,v):=\frac{\vert v\vert^2}{2}+\sum_{\alpha}\frac{1}{4\pi\vert x-\xi_{\alpha}(s)\vert}+K_1,
\end{equation*}
where $K_1>1$ is a large constant depending only on $\|\mathring{f}\|_{1}$, $\|\mathring{f}\|_{\infty}$, $\mathcal{E}(0)$. Energy moment associated with $h$ is defined by
\begin{equation*}
H_k(t):=\sup_{0\le s\le t}\tilde{H}_k(s),\;\text{where}\;\tilde{H}_k(t):=\int_{\Omega\times\mathbb{R}^3} h^{\frac{k}{2}}f(t,x,v)\,\mathrm{d}x\,\mathrm{d}v.
\end{equation*}
The $k$-th order singular moment is defined  by
\begin{equation*}
L_k(t):=\sum_{\alpha}\int_0^t\int_{\Omega\times\mathbb{R}^3} \frac{h^{\frac{k}{2}}f(s,x,v)}{\vert x-\xi_{\alpha}(s)\vert^2}\,\mathrm{d}x\,\mathrm{d}v\,\mathrm{d}s.
\end{equation*}

To prove Theorem~\ref{thm-absorption-Neumann-case}, it is sufficient to verify the continuation criterion in Theorem~\ref{thm-local-wellposedness-absorption} for $\#=\mathrm{N}$. Note \eqref{eq-continuation-point-charge-absorption} can be obtained by \eqref{eq-uniform-bound-point-charge-1}, \eqref{eq-uniform-bound-point-charge-2}. The condition \eqref{eq-continuation-plasma-absorption} can be verified directly by a moment propagation result for $H_k(t)$, i.e.~Proposition~\ref{prn-Hk-unif-bound-LP} below. The aim of this section is to prove Proposition~\ref{prn-Hk-unif-bound-LP} by the extended Lions-Perthame's method. We split the proof into four steps.

First, we recall the interpolation inequalities.
\begin{lemma}\label{lem-interpo-h-LP}
Let $f,\{\xi_{\alpha},\eta_{\alpha}\}$ be a classical solution given in Theorem~\ref{thm-local-wellposedness-absorption}. For all $0\le k\le l$, $0\le a<b<c$, we have
\begin{equation*}
\left\|\int h^{k/2}f(t,\cdot,v)\,\mathrm{d}v\right\|_{\frac{l+3}{k+3}}\le C \tilde{H}_{l}(t)^{\frac{k+3}{l+3}},\quad \tilde{H}_b(t)\le \tilde{H}_a(t)^{\frac{c-b}{c-a}}\tilde{H}_c(t)^{\frac{b-a}{c-a}}.
\end{equation*}
In particular, $\tilde{H}_k(t)\le C'$ for all $0\le k\le 2$. The constant $C$ depends only on $k,l$, $\|\mathring{f}\|_{\infty}$, and $C'$ depends only on $\mathcal{E}(0)$, $\|\mathring{f}\|_1$.
\end{lemma}

\begin{lemma}\label{lem-Hk-by-Lk-E-LP}
Let $f,\{\xi_{\alpha},\eta_{\alpha}\}$ be a classical solution given in Theorem~\ref{thm-local-wellposedness-absorption}. Then for all $k>0$
\begin{equation*}
H_k(t)\le H_k(0)+CL_{k-2}(t)+CH_{k}(t)^{\frac{k+2}{k+3}}\int_0^t(\|E_{\mathrm{N}}(s)\|_{k+3}+1)\,\mathrm{d}s,
\end{equation*}
where $C$ depends only on $\|\mathring{f}\|_1$, $\|\mathring{f}\|_{\infty}$, $\mathcal{E}(0)$ and $k$.
\end{lemma}
\begin{proof}
By \eqref{eq-Vlasov-short} and the boundary condition \eqref{eq-absorption-boundary}, we have
\begin{align*}
\frac{\mathrm{d}}{\mathrm{d}s}\tilde{H}_k(s)=&-\int h^{\frac{k}{2}}f(s)v\cdot n_x\,\mathrm{d}S_x\,\mathrm{d}v+\frac{k}{8\pi}\int h^{\frac{k-2}{2}}f(s)\sum_{\alpha}\eta_{\alpha}(s)\cdot\frac{x-\xi_{\alpha}(s)}{\vert x-\xi_{\alpha}(s)\vert^3}\,\mathrm{d}x\,\mathrm{d}v\\
&+\frac{k}{2}\int h^{\frac{k-2}{2}}f(s)v\cdot\left(E_{\mathrm{N}}^{\rm pla}-\sum_{\alpha}\frac{x-\xi_{\alpha}(s)}{4\pi\vert x-\xi_{\alpha}(s)\vert^3}\right)\,\mathrm{d}x\,\mathrm{d}v\\
\le&-\int_{v\cdot n_x\le 0} h^{\frac{k}{2}}g(s)v\cdot n_x\,\mathrm{d}S_x\,\mathrm{d}v+C\sum_{\alpha}\int \frac{h^{\frac{k-2}{2}}f(s)}{\vert x-\xi_{\alpha}(s)\vert^2}\,\mathrm{d}x\,\mathrm{d}v\\
&+C\int h^{\frac{k-1}{2}}f(s)\left(\vert E_{\mathrm{N}}\vert+1\right)\,\mathrm{d}x\,\mathrm{d}v\\
\le& C+C\sum_{\alpha}\int \frac{h^{\frac{k-2}{2}}f(s)}{\vert x-\xi_{\alpha}(s)\vert^2}\,\mathrm{d}x\,\mathrm{d}v+C\int h^{\frac{k-1}{2}}f(s)\left(\vert E_{\mathrm{N}}\vert+1\right)\,\mathrm{d}x\,\mathrm{d}v.
\end{align*}
By H\"older's inequality and Lemma~\ref{lem-interpo-h-LP}, we have
\begin{align*}
&\int h^{\frac{k-1}{2}}f(s)\left(\vert E_{\mathrm{N}}\vert+1\right)\,\mathrm{d}x\,\mathrm{d}v\nonumber\\
&\le\left\|\int h^{\frac{k-1}{2}}f(s,\cdot,v)\,\mathrm{d}v\right\|_{\frac{k+3}{k+2}}\|E_{\mathrm{N}}(s)+1\|_{k+3}\le C\tilde{H}_{k}(s)^{\frac{k+2}{k+3}}(\|E_{\mathrm{N}}(s)\|_{k+3}+1).
\end{align*}
Inserting it into the previous inequality, and integrating in time from $0$ to $t$, we have
\begin{equation*}
H_k(t)\le H_k(0)+CL_{k-2}(t)+CH_{k}(t)^{\frac{k+2}{k+3}}\int_0^t(\|E_{\mathrm{N}}(s)\|_{k+3}+1)\,\mathrm{d}s.
\end{equation*}
\end{proof}

\subsection{The first step: The estimates on the singular moment}

\begin{lemma}\label{lem-Lk-by-Hk-E-LP}
Let $f,\{\xi_{\alpha},\eta_{\alpha}\}$ be a classical solution given in Theorem~\ref{thm-local-wellposedness-absorption}. Then $\forall l>k>0$
\begin{equation*}
L_k(t)\le C\left(H_{k+1}(t)(t+1)+H_{k}(t)L_0(t)+H_{l}(t)^{\frac{k+3}{l+3}}\int_0^t\|E_{\mathrm{N}}(s)\|_{\frac{l+3}{l-k}}\,\mathrm{d}s+\int_0^t\tilde{H}_{k+2}(s)\,\mathrm{d}s\right).
\end{equation*}
where $C$ depends only on $\|\mathring{f}\|_1$, $\|\mathring{f}\|_{\infty}$, $\mathcal{E}(0)$ and $l,k$.
\end{lemma}
\begin{proof}
Recall that $\delta_2$ is given in Proposition~\ref{prn-conservation-laws}. We denote 
\begin{equation*}
P_1(s)=\chi_{\frac{\delta_2}{8}}(\vert x-\xi_{\alpha}(s)\vert).
\end{equation*}
We split the integral domains in $L_k(t)$ into two parts to obtain
\begin{align}
&\int_0^t\int \frac{h^{\frac{k}{2}}f(s)}{\vert x-\xi_{\alpha}(s)\vert^2}\,\mathrm{d}x\,\mathrm{d}v\,\mathrm{d}s\nonumber\\
&\le\int_0^t\int \frac{P_1h^{\frac{k}{2}}f(s)}{\vert x-\xi_{\alpha}(s)\vert^2}\,\mathrm{d}x\,\mathrm{d}v\,\mathrm{d}s+\int_0^t\int \frac{(1-P_1)h^{\frac{k}{2}}f(s)}{\vert x-\xi_{\alpha}(s)\vert^2}\,\mathrm{d}x\,\mathrm{d}v\,\mathrm{d}s\nonumber\\
&\le\int_0^t\int \frac{P_1h^{\frac{k}{2}}f(s)}{\vert x-\xi_{\alpha}(s)\vert^2}\,\mathrm{d}x\,\mathrm{d}v\,\mathrm{d}s+64\delta_2^{-2}H_k(t)t.\label{eq-pf-lem-Lk-by-Hk-E-LP-1}
\end{align}

To estimate the first term in the right hand side of \eqref{eq-pf-lem-Lk-by-Hk-E-LP-1}, we introduce
\begin{equation*}
P_2(s)=\frac{x-\xi_{\alpha}(s)}{\vert x-\xi_{\alpha}(s)\vert}\cdot(v-\eta_{\alpha}(s)).
\end{equation*}
Note
\begin{align*}
&\frac{\mathrm{d}}{\mathrm{d}s}\int P_1P_2h^{\frac{k}{2}}f(s)\,\mathrm{d}x\,\mathrm{d}v\\
&=\int\partial_sP_1P_2h^{\frac{k}{2}}f(s)\,\mathrm{d}x\,\mathrm{d}v+\int P_1\partial_sP_2h^{\frac{k}{2}}f(s)\,\mathrm{d}x\,\mathrm{d}v+\int P_1P_2\partial_sh^{\frac{k}{2}}f(s)\,\mathrm{d}x\,\mathrm{d}v\\
&\quad+\int P_1P_2h^{\frac{k}{2}}\partial_sf(s)\,\mathrm{d}x\,\mathrm{d}v\\
&=\int\partial_sP_1P_2h^{\frac{k}{2}}f\,\mathrm{d}x\,\mathrm{d}v-\int P_1P_2h^{\frac{k}{2}}v\cdot n_xf\,\mathrm{d}S_x\,\mathrm{d}v\\
&\quad+\int P_1\partial_sP_2h^{\frac{k}{2}}f+ P_1P_2\partial_sh^{\frac{k}{2}}f+\nabla_x(P_1P_2h^{\frac{k}{2}})\cdot v f+\nabla_v(P_1P_2h^{\frac{k}{2}})\cdot E_{\mathrm{N}}^{\rm pla}f\,\mathrm{d}x\,\mathrm{d}v\\
&=:I_{\alpha1}(s)+I_{\alpha2}(s)+I_{\alpha3}(s).
\end{align*}
For $I_{\alpha1}$,
\begin{equation*}
\vert I_{\alpha1}(s)\vert=\left\vert\int 8\delta_2^{-1}\chi_{\frac{\delta_2}{8}}'(\vert x-\xi_{\alpha}(s)\vert)(P_2)^2h^{\frac{k}{2}}f(s)\,\mathrm{d}x\,\mathrm{d}v\right\vert\le C\tilde{H}_{k+2}(s).
\end{equation*}
For $I_{\alpha2}$, since $P_1=0$ if $\vert x-\xi_{\alpha}(s)\vert\ge\frac{\delta_2}{4}$, hence
\begin{equation*}
I_{\alpha2}\equiv 0.
\end{equation*}
For $I_{\alpha3}(s)$, first, it can be computed directly that
\begin{align*}
P_1\partial_sP_2h^{\frac{k}{2}}&=P_1h^{\frac{k}{2}}\left[\frac{\eta_{\alpha}\cdot(\eta_{\alpha}-v)+(\xi_{\alpha}-x)\cdot E_{\mathrm{N},\alpha}^{\rm cha}(s,\xi_{\alpha})}{\vert \xi_{\alpha}-x\vert}+P_2\frac{(x-\xi_{\alpha})\cdot\eta_{\alpha}}{\vert \xi_{\alpha}-x\vert^2}\right],\\
P_1P_2\partial_sh^{\frac{k}{2}}&=P_1P_2\left[\frac{k}{2}h^{\frac{k-2}{2}}\sum_{\beta}\frac{x-\xi_{\beta}}{4\pi\vert x-\xi_{\beta}\vert^3}\cdot\eta_{\beta}\right],\\
\nabla_x(P_1P_2h^{\frac{k}{2}})\cdot v&=P_2h^{\frac{k}{2}}\left[8\delta_2^{-1}\chi_{\frac{\delta_2}{8}}'(\vert x-\xi_{\alpha}\vert)\frac{x-\xi_{\alpha}}{\vert x-\xi_{\alpha}\vert}\right]\cdot v\\
&\quad+P_1h^{\frac{k}{2}}\left[\frac{v-\eta_{\alpha}}{\vert x-\xi_{\alpha}\vert}-\frac{x-\xi_{\alpha}}{\vert x-\xi_{\alpha}\vert^3}(x-\xi_{\alpha})\cdot(v-\eta_{\alpha})\right]\cdot v\\
&\quad +P_1P_2\left[-\frac{k}{2}h^{\frac{k-2}{2}}\sum_{\beta}\frac{x-\xi_{\beta}}{4\pi\vert x-\xi_{\beta}\vert^3}\right]\cdot v,\\
\nabla_v(P_1P_2h^{\frac{k}{2}})\cdot E_{\mathrm{N}}^{\rm pla}&=P_1h^{\frac{k-2}{2}}\left[h\frac{x-\xi_{\alpha}}{\vert x-\xi_{\alpha}\vert}+\frac{k}{2}P_2v\right]\cdot\left[E_{\mathrm{N}}(s,x)+\sum_{\beta}\nabla_x\bar{g}_{\mathrm{N}}(x,\xi_{\beta})\right]\\
&\quad+P_1h^{\frac{k-2}{2}}\left[h\frac{x-\xi_{\alpha}}{\vert x-\xi_{\alpha}\vert}+\frac{k}{2}P_2v\right]\cdot \sum_{\beta}\frac{x-\xi_{\beta}}{4\pi\vert x-\xi_{\beta}\vert^3}.
\end{align*}
Then
\begin{align*}
&P_1\partial_sP_2h^{\frac{k}{2}}+P_1P_2\partial_sh^{\frac{k}{2}}+\nabla_x(P_1P_2h^{\frac{k}{2}})\cdot v+\nabla_v(P_1P_2h^{\frac{k}{2}})\cdot E_{\mathrm{N}}^{\rm pla}\\
&=P_1h^{\frac{k}{2}}\left[\frac{\vert \eta_{\alpha}-v\vert^2-(P_2)^2+(\xi_{\alpha}-x)\cdot E_{\mathrm{N},\alpha}^{\rm cha}(s,\xi_{\alpha})}{\vert \xi_{\alpha}-x\vert}\right]\\
&\quad+P_2h^{\frac{k}{2}}\left[8\delta_2^{-1}\chi_{\frac{\delta_2}{8}}'(\vert x-\xi_{\alpha}\vert)\frac{x-\xi_{\alpha}}{\vert x-\xi_{\alpha}\vert}\right]\cdot v\\
&\quad+P_1h^{\frac{k-2}{2}}\left[h\frac{x-\xi_{\alpha}}{\vert x-\xi_{\alpha}\vert}+\frac{k}{2}P_2v\right]\cdot\left[E_{\mathrm{N}}(s,x)+\sum_{\beta}\nabla_x\bar{g}_{\mathrm{N}}(x,\xi_{\beta})\right]\\
&\quad+\frac{k}{2}P_1P_2h^{\frac{k-2}{2}}\sum_{\beta}\frac{(x-\xi_{\beta})\cdot\eta_{\beta}}{4\pi\vert x-\xi_{\beta}\vert^3}+P_1h^{\frac{k}{2}}\frac{x-\xi_{\alpha}}{\vert x-\xi_{\alpha}\vert}\cdot \sum_{\beta:\beta\ne\alpha}\frac{x-\xi_{\beta}}{4\pi\vert x-\xi_{\beta}\vert^3}+\frac{P_1h^{\frac{k}{2}}}{4\pi\vert x-\xi_{\alpha}\vert^2}\\
&\ge -Ch^{\frac{k}{2}}(\vert E_{\mathrm{N}}(s,\xi_{\alpha})\vert+1)-Ch^{\frac{k+1}{2}}-Ch^{\frac{k}{2}}(\vert E_{\mathrm{N}}(s,x)\vert+1)\\
&\quad-C\sum_{\beta}\frac{h^{\frac{k-2}{2}}}{\vert x-\xi_{\beta}\vert^2}-Ch^{\frac{k}{2}}+\frac{P_1h^{\frac{k}{2}}}{4\pi\vert x-\xi_{\alpha}\vert^2}\\
&\ge -Ch^{\frac{k}{2}}(\vert E_{\mathrm{N}}(s,\xi_{\alpha})\vert+\vert E_{\mathrm{N}}(s,x)\vert)-Ch^{\frac{k+1}{2}}-C\sum_{\beta}\frac{h^{\frac{k-2}{2}}}{\vert x-\xi_{\beta}\vert^2}+\frac{P_1h^{\frac{k}{2}}}{4\pi\vert x-\xi_{\alpha}\vert^2}.
\end{align*}

Hence
\begin{align*}
&\frac{\mathrm{d}}{\mathrm{d}s}\int P_1P_2h^{\frac{k}{2}}f(s)\,\mathrm{d}x\,\mathrm{d}v=I_{\alpha1}(s)+I_{\alpha2}(s)+I_{\alpha3}(s)\\
&\ge -C\tilde{H}_{k+2}(s)-C\tilde{H}_{l}(s)^{\frac{k+3}{l+3}}\|E_{\mathrm{N}}(s)\|_{\frac{l+3}{l-k}}-C\tilde{H}_{k}(s)\vert E_{\mathrm{N}}(s,\xi_{\alpha})\vert\\
&\quad-C\int\sum_{\beta}\frac{h^{\frac{k-2}{2}}}{\vert x-\xi_{\beta}\vert^2}\,\mathrm{d}x\,\mathrm{d}v+\int \frac{P_1h^{\frac{k}{2}}f(s)}{\vert x-\xi_{\alpha}(s)\vert^2}\,\mathrm{d}x\,\mathrm{d}v.
\end{align*}
Integrating in time from $0$ to $t$, we get
\begin{align*}
&\int_0^t\int\frac{P_1h^{\frac{k}{2}}f(s)}{\vert x-\xi_{\alpha}(s)\vert^2}\,\mathrm{d}x\,\mathrm{d}v\,\mathrm{d}s\\
&\le\int P_1P_2h^{\frac{k}{2}}f(t)\,\mathrm{d}x\,\mathrm{d}v-\int P_1P_2h^{\frac{k}{2}}f(0)\,\mathrm{d}x\,\mathrm{d}v+C\int_0^t\tilde{H}_{k+2}(s)\,\mathrm{d}s\\
&\quad+CH_{l}(t)^{\frac{k+3}{l+3}}\int_0^t\|E_{\mathrm{N}}(s)\|_{\frac{l+3}{l-k}}\,\mathrm{d}s+CH_{k}(t)L_0(t)+L_{k-2}(t).
\end{align*}
Inserting it into \eqref{eq-pf-lem-Lk-by-Hk-E-LP-1}, we have
\begin{align*}
L_k(t)\le CH_{k+1}(t)(t+1)+C\int_0^t\tilde{H}_{k+2}(s)\,\mathrm{d}s&+CH_{l}(t)^{\frac{k+3}{l+3}}\int_0^t\|E_{\mathrm{N}}(s)\|_{\frac{l+3}{l-k}}\,\mathrm{d}s\\
&+CH_{k}(t)L_0(t)+L_{k-2}(t).
\end{align*}
Note
\begin{equation*}
CL_{k-1}(t)\le\frac{1}{2}L_{k}(t)
\end{equation*}
if we take $K_1$ large enough.
\end{proof}

An consequence of Lemma~\ref{lem-Lk-by-Hk-E-LP} is the following estimate on $L_0(t)$.
\begin{lemma}\label{lem-L0-by-Hk}
Let $f,\{\xi_{\alpha},\eta_{\alpha}\}$ be a classical solution given in Theorem~\ref{thm-local-wellposedness-absorption}. Then $\forall t\in[0,\mathring{T})$, $a\in(0,\frac{1}{8})$, $\tilde{k}>2+a$, there holds
\begin{equation*}
L_0(t)\le C'\left(t+1+H_{\tilde{k}}(t)^{\frac{a}{\tilde{k}-2}}t\right),
\end{equation*}
where $C'$ depends only on $a,\tilde{k}$, $\|\mathring{f}\|_1$, $\|\mathring{f}\|_{\infty}$, $\|h_{\mathrm{N}}\|_{\infty}$ and $\mathcal{E}(0)$.
\end{lemma}
\begin{proof}
Let $\epsilon,R>0$ be determined later, we split $\Omega\times\mathbb{R}^3$ into three parts $A_1,A_2,A_3$ as
\begin{align*}
A_1:=&\{(x,v):\vert x-\xi_{\alpha}(s)\vert >\epsilon\},\\
A_2:=&\{(x,v):h(s,x,v)>R\},\\
A_3:=&\{(x,v):\vert x-\xi_{\alpha}(s)\vert \le\epsilon,\,h(s,x,v)\le R\}.
\end{align*}
Then 
\begin{align}
&\int_0^t\iint \frac{f(s,x,v)}{\vert x-\xi_{\alpha}(s)\vert^2}\,\mathrm{d}x\,\mathrm{d}v\,\mathrm{d}s\nonumber\\
&\le\int_0^t\iint_{A_1}+\iint_{A_2}+\iint_{A_3} \frac{f(s,x,v)}{\vert x-\xi_{\alpha}(s)\vert^2}\,\mathrm{d}x\,\mathrm{d}v\,\mathrm{d}s\nonumber\\
&\le C\epsilon^{-2}(\|\mathring{f}\|_1+1)t+R^{-\frac{a}{2}}\int_0^t\iint \frac{h^{\frac{a}{2}}f(s,x,v)}{\vert x-\xi_{\alpha}(s)\vert^2}\,\mathrm{d}x\,\mathrm{d}v\,\mathrm{d}s\nonumber\\
&\quad+\int_0^t\iint_{A_3} \frac{f(s,x,v)}{\vert x-\xi_{\alpha}(s)\vert^2}\,\mathrm{d}x\,\mathrm{d}v\,\mathrm{d}s.\label{eq-lem-L0-by-Hk-1}
\end{align}
Since $\vert v\vert \le 2\sqrt{h}(s,x,v)$, we have
\begin{align*}
&\int_0^t\iint_{A_3} \frac{f(s,x,v)}{\vert x-\xi_{\alpha}(s)\vert^2}\,\mathrm{d}x\,\mathrm{d}v\,\mathrm{d}s\\
&\le \int_0^t\iint_{\vert x-\xi_{\alpha}(s)\vert \le\epsilon,\vert v\vert \le 2\sqrt{R}} \frac{f(s,x,v)}{\vert x-\xi_{\alpha}(s)\vert^2}\,\mathrm{d}x\,\mathrm{d}v\,\mathrm{d}s\\
&\le C(\|\mathring{f}\|_{\infty}+1)R^{3/2}\epsilon t.
\end{align*}
Inserting it into \eqref{eq-lem-L0-by-Hk-1}, we have
\begin{equation}\label{eq-lem-L0-by-Hk-2}
L_0(t)=\sum_{\alpha=1}^N\int_0^t\iint \frac{f(s,x,v)}{\vert x-\xi_{\alpha}(s)\vert^2}\,\mathrm{d}x\,\mathrm{d}v\,\mathrm{d}s\le C\epsilon^{-2}t+R^{-a/2}L_{a}(t)+CR^{3/2}\epsilon t.
\end{equation}

Take $k=a$, $l=2$ in Lemma~\ref{lem-Lk-by-Hk-E-LP}, we have
\begin{align}L_a(t)&\le C\Big[(t+1)H_{1+a}(t)+H_{2}(t)^{\frac{a+3}{5}}\int_0^t\Big(\|E_{\mathrm{N}}(s)\|_{\frac{5}{2-a}}+1\Big)\,\mathrm{d}s\nonumber\\
&\qquad\qquad +\int_0^t H_{2+a}(s)\,\mathrm{d}s+H_{a}(t)L_0(t)\Big]\nonumber\\
&\le C\Big[(t+1)H_{1+a}(t)+H_{2}(t)^{\frac{a+3}{5}}\int_0^t\Big(\|E_{\mathrm{N}}(s)\|_{\frac{5}{2-a}}+1\Big)\,\mathrm{d}s\nonumber\\
&\qquad\qquad +H_{2}(t)^{\frac{\tilde{k}-2-a}{\tilde{k}-2}}H_{\tilde{k}}(t)^{\frac{a}{\tilde{k}-2}}t\Big]+CH_{a}(t)L_0(t)\nonumber\\
&\quad=:I(t)+J(t)L_0(t).\label{eq-lem-L0-by-Hk-3}
\end{align}

By Proposition~\ref{prn-conservation-laws} and Lemma~\ref{lem-interpo-h-LP}, there exist constants $C_1,C_2>0$ depending only on $a,\tilde{k}$, $\|\mathring{f}\|_1$, $\|\mathring{f}\|_{\infty}$, $\|h_{\mathrm{N}}\|_{\infty}$ and $\mathcal{E}(0)$, such that
\begin{equation}\label{eq-lem-L0-by-Hk-4}
I(t)\le C_1\left(t+1+H_{\tilde{k}}(t)^{\frac{a}{\tilde{k}-2}}t\right),\quad J(t)\le C_2.
\end{equation}
Now take $\epsilon=1, R=\max\{1,(2C_2)^{16}\}$ and insert \eqref{eq-lem-L0-by-Hk-3}, \eqref{eq-lem-L0-by-Hk-4} into \eqref{eq-lem-L0-by-Hk-2}, we have
\begin{equation*}
L_0(t)\le Ct+R^{-1/16}C_1\left(t+1+H_{\tilde{k}}(t)^{\frac{a}{\tilde{k}-2}}t\right)+\frac{1}{2}L_0(t)+CR^{3/2}t,
\end{equation*}
which implies immediately the lemma with $C'=2(C+C_1+CR^{3/2})$.

\end{proof}

Now we are able to establish the following proposition, which brings us back to the standard Lions-Perthame's argument.
\begin{proposition}\label{prn-Hk-by-Hk-E-LP}
Let $f,\{\xi_{\alpha},\eta_{\alpha}\}$ be a classical solution given in Theorem~\ref{thm-local-wellposedness-absorption}. Then for all $k\ge 0$, there holds
\begin{equation*}
H_k(t)\le H_k(0)+CH_{k}(t)^{\frac{k+2}{k+3}}\left(\int_0^t\|E_{\mathrm{N}}(s)\|_{k+3}\,\mathrm{d}s+t+1\right)+C\int_0^t H_{k}(s)\,\mathrm{d}s.
\end{equation*}
where $C$ depends only on $k$, $\|\mathring{f}\|_1$, $\|\mathring{f}\|_{\infty}$, $\|h_{\mathrm{N}}\|_{\infty}$ and $\mathcal{E}(0)$.
\end{proposition}
\begin{proof}
Combining Lemmas~\ref{lem-Lk-by-Hk-E-LP}-\ref{lem-L0-by-Hk}, we have
\begin{align*}
L_{k-2}(t)&\le C\left((t+1)H_{k-1}(t)+H_{k-1}(t)^{\frac{k+1}{k+2}}\int_0^t\left(\|E_{\mathrm{N}}(s)\|_{k+2}+1\right)\,\mathrm{d}s+\int_0^t H_{k}(s)\,\mathrm{d}s\right)\\
&\quad+CH_{k-2}(t)\left(t+1+H_{k}(t)^{\frac{a}{k-2}}t\right),
\end{align*}
Inserting it into Lemma~\ref{lem-Hk-by-Lk-E-LP}, we have
\begin{align*}
H_k(t)&\le H_k(0)+C\left(H_{k}(t)^{\frac{k+2}{k+3}}\int_0^t\left(\|E_{\mathrm{N}}(s)\|_{k+3}+1\right)\,\mathrm{d}s\right)+CH_{k-2}(t)\left(t+1+H_{k}(t)^{\frac{a}{k-2}}t\right)\\
&\quad+ C\left((t+1)H_{k-1}(t)+H_{k-1}(t)^{\frac{k+1}{k+2}}\int_0^t\left(\|E_{\mathrm{N}}(s)\|_{k+2}+1\right)\,\mathrm{d}s+\int_0^t H_{k}(s)\,\mathrm{d}s\right),
\end{align*}
By Lemma~\ref{lem-interpo-h-LP} and Proposition~\ref{prn-conservation-laws}, we have 
\begin{gather*}
H_i(t)\le C\;\text{for all}\;i\le 2,\\
H_i(t)\le CH_k(t)^{\frac{i-2}{k-2}}\le CH_k(t)^{\frac{k+2}{k+3}}\;\text{for all}\;2\le i\le k-1,
\end{gather*}
and by \eqref{eq-Lp-uniform-bound-electric-field}, for all $0\le i\le k$
\begin{equation*}
\|E_{\mathrm{N}}(s)\|_{i+3}\le \|E_{\mathrm{N}}(s)\|_3^{\frac{3(k-i)}{i+3}}\|E_{\mathrm{N}}(s)\|_{k+3}^{\frac{ik+3i}{ik+3k}}\le C(\|E_{\mathrm{N}}(s)\|_{k+3}+1).
\end{equation*}
Hence 
\begin{align*}
H_k(t)\le& H_k(0)+CH_{k}(t)^{\frac{k+2}{k+3}}\left(\int_0^t\|E_{\mathrm{N}}(s)\|_{k+3}\,\mathrm{d}s+t+1\right)+C\int_0^t H_{k}(s)\,\mathrm{d}s\\
&+CH_{k}(t)^{\frac{k-4+a}{k-2}}t.
\end{align*}
Now take $a=\frac{1}{16}$, we get the conclusion.
\end{proof}

\subsection{The second step: Representation of \(\rho\)}

Now we represent the macrocharge density $\rho$ along the back-time straight-line characteristics
\begin{equation*}
X^{\rm sl}(s,t,x,v)=x+(s-t)v, \quad V^{\rm sl}(s,t,x,v)\equiv v. 
\end{equation*}
We denote by $B(t, x, v) = (t^{\rm sl}, x^{\rm sl}, v)$ the possible boundary point when the trajectory hits $\partial \Omega$. We first note that for $v \neq 0$, there exists a unique $x^{\rm sl} = x^{\rm sl}(x, v) \in \partial \Omega$ along the back-time straight-line characteristics from $(t, x, v)$ since $\Omega$ is convex. Let $t^{\rm sl}$ be the time when the characteristics from $(t, x, v)$ hits the boundary. Then $x^{\rm sl} = x +  (t^{\rm sl} - t)v$ and $t^{\rm sl} - t = [(x^{\rm sl} - x) \cdot v]/|v|^2$. We define
\[
a(x, v) = -[(x^{\rm sl} - x) \cdot v]/|v|^2 =t-t^{\rm sl},
\]
to see that the function $a(x, v)$ is locally differentiable as follows: Note
\[
0 = \operatorname{d}_{\partial\Omega}(x^{\rm sl}) = \operatorname{d}_{\partial\Omega}\big(x + v (t^{\rm sl} - t)\big).
\]
By \cite[Lemma 14.16]{GT01}, \(\operatorname{d}_{\partial\Omega}\) is \(C^4\) near the boundary and \(\nabla\operatorname{d}_{\partial\Omega}(x)=-n_x\) if \(x\in\partial\Omega\). Set $s = t^{\rm sl} - t$ to get that $0 = \operatorname{d}_{\partial\Omega}(x + s v) = \operatorname{d}_{\partial\Omega}(s; x, v)$ and $\partial \operatorname{d}_{\partial\Omega}/ \partial s = \nabla \operatorname{d}_{\partial\Omega}(x^{\rm sl}) \cdot v = n_{x^{\rm sl}} \cdot v < 0$ since $\Omega$ is convex. By the implicit function theory, $s = t^{\rm sl} - t = -a(x, v)$ is a locally differentiable function of $x$ and $v$. We collect three preliminary lemmas well-established in \cite{Guo94,Hwa04}.
\begin{lemma}\label{lem-back-time-a-equality}
For \(v \neq 0\), \(x \in \Omega\), it holds
\[ \frac{1}{2}\nabla_x a^2 + \nabla_v a = 0.\]
\end{lemma}
Based on Lemma~\ref{lem-back-time-a-equality}, we obtain the representation formula for the macrocharge density.
\begin{lemma}\label{lem-represen-rho-back-time-line-character}
Let
\[
B(t, x, v) = (t - a(x, v),\, x - a(x, v) v,\, v) \in [0,\mathring{T})\times\Gamma.
\]
Then
\begin{align*}
\rho (t, x) =\ & \int_{a(x,v) \geq t} f_0(x - t v, v)\,\mathrm{d}v 
+ \int_{a(x,v) \leq t} g \circ B(t, x, v)\,\mathrm{d}v \\
& - \operatorname{div}_x \int_{a(x,v) \geq t} \int_0^t (t-s)\, (E_{\mathrm{N}}^{\rm pla}  f)(s, x - (t-s) v, v)\, \,\mathrm{d}s\,\mathrm{d}v \\
& - \operatorname{div}_x \int_{a(x,v) \leq t} \int_{t-a(x,v)}^t (t-s)\, (E_{\mathrm{N}}^{\rm pla} f)(s, x - (t-s) v, v)\, \,\mathrm{d}s\,\mathrm{d}v.
\end{align*}
\end{lemma}
Based on Lemma~\ref{lem-represen-rho-back-time-line-character}, we obtain the preliminary estimate on the electric field.
\begin{lemma}\label{lem-prel-estim-electric-field-by-repres-rho}
We have
\begin{equation*}
\left\| E_{\mathrm{N}}(t) \right\|_{p} 
\leq 
C+\sum_{i=1}^2\left\| 
\int_{\mathbb{R}^3} \int_0^t (t-s) \tilde{f}_i\left| E_{\mathrm{N}}^{\rm pla}(s, x - (t-s)v) \right| \,\mathrm{d}s\,\mathrm{d}v 
\right\|_{p},
\end{equation*}
where 
\begin{align*}
\tilde{f}_1&=\mathds{1}_{\{a \geq t\}}(v) f(s, x - (t-s)v, v),\\
\tilde{f}_2&=\mathds{1}_{\{a \leq t\}}(v) \mathds{1}_{(t-a, t)}(s)f(s, x - (t-s)v, v),
\end{align*}
and $1 \leq p < \infty$, $C$ is a constant depending only on the data $\mathring{f}$ and $g^{\rm pla}$.
\end{lemma}


\subsection{The Third Step: The Estimate on \(E_{\mathrm{N}}\) by \(H_k\)}

\begin{lemma}\label{lem-EN-by-Hk-LP}
Let $f,\{\xi_{\alpha},\eta_{\alpha}\}$ be a classical solution given in Theorem~\ref{thm-local-wellposedness-absorption}. Assume $k>3$, then there exist $r_1>\frac{3}{2}$, $\varepsilon>0$ such that for all $0<t_0<t<\mathring{T}$,
\begin{align*}
\left\| E_{\mathrm{N}}(t) \right\|_{k+3} 
\leq 
&C+C \ln \frac{t}{t-t_0}  H_k(t)^{\frac{1}{k+3}}+C(t-t_0)^{2-\frac{3}{r_1}} e^{Ct}H_{k}(t)^{\frac{3(k_1+3)}{(k+3)^2}}\nonumber\\
&+C(t-t_0)^{\frac{\varepsilon}{2}}e^{Ct}H_{k}(t)^{\frac{3(k_2+3)}{(k+3)^2}},
\end{align*}
where $C$ depends on $k,r_1,\varepsilon$, $\|f_0\|_1$, $\|f_0\|_{\infty}$, $\|h_{\mathrm{N}}\|_{\infty}$, $\mathcal{E}(0)$, $H_k(0)$, $H_{k_1}(0)$, $H_{k_2}(0)$, and $k_1,k_2$ are determined by $k,r_1,\varepsilon$.
\end{lemma}

\begin{proof}
From Lemma~\ref{lem-prel-estim-electric-field-by-repres-rho}, we have
\begin{equation}\label{eq-pf-lem-prel-estim-electric-field-by-repres-rho-1}
\left\| E_{\mathrm{N}}(t) \right\|_{k+3} 
\leq 
C+\sum_{i=1}^2\left\| 
\int_{\mathbb{R}^3} \int_0^t (t-s) \tilde{f}_i\left| E_{\mathrm{N}}^{\rm pla}(s, x - (t-s)v) \right| \,\mathrm{d}s\,\mathrm{d}v 
\right\|_{k+3}.
\end{equation}
We shall estimate the two terms in the right hand side by the sum of the long-time integral and the short-time integral:
\begin{equation*}
\left\| \int_0^{t_0} \cdots \right\|_{k+3} + \left\| \int_{t_0}^t \cdots \right\|_{k+3},
\end{equation*}
For the long-time integral over $(0, t_0)$, by the H\"older inequality with $r' = 3,\, r = \frac{3}{2}$, Lemma~\ref{lem-interpo-h-LP} and \eqref{eq-Lp-uniform-bound-electric-field}
\begin{align}
\left\| \int_0^{t_0} \cdots \right\|_{k+3}
&\leq C \int_0^{t_0} \frac{1}{t-s}\left(\| E_{\mathrm{N}}(s) \|_{\frac{3}{2}}+\| F_{\mathrm{N}}(s) \|_{\frac{3}{2},{\rm w}}\right) H_k(s)^{\frac{1}{k+3}}\,\mathrm{d}s\nonumber \\
&\leq C \ln \frac{t}{t-t_0} \sup_{s \in (0, t)} \left(\| E_{\mathrm{N}}(s) \|_{\frac{3}{2}}+\| F_{\mathrm{N}}(s) \|_{\frac{3}{2},{\rm w}}\right) H_k(s)^{\frac{1}{k+3}}\nonumber\\
&\leq C \ln \frac{t}{t-t_0}  H_k(t)^{\frac{1}{k+3}}.\label{eq-pf-lem-prel-estim-electric-field-by-repres-rho-2}
\end{align}

For the short-time integral over $(t_0, t)$, note
\begin{align*}
&\left\|\int_{t_0}^{t} \cdots \tilde{f}_i\left| E_{\mathrm{N}}^{\rm pla}\right|\cdots \right\|_{k+3}\\
&\leq \left\|\int_{t_0}^{t}\cdots \tilde{f}_i\left| E_{\mathrm{N}}\right| \cdots \right\|_{k+3}+\left\|\int_{t_0}^{t}\cdots \tilde{f}_i\left| F_{\mathrm{N}}\right|\cdots \right\|_{k+3}.
\end{align*}
Take $\frac{3}{2}<r_1\le\frac{15}{4}$, by the H\"older inequality, Lemma~\ref{lem-interpo-h-LP} and \eqref{eq-Lp-uniform-bound-electric-field}, we have
\begin{align}
&\left\|\int_{t_0}^{t}\cdots \tilde{f}_i\left| E_{\mathrm{N}}\right| \cdots \right\|_{k+3}\nonumber \\
&\leq C \int_{t_0}^{t} (t-s)^{1-\frac{3}{r_1}}\| E_{\mathrm{N}}(s) \|_{r_1} H_{k_1}(s)^{\frac{1}{k+3}}\,\mathrm{d}s\nonumber \\
&\leq C  (t-t_0)^{2-\frac{3}{r_1}} H_{k_1}(t)^{\frac{1}{k+3}},\label{eq-pf-lem-prel-estim-electric-field-by-repres-rho-3}
\end{align}
where $k_1$ is given by $\frac{k_1+3}{3}=\frac{k+3}{r_1'}$. Note
\begin{align*}
&\left\|\int_{t_0}^{t}\cdots \tilde{f}_i|F_{\mathrm{N}}|\cdots \right\|_{k+3}\\
&\le C\sum_{\alpha}\left\| 
\int_{\mathbb{R}^3} \int_{t_0}^{t}  \frac{(t-s) \tilde{f}_i}{|x - (t-s)v-\xi_{\alpha}(s)|^2} \,\mathrm{d}s\,\mathrm{d}v
\right\|_{k+3},
\end{align*}
and by the H\"older's inequality with $r_2=\frac{3}{2+\varepsilon/2}$,
\begin{align*}
&\int_{\mathbb{R}^3} \frac{(t-s) \tilde{f}_i}{|x - (t-s)v-\xi_{\alpha}(s)|^2} \,\mathrm{d}v\\
&=\int_{\mathbb{R}^3} \frac{(t-s)^{1+\varepsilon} |v|^{\varepsilon}\tilde{f}_i}{|x - (t-s)v-\xi_{\alpha}(s)|^2|(t-s)v|^{\varepsilon}} \,\mathrm{d}v\\
&\le C(t-s)^{1+\varepsilon-\frac{3}{r_2}}\left(\int_{\mathbb{R}^3} \frac{1}{|x - v-\xi_{\alpha}(s)|^{2r_2}|v|^{\varepsilon r_2}} \,\mathrm{d}v\right)^{
\frac{1}{r_2}}\left(\int_{\mathbb{R}^3} |v|^{\varepsilon r_2'}\tilde{f}_i^{r_2'}\,\mathrm{d}v\right)^{
\frac{1}{r_2'}}\\
&\le C(t-s)^{1+\varepsilon-\frac{3}{r_2}}|x -\xi_{\alpha}(s)|^{\frac{3}{r_2}-(2+\varepsilon)}\left(\int_{\mathbb{R}^3} |v|^{\varepsilon r_2'}\tilde{f}_i^{r_2'}\,\mathrm{d}v\right)^{
\frac{1}{r_2'}},
\end{align*}
where $\varepsilon>0$ is small enough. Hence we have
\begin{align*}
&\left\|\int_{t_0}^{t}\cdots \tilde{f}_i|F_{\mathrm{N}}|\cdots \right\|_{k+3}\\
&\le C\sum_{\alpha}\int_{t_0}^{t}(t-s)^{\frac{\varepsilon}{2}-1}\left\||x -\xi_{\alpha}(s)|^{-\frac{\varepsilon}{2}}\left(\int_{\mathbb{R}^3} |v|^{\varepsilon r_2'}\tilde{f}_i^{r_2'}\,\mathrm{d}v\right)^{
\frac{1}{r_2'}}\right\|_{k+3}\,\mathrm{d}s.
\end{align*}
Now, take $r_3=\frac{2}{\varepsilon(k+3)}$, by the H\"older's inequality, we have
\begin{align*}
&\left\||x -\xi_{\alpha}(s)|^{-\frac{\varepsilon}{2}}\left(\int_{\mathbb{R}^3} |v|^{\varepsilon r_2'}\tilde{f}_i^{r_2'}\,\mathrm{d}v\right)^{
\frac{1}{r_2'}}\right\|_{k+3}\\
&\le C\left\||x -\xi_{\alpha}(s)|^{-\frac{\varepsilon(k+3)}{2}}\right\|_{r_3}^{\frac{1}{k+3}}\left\|\left(\int_{\mathbb{R}^3} |v|^{\varepsilon r_2'}\tilde{f}_i^{r_2'}\,\mathrm{d}v\right)^{
\frac{k+3}{r_2'}}\right\|_{r_3'}^{\frac{1}{k+3}}\\
&\le C\left\|\left(\int_{\mathbb{R}^3} |v|^{\varepsilon r_2'}\tilde{f}_i\,\mathrm{d}v\right)^{
\frac{k+3}{r_2'}}\right\|_{r_3'}^{\frac{1}{k+3}}.
\end{align*}
By Lemma~\ref{lem-interpo-h-LP}, we have
\begin{equation*}
\left\|\left(\int_{\mathbb{R}^3} |v|^{\varepsilon r_2'}\tilde{f}_i\,\mathrm{d}v\right)^{
\frac{k+3}{r_2'}}\right\|_{r_3'}^{\frac{1}{k+3}}\le CH_{k_2}(s)^{\frac{1}{k+3}},
\end{equation*}
where $k_2$ is given by $\frac{k_2+3}{\varepsilon r_2'+3}=\frac{(k+3)r_3'}{r_2'}$. Hence
\begin{equation}\label{eq-pf-lem-prel-estim-electric-field-by-repres-rho-4}
\left\|\int_{t_0}^{t}\cdots \tilde{f}_i|F_{\mathrm{N}}|\cdots \right\|_{k+3}\le C(t-t_0)^{\frac{\varepsilon}{2}}H_{k_2}(t)^{\frac{1}{k+3}}.
\end{equation}
Inserting \eqref{eq-pf-lem-prel-estim-electric-field-by-repres-rho-2}, \eqref{eq-pf-lem-prel-estim-electric-field-by-repres-rho-3}, \eqref{eq-pf-lem-prel-estim-electric-field-by-repres-rho-4} into \eqref{eq-pf-lem-prel-estim-electric-field-by-repres-rho-1}, we have
\begin{align}
\left\| E_{\mathrm{N}}(t) \right\|_{k+3} 
\leq 
&C+C \ln \frac{t}{t-t_0}  H_k(t)^{\frac{1}{k+3}}+C  (t-t_0)^{2-\frac{3}{r_1}} H_{k_1}(t)^{\frac{1}{k+3}}\nonumber\\
&+C(t-t_0)^{\frac{\varepsilon}{2}}H_{k_2}(t)^{\frac{1}{k+3}}.\label{eq-pf-lem-prel-estim-electric-field-by-repres-rho-5}
\end{align}

Now we estimate $H_{k_1},H_{k_2}$ by $H_k$. It can be deduced by Proposition~\ref{prn-Hk-by-Hk-E-LP} that
\begin{equation*}
H_{k_1}(t)\le  Ce^{Ct}\left[H_{k_1}(0)+\sup_{0\le s\le t}\|E_{\mathrm{N}}(s)\|_{k_1+3}^{k_1+3}\right].
\end{equation*}
By Lemma~\ref{lem-elliptic-estimate-electric-field} and Lemma~\ref{lem-interpo-h-LP},
\begin{equation*}
\|E_{\mathrm{N}}(s)\|_{k_1+3}^{k_1+3}\le C\|\rho(s)\|_{\frac{3(k_1+3)}{k_1+6}}^{k_1+3}\le CH_{\frac{6k_1+9}{k_1+6}}(t)^{\frac{k_1+6}{3}}.
\end{equation*}
Take $r_1'=6+9\gamma-k$. Since $k>3$, we have $r_1'\in[\frac{15}{11},3)$ as long as $\gamma\in(\max(0,\frac{11k-51}{99}),\frac{k-3}{9})$. Then recall $k_1=\frac{3(k+3-r_1')}{r_1'}$, we have $\frac{6k_1+9}{k_1+6}=\frac{k-3\gamma}{1+\gamma}< k$. Note $\frac{k_1+6}{3(k_1+3)}=\frac{3(1+\gamma)}{k+3}$ and by Lemma~\ref{lem-interpo-h-LP} with $b=\frac{6k_1+9}{k_1+6}$, $a=2$, $c=k$, we have
\begin{equation*}
H_{\frac{6k_1+9}{k_1+6}}(t)^{\frac{k_1+6}{3}}\le CH_k(t)^{(1-\frac{5\gamma}{k-2})\frac{3(k_1+3)}{k+3}}\le CH_k(t)^{\frac{3(k_1+3)}{k+3}},
\end{equation*}
Hence choosing appropriate $r_1\in(\frac{3}{2},\frac{15}{4}]$, we have
\begin{equation}\label{eq-pf-lem-prel-estim-electric-field-by-repres-rho-6}
H_{k_1}(t)\le Ce^{Ct}\left[H_{k_1}(0)+H_k(t)^{\frac{3(k_1+3)}{k+3}}\right],
\end{equation}
with $k_1$ determined by $k,r_1$.

Now by the definitions of $r_2,r_3,k_2$, we have
\begin{equation*}
\frac{6k_2+9}{k_2+6}=\frac{(12k+18)+(15k+48)\varepsilon}{(2k+12)-(6+2k)\varepsilon}.
\end{equation*}
Note $k>3$ implies that, there exists $\varepsilon>0$ such that
\begin{equation*}
\frac{(12k+18)+(15k+48)\varepsilon}{(2k+12)-(6+2k)\varepsilon}<k.
\end{equation*}
Hence, similarly to \eqref{eq-pf-lem-prel-estim-electric-field-by-repres-rho-6}, we can deduce that
\begin{equation}\label{eq-pf-lem-prel-estim-electric-field-by-repres-rho-7}
H_{k_2}(t)\le Ce^{Ct}\left[H_{k_2}(0)+H_k(t)^{\frac{3(k_2+3)}{k+3}}\right],
\end{equation}
with $k_2$ determined by $k,\varepsilon$. 

Inserting \eqref{eq-pf-lem-prel-estim-electric-field-by-repres-rho-6}, \eqref{eq-pf-lem-prel-estim-electric-field-by-repres-rho-7} into \eqref{eq-pf-lem-prel-estim-electric-field-by-repres-rho-5}, we have
\begin{align*}
\left\| E_{\mathrm{N}}(t) \right\|_{k+3} 
\leq 
&C+C \ln \frac{t}{t-t_0}  H_k(t)^{\frac{1}{k+3}}+C(t-t_0)^{2-\frac{3}{r_1}}e^{Ct} H_{k}(t)^{\frac{3(k_1+3)}{(k+3)^2}}\nonumber\\
&+C(t-t_0)^{\frac{\varepsilon}{2}}e^{Ct}H_{k}(t)^{\frac{3(k_2+3)}{(k+3)^2}}.
\end{align*}

\end{proof}

\subsection{The fourth step: a Gr\"onwall's argument}

\begin{proposition}\label{prn-Hk-unif-bound-LP}
Let $f,\{\xi_{\alpha},\eta_{\alpha}\}$ be a classical solution given in Theorem~\ref{thm-local-wellposedness-absorption}. Assume $k>3$, then 
\[
H_k(t)\le\exp(\exp(C(1+t))).
\]
\end{proposition}

\begin{proof}
Denote $\mu=\frac{t-t_0}{t}$, $\delta=\min\{2-\frac{3}{r_1},\frac{\varepsilon}{2}\}$, $\bar{k}=\max\{\frac{3(k_1+3)}{(k+3)^2},\frac{3(k_2+3)}{(k+3)^2}\}$. We have from Lemma~\ref{lem-EN-by-Hk-LP}
\begin{equation}\label{eq-pf-prn-Hk-unif-bound-LP-1}
\|E_{\mathrm{N}}(t)\|_{k+3}\le C+C \ln\big(\frac{1}{\mu}\big)H_k(t)^{\frac{1}{k+3}}+C\mu^{\delta}e^{Ct}H_k(t)^{\bar{k}}.
\end{equation}
Now choosing $t_0$ such that
\begin{equation*}
\mu^{\delta}H_k(t)^{\bar{k}}=1,
\end{equation*}
inserting it into \eqref{eq-pf-prn-Hk-unif-bound-LP-1}, we have
\begin{equation*}
\|E(t)\|_{k+3}\le Ce^{Ct}+C \ln H_kH_k(t)^{\frac{1}{k+3}}.
\end{equation*}
Inserting it into the estimate in Proposition~\ref{prn-Hk-by-Hk-E-LP}, we have
\begin{equation*}
H_k(t)\le CH_{k}(t)^{\frac{k+2}{k+3}}e^{Ct}+CH_{k}(t)^{\frac{k+2}{k+3}}\int_0^t\ln H_k(s)H_k(s)^{\frac{1}{k+3}}\,ds,
\end{equation*}
solving this integral inequality, we get the moment propagation as follows
\begin{equation*}
H_k(t)\le \exp(\exp(C(1+t))).
\end{equation*}
\end{proof}

\section{The global existence for reflection-Neumann boundary}\label{sec-global-reflection}

We define the pointwise energy related to the $\alpha$-th charge as in \cite{MMP11}
\begin{equation*}
h_{\alpha}(t,x,v)=\frac{\vert v-\eta_{\alpha}(t)\vert^2}{2}+\frac{1}{\vert x-\xi_{\alpha}(t)\vert}+K_1,
\end{equation*}
where $K_1>1$ is a large constant depending only on $\|\mathring{f}\|_{1}$, $\|\mathring{f}\|_{\infty}$, $\mathcal{E}(0)$. We define
\begin{equation*}
Q_{t,\delta}=\max_{\alpha}\sup\{\sqrt{h_{\alpha}}(t,x,v): (x,v)\in\operatorname{supp} f(s),\,t-\delta\le s\le t\}.
\end{equation*}
Note
\begin{equation*}
Q_{t,\delta}=\max_{\alpha}\sup\{\sqrt{h_{\alpha}}(t,X(s),V(s)): (x,v)\in\operatorname{supp} f(t-\delta),\,t-\delta\le s\le t\}.
\end{equation*}
where $X(s)=X(s,t-\delta,x,v)$, $V(s)=V(s,t-\delta,x,v)$.

Theorem~\ref{thm-reflection-Neumann-case} follows by verifying the continuation criterion in Theorem~\ref{thm-local-wellposedness-reflection} for $\#=\mathrm{N}$. Note \eqref{eq-continuation-point-charge-reflection} can be obtained by \eqref{eq-uniform-bound-point-charge-1}, \eqref{eq-uniform-bound-point-charge-2}. The aim of the section is to prove \eqref{eq-continuation-plasma-reflection}, which is the following proposition.
\begin{proposition}\label{prn-uniform-bound-pointwise-energy}
We have for all $T>0$
\begin{equation*}
Q_{T,T}\le(Q_{0,0}+C)\exp(C(1+T)).
\end{equation*}
\end{proposition}

The method relies on a suitable splitting of $[0,T]$ into small intervals. More precisely, we set
\begin{equation}\label{eq-setting-delta}
\delta=c_0Q_{T,T}^{-1},
\end{equation}
where $c_0$ will be chosen later small enough depending only on $\mathcal{E}(0),\|\mathring{f}\|_{1},\|\mathring{f}\|_{\infty}$. Fix arbitrary $t\in[\delta,T]$, we devote to obtain the estimate on $Q_{t,\delta}$.

Note
\begin{align*}
&\frac{\mathrm{d}}{\mathrm{d}s}\sqrt{h_{\alpha}}(s,X(s),V(s))\\
&=\frac{V(s)-\eta_{\alpha}(s)}{2\sqrt{h_{\alpha}}(X(s),V(s))}\cdot\left(E_{\mathrm{N}}^{\rm pla}(s,X(s))-E_{\mathrm{N},\alpha}^{\rm cha}(s)-\frac{X(s)-\xi_{\alpha}(s)}{\vert X(s)-\xi_{\alpha}(s)\vert^3}\right)\\
&\le C\iint_{\Omega\times\mathbb{R}^3}\frac{f(s,y,w)}{\vert X(s)-y\vert^2}\,\mathrm{d}y\,\mathrm{d}w+C\sum_{\beta:\beta\ne\alpha}\frac{1}{\vert X(s)-\xi_{\alpha}(s)\vert^2}\\
&\quad+C\sum_{\alpha}\iint_{\Omega\times\mathbb{R}^3}\frac{f(s,y,w)}{\vert \xi_{\alpha}(s)-y\vert^2}\,\mathrm{d}y\,\mathrm{d}w+C,
\end{align*}
where $C$ is a constant depending on \eqref{eq-uniform-bound-point-charge-1}, \eqref{eq-uniform-bound-point-charge-2}.

Let $\tau\in[t-\delta,t]$. Integrating the inequality from $t-\delta$ to $\tau$ to obtain
\begin{align}
&\vert\sqrt{h_{\alpha}}(\tau,X(\tau),V(\tau))-\sqrt{h_{\alpha}}(t-\delta,x,v)\vert\nonumber\\
&\le C\int_{t-\delta}^{\tau}\iint_{\Omega\times\mathbb{R}^3}\frac{f(s,y,w)}{\vert X(s)-y\vert^2}\,\mathrm{d}y\,\mathrm{d}w\,\mathrm{d}s+C\sum_{\beta:\beta\ne\alpha}\int_{t-\delta}^{\tau}\frac{1}{\vert X(s)-\xi_{\alpha}(s)\vert^2}\,\mathrm{d}s\nonumber\\
&\quad+C\sum_{\alpha}\int_{t-\delta}^{\tau}\iint_{\Omega\times\mathbb{R}^3}\frac{f(s,y,w)}{\vert \xi_{\alpha}(s)-y\vert^2}\,\mathrm{d}y\,\mathrm{d}w\,\mathrm{d}s+C(\tau-t+\delta).\label{eq-h-alpha-1}
\end{align}
By \eqref{eq-pre-estimates-electric-field}, \eqref{eq-setting-delta} and taking $c_0$ small enough such that we have
\begin{equation}\label{eq-h-alpha-2}
\vert\sqrt{h_{\alpha}}(\tau,X(\tau),V(\tau))-\sqrt{h_{\alpha}}(t-\delta,x,v)\vert\le CQ_{t,\delta}^2\delta\le \frac{Q_{t,\delta}}{32}.
\end{equation}

To obtain the estimate on $Q_{t,\delta}$, we need to estimate the integral in the right hand side of \eqref{eq-h-alpha-1}. We denote
\begin{align*}
\bar{\mathcal{I}}_{\alpha}&=\sum_{\beta:\beta\ne\alpha}\int_{t-\delta}^{t}\frac{1}{\vert X(s)-\xi_{\beta}(s)\vert^2}\,\mathrm{d}s,\\
\tilde{\mathcal{I}}&=\sum_{\alpha}\int_{t-\delta}^{t}\iint_{\Omega\times\mathbb{R}^3}\frac{f(s,y,w)}{\vert \xi_{\alpha}(s)-y\vert^2}\,\mathrm{d}y\,\mathrm{d}w\,\mathrm{d}s.
\end{align*}

By measure preserving transformation $(y,w)\mapsto(Y(s),W(s))$, we have:
\begin{align}
\mathcal{J}:=&\int_{t-\delta}^{t}\iint_{\Omega\times\mathbb{R}^3}\frac{f(s,y,w)}{\vert X(s)-y\vert^2}\,\mathrm{d}y\,\mathrm{d}w\,\mathrm{d}s\nonumber\\
=&\int_{t-\delta}^{t}\iint_{\Omega\times\mathbb{R}^3}\frac{f(t-\delta,y,w)}{\vert X(s)-Y(s)\vert^2}\,\mathrm{d}y\,\mathrm{d}w\,\mathrm{d}s,\label{eq-measure-preserving-transformation-J}
\end{align}
where $Y(s)=Y(s,t-\delta,y,w)$, $W(s)=W(s,t-\delta,y,w)$. According to \eqref{eq-measure-preserving-transformation-J}, we will divide $\mathcal{J}$ into several parts for the purpose of separation of the singular sets.

\subsection{Separation of the singular sets}

Recall that $\delta_2$ below is given in Proposition~\ref{prn-conservation-laws}.

\begin{lemma}\label{lem-Separation-singular-sets}
Let $(y,w)\in\operatorname{supp}f(t-\delta)$ and $Y(s)=Y(s,t-\delta,y,w)$, $W(s)=W(s,t-\delta,y,w)$. If there exist $t_0\in(t-\delta,t)$, $\alpha\in\mathcal{I}_M$ and $a\in(0,\frac{31}{32})$ such that $\vert Y(t_0)-\xi_{\alpha}(t_0)\vert\le a\delta_2$, then $\forall s\in(t-\delta,t)$
\begin{align}
\vert Y(s)-\xi_{\alpha}(s)\vert&\le \left(a+\frac{1}{32}\right)\delta_2,\label{eq-Separation-singular-sets-1}\\
\operatorname{d}_{\partial\Omega}(Y(s))&\ge\left(\frac{31}{32}-a\right)\delta_2,\label{eq-Separation-singular-sets-2}\\
\min_{\beta:\beta\ne\alpha}\vert Y(s)-\xi_{\beta}(s)\vert&\ge \left(\frac{31}{32}-a\right)\delta_2.\label{eq-Separation-singular-sets-3}
\end{align}

If there exist $t_0\in(t-\delta,t)$ and $b>\frac{1}{32}$ such that $\min_{\beta}\vert Y(t_0)-\xi_{\beta}(t_0)\vert\ge b\delta_2$, then $\forall s\in(t-\delta,t)$
\begin{equation}\label{eq-Separation-singular-sets-4}
\min_{\beta}\vert Y(s)-\xi_{\beta}(s)\vert\ge \left(b-\frac{1}{32}\right)\delta_2.
\end{equation}

\end{lemma}
\begin{proof}
Firstly we consider the case $\vert Y(t_0)-\xi_{\alpha}(t_0)\vert\le a\delta_2$. Denote $I(s)=\vert Y(s)-\xi_{\alpha}(s)\vert^2$. Differentiating $I(s)$ we get 
\begin{equation*}
\dot{I}=2(Y-\xi_{\alpha})\cdot(W-\eta_{\alpha}).
\end{equation*}
Then, by \eqref{eq-setting-delta}, for all $s\in(t-\delta,t)$
\begin{equation*}
\vert\sqrt{I(s)}-\sqrt{I(t_0)}\vert\le 2Q_{t,\delta}\delta\le\frac{\delta_2}{32}
\end{equation*}
and \eqref{eq-Separation-singular-sets-1} follows. \eqref{eq-Separation-singular-sets-4} can be proved similarly. \eqref{eq-Separation-singular-sets-2}, \eqref{eq-Separation-singular-sets-3} can be deduced immediately by \eqref{eq-Separation-singular-sets-1} and \eqref{eq-uniform-bound-point-charge-1}, \eqref{eq-uniform-bound-point-charge-2}.
\end{proof}

Now we select $(\bar{x},\bar{v})\in\operatorname{supp}f(t-\delta)$, $\bar{\alpha}\in\mathcal{I}_M$, $\bar{s}\in[t-\delta,t]$ such that 
\begin{equation*}
\sqrt{h_{\bar{\alpha}}}(\bar{s},X(\bar{s},t-\delta,\bar{x},\bar{v}),V(\bar{s},t-\delta,\bar{x},\bar{v}))=Q_{t,\delta}.
\end{equation*}
By \eqref{eq-h-alpha-2}, we have $\sqrt{h_{\bar{\alpha}}}(t-\delta,\bar{x},\bar{v})\ge Q_{t,\delta}/2$. Abuse of notation,  we denote $X(s)=X(s,t-\delta,\bar{x},\bar{v})$, $V(s)=V(s,t-\delta,\bar{x},\bar{v})$. We denote $\Gamma=\Omega\times\mathbb{R}^3$, $\Pi=(t-\delta,t)\times\Gamma$. Thanks to Lemma~\ref{lem-Separation-singular-sets}, we can consider two cases as follows. 


{\bf Case I.} There exists $t_0\in(t-\delta,t)$ such that $\min_{\alpha}\vert X(t_0)-\xi_{\alpha}(t_0)\vert>\frac{\delta_2}{4}$. Then, by \eqref{eq-Separation-singular-sets-4}, we have 
\begin{equation*}
\inf_{s\in(t-\delta,t)}\min_{\alpha}\vert X(s)-\xi_{\alpha}(s)\vert\ge\frac{7\delta_2}{32}.
\end{equation*}
We split $\Pi$ into two parts as follows.
\begin{align*}
\Pi&=A_1+A_2=(t-\delta,t)\times \bar{A}_1+(t-\delta,t)\times \bar{A}_2,\\
\bar{A}_1&=\left\{(y,w)\in\Gamma:\exists t_0=t_0(y,w)\in(t-\delta,t)\;\text{s.t.}\;\min_{\alpha}\vert Y(t_0)-\xi_{\alpha}(t_0)\vert>\frac{3\delta_2}{16}\right\},\\
\bar{A}_2&=\left\{(y,w)\in\Gamma:\sup_{s\in(t-\delta,t)}\min_{\alpha}\vert Y(s)-\xi_{\alpha}(s)\vert\le\frac{3\delta_2}{16}\right\}.
\end{align*}
By Lemma~\ref{lem-Separation-singular-sets}, we have
\begin{equation*}
\bar{A}_1\subset\left\{(y,w)\in\Gamma:\inf_{s\in(t-\delta,t)}\min_{\alpha}\vert Y(s)-\xi_{\alpha}(s)\vert>\frac{5\delta_2}{32}\right\},
\end{equation*}
and it is obvious that
\begin{equation*}
\bar{A}_2\subset\left\{(y,w)\in\Gamma:\inf_{s\in(t-\delta,t)}\operatorname{d}_{\partial\Omega}(Y(s))\ge\frac{13\delta_2}{16}\right\}.
\end{equation*}

{\bf Case II.} $\sup_{s\in(t-\delta,t)}\min_{\alpha}\vert X(s)-\xi_{\alpha}(s)\vert\le\frac{\delta_2}{4}$. By \eqref{eq-uniform-bound-point-charge-2} and Lemma~\ref{lem-Separation-singular-sets}, there exists $\alpha_0\in\mathcal{I}_{M}$ such that $\sup_{s\in(t-\delta,t)}\vert X(s)-\xi_{\alpha_0}(s)\vert\le\frac{\delta_2}{4}$ and $\inf_{s\in(t-\delta,t)}\min_{\beta:\beta\ne\alpha_0}\vert X(s)-\xi_{\beta}(s)\vert\ge\frac{3\delta_2}{4}$. Note also in this case, \eqref{eq-h-alpha-2} for $\alpha=\alpha_0$ can be improved to
\begin{equation}\label{eq-h-alpha-3}
\vert\sqrt{h_{\alpha_0}}(\tau,X(\tau),V(\tau))-\sqrt{h_{\alpha_0}}(t-\delta,x,v)\vert\le \frac{1}{32}Q_{t,\delta}^{\frac{1}{3}}.
\end{equation}
We split $\Pi$ into two parts as follows.
\begin{align*}
\Pi&=A_3+A_4=(t-\delta,t)\times \bar{A}_3+(t-\delta,t)\times \bar{A}_4,\\
\bar{A}_3&=\left\{(y,w)\in\Gamma:\inf_{s\in(t-\delta,t)}\vert Y(s)-\xi_{\alpha_0}(s)\vert>\frac{3\delta_2}{8}\right\},\\
\bar{A}_4&=\left\{(y,w)\in\Gamma:\;\exists t_0=t_0(y,w)\in(t-\delta,t)\;\text{s.t.}\;\vert Y(t_0)-\xi_{\alpha_0}(t_0)\vert\le\frac{3\delta_2}{8}\right\}.
\end{align*}
By Lemma~\ref{lem-Separation-singular-sets}, we have
\begin{equation*}
\bar{A}_4\subset\left\{(y,w)\in\Gamma:\inf_{s\in(t-\delta,t)}\operatorname{d}_{\partial\Omega}(Y(s))\ge\frac{19\delta_2}{32}\right\}.
\end{equation*}
Note that there are two subcases for Case II, $\alpha_0=\bar{\alpha}$ and $\alpha_0\ne\bar{\alpha}$. They can be considered uniformly by the following lemma.
\begin{lemma}\label{lem-estimate-a0-ne-bar-alpha}
There exists a constant $C>0$ such that for $t\in[0,T]$ we have for all $\alpha,\beta$
\begin{equation*}
\sqrt{h_{\beta}}(t,X,V)\le\sqrt{h_{\alpha}}(t,X,V)+C,\quad\forall X\in B\left(\xi_{\alpha}(t),\frac{9\delta_2}{32}\right),\,\forall V\in\mathbb{R}^3.
\end{equation*}
\end{lemma}
Indeed, by the choice of $\bar{s},\bar{x},\bar{v}$ and Lemma~\ref{lem-estimate-a0-ne-bar-alpha}, we have for both subcases
\begin{equation*}
\sqrt{h_{\alpha_0}}(\bar{s},X(\bar{s}),V(\bar{s}))\ge\sqrt{h_{\bar{\alpha}}}(\bar{s},X(\bar{s}),V(\bar{s}))-C\ge\frac{3}{4}Q_{t,\delta}.
\end{equation*}
Combining with \eqref{eq-h-alpha-3}, we have
\begin{equation}\label{eq-lower-bound-h-alpha0-t-delta}
\sqrt{h_{\alpha_0}}(t-\delta,\bar{x},\bar{v})\ge Q_{t,\delta}/2.
\end{equation}

We denote
\begin{equation*}
\mathcal{J}_i:=\int_{A_i}\frac{f(t-\delta,y,w)}{\vert X(s)-Y(s)\vert^2}\,\mathrm{d}y\,\mathrm{d}w\,\mathrm{d}s.
\end{equation*}
Then $\mathcal{J}$ can be written as:
\begin{equation*}
\mathcal{J}=\mathcal{J}_1+\mathcal{J}_2=\mathcal{J}_3+\mathcal{J}_4.
\end{equation*}

Note for $A_2,A_3$, we have $\vert X(s)-Y(s)\vert\ge\frac{\delta_1}{32}$, hence $\mathcal{J}_2,\mathcal{J}_3$ can be handled directly as 
\begin{equation}\label{eq-estimate-J2-J3}
\mathcal{J}_2\le C\delta_2^{-2}\|f\|_{L^1}\delta,\quad\mathcal{J}_3\le C\delta_2^{-2}\|f\|_{L^1}\delta.
\end{equation}

The core in this section is to estimate $\mathcal{J}_1$, $\mathcal{J}_4$. The main ideas are as follows: in $A_1$, the characteristics $X$ and $Y$ keep a strictly positive distance from the point charges, hence we only need to consider the boundary effect for $\mathcal{J}_1$; in $A_4$ and Case II.2, the characteristics $X$ and $Y$ keep a strictly positive distance from the boundary $\partial\Omega$, hence we only need to consider the impact of the point charge for $\mathcal{J}_4$. Based on these observations, we adapt the method of \cite{HV10} to the estimate on $\mathcal{J}_1$ and \cite{MMP11} to the estimate on $\mathcal{J}_4$ with proper modification.

\subsection{The estimate on $\mathcal{J}_1$}

We further split the time interval in $\mathcal{J}_1$ into smaller step-sized intervals as follows:
\begin{equation*}
\tilde{\mathcal{J}}_1(t):=\int_{t-\tilde{\delta}}^t\int_{\bar{A}_1}\frac{f(t-\delta,y,w)}{\vert X(s)-Y(s)\vert^2}\,\mathrm{d}y\,\mathrm{d}w\,\mathrm{d}s,
\end{equation*}
where
\begin{equation}\label{eq-setting-tilde-delta}
\tilde{\delta}=c_0Q_{T,T}^{-\frac{5}{4}}.
\end{equation}

The following result  establishes that characteristic flows bouncing repeatedly in some time interval should have a small value of the normal component $v_{\perp}$. 
\begin{lemma}\label{lem-estimate-v-perp}
Let $(X(s),V(s))$ be the characteristics satisfying Case I. If $(X,V)$ has more than one bounce in the interval $[t-\tilde{\delta},t]$, then we have for all $s\in[t-\tilde{\delta},t]$
\begin{equation*}
\vert v_{\perp}(s)\vert\le CQ_{t,\tilde{\delta}}^2\tilde{\delta}.
\end{equation*}
\end{lemma}
\begin{proof}
By \eqref{eq-pre-estimates-electric-field} and \eqref{eq-character-ODE-new-coordinates}, we have
\begin{equation}\label{eq-estimate-derivative-v-perp}
\left\vert\frac{\mathrm{d}\vert v_{\perp}\vert}{\mathrm{d}s}\right\vert\le CQ_{t,\tilde{\delta}}^{2}.
\end{equation}
If $(X,V)$ has more than one bounce, then we have $v_{\perp}(\bar{s})=0$ for some $\bar{s}\in [t-\tilde{\delta},t]$. Integrating \eqref{eq-estimate-derivative-v-perp} we have
\begin{equation*}
\vert v_{\perp}\vert\le CQ_{t,\tilde{\delta}}^{2}\tilde{\delta}.
\end{equation*}
\end{proof}

The following lemma implies the stability of the velocity during a small time interval. Note that $Q_{t,\tilde{\delta}}^2$ in Lemma~\ref{lem-estimate-v-perp} and \eqref{eq-boundary-stability-velocity-2} is the main reason why we replace $\delta$ with $\tilde{\delta}$, due to the unflat boundary.
\begin{lemma}\label{lem-boundary-stability-velocity}
Let $(X(s),V(s))$ be the characteristics satisfying Case I.  If there is no bounce in $(t-\tilde{\delta},t)$, then for all $s,\tau\in(t-\tilde{\delta},t)$, we have
\begin{equation}\label{eq-boundary-stability-velocity-0}
\vert V(s)-V(\tau)\vert\le CQ_{t,\tilde{\delta}}^{4/3}\tilde{\delta}.
\end{equation}
If there is exactly one bounce at $t_1\in(t-\tilde{\delta},t)$, then we have
\begin{equation}\label{eq-boundary-stability-velocity-1}
\begin{split}
\vert V(s)-V(t-\tilde{\delta})\vert&\le CQ_{t,\tilde{\delta}}^{4/3}\tilde{\delta},\;\text{if}\;\;t-\tilde{\delta}<s<t_1,\\
\vert V(s)-V(t)\vert&\le CQ_{t,\tilde{\delta}}^{4/3}\tilde{\delta},\;\text{if}\;\;t_1<s<t.
\end{split}
\end{equation}
If there are at least two bounce in $(t-\tilde{\delta},t)$, then for all $s,\tau\in(t-\tilde{\delta},t)$, we have
\begin{equation}\label{eq-boundary-stability-velocity-2}
\vert V(s)-V(\tau)\vert\le CQ_{t,\tilde{\delta}}^2\tilde{\delta}.
\end{equation}
\end{lemma}
\begin{proof}
\eqref{eq-boundary-stability-velocity-0} is standard. \eqref{eq-boundary-stability-velocity-1} is a consequence of \eqref{eq-boundary-stability-velocity-0}. Now we consider \eqref{eq-boundary-stability-velocity-2}. W.L.G., assume $s>\tau$. We can prove analogously to Lemma~\ref{lem-Separation-singular-sets} that $X$ is close enough to the boundary in $[t-\tilde{\delta},t]$. Hence we can write $V$ by
\begin{equation*}
V(s)=\sum_{i=1,2}w_i(s)u_i(s)-v_{\perp}(s)n(s).
\end{equation*}
Hence
\begin{align*}
\vert V(s)-V(\tau)\vert=&\left\vert\sum_{i=1,2}w_i(s)u_i(s)-v_{\perp}(s)n(s)-\sum_{i=1,2}w_i(\tau)u_i(\tau)+v_{\perp}(\tau)n(\tau)\right\vert\\
\le&\sum_{i=1,2}\vert w_i(s)-w_i(\tau)\vert\vert u_i(s)\vert+\sum_{i=1,2}\vert u_i(s)-u_i(\tau)\vert\vert w_i(\tau)\vert+\vert v_{\perp}(\tau)\vert+\vert v_{\perp}(s)\vert.
\end{align*}
Since $\partial\Omega$ is $C^4$ and $\min_{\alpha}\vert X(s)-\xi_{\alpha}(s)\vert\ge\frac{7\delta_2}{32}$, by \eqref{eq-pre-estimates-electric-field} and \eqref{eq-character-ODE-new-coordinates} we have
\begin{align*}
&\vert\partial_t\mu_i\vert\le CQ_{t,\tilde{\delta}},\quad \vert\partial_tw_i\vert\le CQ_{t,\tilde{\delta}}^2,\\
&\vert u_i(\tau)-u_i(s)\vert\le\sum_{j=1,2}\int_{\tau}^s\vert\partial_{\mu_j}u_i\partial_t\mu_j\vert\,\mathrm{d}\tau\le CQ_{t,\tilde{\delta}}(s-\tau).
\end{align*}
And by Lemma~\ref{lem-estimate-v-perp}, we finish the proof.
\end{proof}

Define $\tilde{L}(w)=\min\{\vert w-V(t-\tilde{\delta})\vert,\vert w-V(t)\vert\}$. We split $A_1$ into three parts by
\begin{align*}
\tilde{S}_1&=\left\{(s,y,w)\in A_1:\min\{\vert W(s)\vert,\tilde{L}(W(s))\}\le\tilde{R}\right\},\\
\tilde{S}_2&=\left\{(s,y,w)\in A_1:\vert X(s)-Y(s)\vert\le\tilde{\mathcal{L}}\;\text{or}\;\vert X(s)-Y(s)\vert\ge\delta_2/16\right\}\setminus\tilde{S}_1,\\
\tilde{S}_3&=A_1\setminus(\tilde{S}_1\cup\tilde{S}_2),
\end{align*}
where
\begin{equation}\label{eq-J1-parameters-setting}
\tilde{R}=Q_{t,\tilde{\delta}}^{3/4},\quad\tilde{\mathcal{L}}=\frac{\tilde{l}}{\vert w\vert^2}\tilde{\mathcal{L}}_0,\quad\tilde{l}=Q_{t,\tilde{\delta}}^{1-\epsilon},
\end{equation}
$\epsilon>0$ is a small enough constant. According to Lemma~\ref{lem-boundary-stability-velocity}, we define $\tilde{\mathcal{L}}_0(s,y,w)$ as follows. If both $(X,V)$ and $(Y,W)$ have no bounce or have more than one bounce, then
\begin{equation*}
\tilde{\mathcal{L}}_0=\frac{1}{\tilde{L}(w)}.
\end{equation*}
If $(X,V)$ has no bounce or has more than one bounce, $(Y,W)$ has exactly one bounce at $t_1\in(t-\tilde{\delta},t)$, then
\begin{equation*}
\tilde{\mathcal{L}}_0=\left\{\begin{split}
\frac{1}{\tilde{L}(w)}\;&\;\text{if}\;\;t-\tilde{\delta}<s\le t_1,\\
\frac{1}{\tilde{L}(W(t))}\;&\;\text{if}\;\;t_1<s<t.
\end{split}\right.
\end{equation*}
If $(X,V)$ has exactly one bounce at $t_1\in(t-\tilde{\delta},t)$, $(Y,W)$ has no bounce or has more than one bounce, then
\begin{equation*}
\tilde{\mathcal{L}}_0=\left\{\begin{split}
\frac{1}{\vert w-V(t-\tilde{\delta})\vert}\;&\;\text{if}\;\;t-\tilde{\delta}<s\le t_1,\\
\frac{1}{\vert w-V(t)\vert}\;&\;\text{if}\;\;t_1<s<t.
\end{split}\right.
\end{equation*}
If $(X,V)$ has exactly one bounce at $t_1\in(t-\tilde{\delta},t)$, $(Y,W)$ has exactly one bounce at $t_2\in(t-\tilde{\delta},t)$. W.L.G., assume $t_2>t_1$, then
\begin{equation*}
\tilde{\mathcal{L}}_0=\left\{\begin{split}
\frac{1}{\vert w-V(t-\tilde{\delta})\vert}\;&\;\text{if}\;\;t-\tilde{\delta}<s\le t_1,\\
\frac{1}{\vert w-V(t)\vert}\;&\;\text{if}\;\;t_1<s\le t_2,\\
\frac{1}{\vert W(t)-V(t)\vert}\;&\;\text{if}\;\;t_2<s<t.
\end{split}\right.
\end{equation*}

\begin{lemma}\label{lem-J1-S2-S3-velocity-stability}
In the sets $\tilde{S}_2$ and $\tilde{S}_3$, we have
\begin{equation*}
\vert w\vert/2\le\vert W(s)\vert\le 2\vert w\vert,\quad\tilde{\mathcal{L}}_0^{-1}(s)\ge\tilde{L}(W(s))/2,\quad\vert V(s)-W(s)\vert\ge\tilde{R}/2.
\end{equation*}
\end{lemma}
\begin{proof}
By \eqref{eq-setting-tilde-delta}, we have $Q_{t,\tilde{\delta}}^2\tilde{\delta}\le c_0\tilde{R}$. The inequalities are consequence of Lemma~\ref{lem-boundary-stability-velocity} and the definition of $\tilde{\mathcal{L}}_0$.
\end{proof}

The following lemma will play a key role as the separation property in \cite[Lemma 14]{HV10}.
\begin{lemma}\label{lem-J1-S3-separation}
In the set $\tilde{S}_3$, we have
\begin{equation}\label{eq-J1-S3-separation}
\int_{t-\tilde{\delta}}^{t}\frac{\mathds{1}_{\{\vert X(s)-Y(s)\vert>\tilde{\mathcal{L}}\}}}{\vert X(s)-Y(s)\vert^2}\,\mathrm{d}s\le\frac{C\vert w\vert^2}{\tilde{l}}.
\end{equation}
\end{lemma}

\begin{proof}

We claim that there exist at most three points $\bar{t}_1,\bar{t}_2,\bar{t}_3\in [t-\tilde{\delta},t]$ such that
\begin{equation*}
\vert X(s)-Y(s)\vert\ge\frac{1}{8}\min_{i=1,2,3}\{\vert s-\bar{t}_i\vert\}\tilde{\mathcal{L}}_0^{-1}(s),\quad\forall s\in [t-\tilde{\delta},t].
\end{equation*}
W.L.G., we consider the case that $(X,V)$ has no bounce or has more than one bounce, $(Y,W)$ has exactly one bounce at $t_1\in(t-\tilde{\delta},t)$, then
\begin{align*}
&\int_{t-\tilde{\delta}}^{t}\frac{\mathds{1}_{\{\vert X(s)-Y(s)\vert>\tilde{\mathcal{L}}\}}}{\vert X(s)-Y(s)\vert^2}\,\mathrm{d}s\\
&\le\sum_{i=1}^3\frac{\vert w\vert^4\tilde{L}(w)^2}{\tilde{l}^2}\int_{\vert s-\bar{t}_i\vert\le m_1}\,\mathrm{d}s+\sum_{i=1}^3\frac{1}{\tilde{L}(w)^2}\int_{\vert s-\bar{t}_i\vert>m_1}\frac{64}{\vert s-\bar{t}_i\vert^2}\,\mathrm{d}s\\
&\quad+\sum_{i=1}^3\frac{\vert w\vert^4\tilde{L}(W(t))^2}{\tilde{l}^2}\int_{\vert s-\bar{t}_i\vert\le m_2}\,\mathrm{d}s+\sum_{i=1}^3\frac{1}{\tilde{L}(W(t))^2}\int_{\vert s-\bar{t}_i\vert>m_2}\frac{64}{\vert s-\bar{t}_i\vert^2}\,\mathrm{d}s\\
&\le C\frac{\vert w\vert^4\tilde{L}(w)^2}{\tilde{l}^2}m_1+\frac{C}{\tilde{L}(w)^2m_1}+C\frac{\vert w\vert^4\tilde{L}(W(t))^2}{\tilde{l}^2}m_2+\frac{C}{\tilde{L}(W(t))^2m_2}.
\end{align*}
Optimizing $m_1,m_2$ we obtain \eqref{eq-J1-S3-separation}. The other cases can be proved in the same way.

\noindent{\bf Proof of the claim.}  We further divide into three cases.

{\bf Case 1.} If both trajectories have at most one bounce in $[t-\tilde{\delta},t]$. W.L.G., assume that $(X,V)$ has exactly one bounce at $s_1\in(t-\tilde{\delta},t)$, $(Y,W)$ has exactly one bounce at $s_2\in(t-\tilde{\delta},t)$ with $t-\tilde{\delta}<s_1<s_2<t$. For convenience, we denote $s_0=t-\tilde{\delta},s_3=t$. As usual, we analyse the quantity $D\in C^2([s_i,s_{i+1}])$ defined as $D(s)=X(s)-Y(s)$ on $(s_i,s_{i+1})$ for $i=0,1,2$ and the values of $D,\dot{D},\ddot{D}$ at endpoints $s_i,s_{i+1}$ are defined by the right and left limits respectively. Pick $\bar{s}_i\in[s_i,s_{i+1}]$ such that $\vert D(\bar{s}_i)\vert$ is minimal in $[s_i,s_{i+1}]$. Set
\begin{equation*}
\bar{D}(s)=D(\bar{s}_i)+(s-\bar{s}_i)\dot{D}(\bar{s}_i).
\end{equation*}
Notice
\begin{equation*}
(s-\bar{s}_i)D(\bar{s}_i)\cdot\dot{D}(\bar{s}_i)\ge 0,\quad\forall s\in[s_i,s_{i+1}],
\end{equation*}
we have
\begin{equation}\label{eq-lem-J1-S3-separation-claim-1}
\vert \bar{D}(s)\vert^2\ge(s-\bar{s}_i)^2\vert\dot{D}(\bar{s}_i)\vert^2=(s-\bar{s}_i)^2\vert V(\bar{s}_i)-W(\bar{s}_i)\vert^2,\quad\forall s\in[s_i,s_{i+1}],
\end{equation}
where $V(\bar{s}_i),W(\bar{s}_i)$ are defined by the right or left limits if $\bar{s}_i=s_i\text{ or }s_{i+1}$. Hence, by Lemma~\ref{lem-boundary-stability-velocity} and Lemma~\ref{lem-J1-S2-S3-velocity-stability}, we have
\begin{equation*}
\vert V(\bar{s}_i)-W(\bar{s}_i)\vert\ge\tilde{\mathcal{L}}_0^{-1}(s)-CQ_{t,\tilde{\delta}}^{4/3}\tilde{\delta}\ge\tilde{\mathcal{L}}_0^{-1}(s)/2.
\end{equation*}

Now it follows from Taylor's theorem that
\begin{equation*}
\vert\ddot{D}(s)-\ddot{\bar{D}}(s)\vert=\vert E_{\mathrm{N}}^{{\rm pla}}(s,X(s))-E_{\mathrm{N}}^{{\rm pla}}(s,Y(s))\vert
\end{equation*}
so by Lemma~\ref{lem-pre-estimates-electric-field} for all $s\in[s_i,s_{i+1}]$
\begin{equation}\label{eq-lem-J1-S3-separation-claim-2}
\vert D(s)-\bar{D}(s)\vert\le\left\vert\int_{\bar{s}_i}^{s}\int_{\bar{s}_i}^{\tau}\vert E_{\mathrm{N}}^{{\rm pla}}(\zeta,X(\zeta))\vert+\vert E_{\mathrm{N}}^{{\rm pla}}(\zeta,Y(\zeta))\vert\,\mathrm{d}\zeta\,\mathrm{d}\tau\right\vert\le CQ_{t,\tilde{\delta}}^{4/3}\tilde{\delta}\vert s-\bar{s}_i\vert.
\end{equation}
Therefore by \eqref{eq-lem-J1-S3-separation-claim-1} and \eqref{eq-lem-J1-S3-separation-claim-2} we have
\begin{equation*}
\vert D(s)\vert\ge\vert\bar{D}(s)\vert-\vert D(s)-\bar{D}(s)\vert\ge\frac{\tilde{\mathcal{L}}_0^{-1}(s)}{4}\vert s-\bar{s}_i\vert.
\end{equation*}
The claim holds for $\bar{t}_i=\bar{s}_{i-1}$, $i=1,2,3$.

{\bf Case 2.} Now if at least one of the trajectories has more than one bounce in $(t-\tilde{\delta},t)$. W.L.G., let $(X,V)$ has more than one bounce in $(t-\tilde{\delta},t)$. Since $\vert X(s)-Y(s)\vert<\delta_2/16$ in the set $\tilde{S}_3$, by Lemma~\ref{lem-Separation-singular-sets}, $X(s),Y(s)\in\partial\Omega+B_{\delta_0}^3$. Hence $X(s),Y(s)$ can be expressed by the local coordinate:
\begin{align*}
X(s)=X_{\parallel}(s)-X_{\perp}(s)n(s),\quad V(s)=V_{\parallel}(s)-V_{\perp}(s)n(s),\quad V_{\parallel}(s)=\sum_{i=1,2}w_i(s)u_i(s),\\
Y(s)=Y_{\parallel}(s)-Y_{\perp}(s)\tilde{n}(s),\quad W(s)=W_{\parallel}(s)-W_{\perp}(s)\tilde{n}(s),\quad W_{\parallel}(s)=\sum_{i=1,2}\tilde{w}_i(s)\tilde{u}_i(s).
\end{align*}
We first consider the case
\begin{equation}\label{eq-lem-J1-S3-separation-claim-3}
\vert W_{\parallel}(s)-V_{\parallel}(s)\vert\ge \frac{1}{2}\vert W(s)-V(s)\vert\;\text{for all}\;s\in[t-\tilde{\delta},t].
\end{equation}
Note that the tangential parts of the trajectories are $C^1([t-\tilde{\delta},t])$, we define $D_{\parallel}(s)=X_{\parallel}(s)-Y_{\parallel}(s)$ on $[t-\tilde{\delta},t]$. Pick $\underline{s}_1\in[t-\tilde{\delta},t]$ such that $\vert D_{\parallel}(\underline{s}_1)\vert=\min_{s\in[t-\tilde{\delta},t]}\vert D_{\parallel}(s)\vert$. Consider the equations \eqref{eq-character-ODE-new-coordinates} for the tangential components of the position and velocity:
\begin{align*}
\dot{D}_{\parallel}(s)&=\sum_{i=1,2}\left(\frac{w_iu_i(s)}{1+k_ix_{\perp}(s)}-\frac{\tilde{w}_i\tilde{u}_i(s)}{1+\tilde{k}_i\tilde{x}_{\perp}(s)}\right)\\
&=V_{\parallel}(s)-W_{\parallel}(s)+\sum_{i=1,2}\left(\frac{\tilde{w}_i\tilde{u}_i\tilde{k}_i\tilde{x}_{\perp}(s)}{1+\tilde{k}_i\tilde{x}_{\perp}(s)}-\frac{w_iu_ik_ix_{\perp}(s)}{1+k_ix_{\perp}(s)}\right)\\
&=V_{\parallel}(\underline{s}_1)-W_{\parallel}(\underline{s}_1)+\mathcal{O}(Q_{t,\tilde{\delta}}^2\tilde{\delta}),
\end{align*}
where we have used the fact $\vert x_{\perp}\vert\le Q_{t,\tilde{\delta}}\tilde{\delta}$. Hence we have
\begin{equation*}
D_{\parallel}(s)=D_{\parallel}(\underline{s}_1)+\left(V_{\parallel}(\underline{s}_1)-W_{\parallel}(\underline{s}_1)+\mathcal{O}(Q_{t,\tilde{\delta}}^2\tilde{\delta})\right)(s-\underline{s}_1).
\end{equation*}
Notice
\begin{equation*}
(s-\underline{s}_1)D_{\parallel}(\underline{s}_1)\cdot\dot{D}_{\parallel}(\underline{s}_1)\ge 0,\quad\forall s\in[t-\tilde{\delta},t].
\end{equation*}
Thus we have
\begin{equation}\label{eq-lem-J1-S3-separation-claim-4}
\vert D_{\parallel}(s)\vert^2\ge(s-\underline{s}_1)^2\vert V_{\parallel}(\underline{s}_1)-W_{\parallel}(\underline{s}_1)+\mathcal{O}(Q_{t,\tilde{\delta}}^2\tilde{\delta})\vert^2,\quad\forall s\in[t-\tilde{\delta},t].
\end{equation}
Note we can prove that 
\begin{equation*}
\vert V_{\parallel}(s)-W_{\parallel}(s)\vert\le 2\vert V_{\parallel}(\tau)-W_{\parallel}(\tau)\vert\;\text{for all}\;s,\tau\in[t-\tilde{\delta},t].
\end{equation*}
By Lemmas~\ref{lem-boundary-stability-velocity} and \ref{lem-J1-S2-S3-velocity-stability}, we have
\begin{equation*}
\vert V_{\parallel}(\underline{s}_1)-W_{\parallel}(\underline{s}_1)\vert\ge\tilde{\mathcal{L}}_0^{-1}(s)-CQ_{t,\tilde{\delta}}^{4/3}\tilde{\delta}\ge\tilde{\mathcal{L}}_0^{-1}(s)/2.
\end{equation*}
Therefore by \eqref{eq-lem-J1-S3-separation-claim-4} and \eqref{eq-lem-J1-S3-separation-claim-3} we have for all $s\in[t-\tilde{\delta},t]$
\begin{equation*}
\vert X(s)-Y(s)\vert\ge\vert D_{\parallel}(s)\vert\ge\frac{\tilde{\mathcal{L}}_0^{-1}(s)}{4}\vert s-\underline{s}_1\vert.
\end{equation*}
The claim holds for $\bar{t}_i=\underline{s}_{1}$, $i=1,2,3$.

{\bf Case 3.} We now consider the complementary case to \eqref{eq-lem-J1-S3-separation-claim-3}. Then there exists $\bar{s}\in[t-\tilde{\delta},t]$ such that
\begin{equation}\label{eq-lem-J1-S3-separation-claim-5}
\vert W_{\perp}(\bar{s})\tilde{n}(\bar{s})-V_{\perp}(\bar{s})n(\bar{s})\vert\ge \frac{1}{2}\vert W(\bar{s})-V(\bar{s})\vert.
\end{equation}
By \eqref{eq-estimate-derivative-v-perp} we have
\begin{align*}
\left\vert \vert W_{\perp}(s)\vert-\vert V_{\perp}(s)\vert\right\vert\ge\left\vert \vert W_{\perp}(\bar{s})\vert-\vert V_{\perp}(\bar{s})\vert\right\vert-CQ_{t,\tilde{\delta}}^2\tilde{\delta}.
\end{align*}
On the other hand, using Lemma~\ref{lem-estimate-v-perp} and the triangle inequality it follows that
\begin{align*}
\left\vert \vert W_{\perp}(s)\vert-\vert V_{\perp}(s)\vert\right\vert&=\left\vert \vert W_{\perp}(s)\tilde{n}(s)-V_{\perp}(s)n(s)+V_{\perp}(s)n(s)\vert-\vert V_{\perp}(s)\vert\right\vert.\\
&\le\vert W_{\perp}(s)\tilde{n}(s)-V_{\perp}(s)n(s)\vert+2\vert V_{\perp}(s)\vert\\
&\le\vert W_{\perp}(s)-V_{\perp}(s)\vert+CQ_{t,\tilde{\delta}}^2\tilde{\delta}.
\end{align*}
Similarly we get
\begin{align*}
\vert \vert W_{\perp}(\bar{s})\vert-\vert V_{\perp}(\bar{s})\vert\vert\ge\vert W_{\perp}(\bar{s})\tilde{n}(\bar{s})-V_{\perp}(\bar{s})n(\bar{s})\vert-CQ_{t,\tilde{\delta}}^2\tilde{\delta}.
\end{align*}
Thus by \eqref{eq-lem-J1-S3-separation-claim-5} we obtain for all $s\in[t-\tilde{\delta},t]$
\begin{align}
\vert W_{\perp}(s)-V_{\perp}(s)\vert&\ge\vert W_{\perp}(\bar{s})\tilde{n}(\bar{s})-V_{\perp}(\bar{s})n(\bar{s})\vert-CQ_{t,\tilde{\delta}}^2\tilde{\delta},\nonumber\\
&\ge\frac{1}{2}\vert W(\bar{s})-V(\bar{s})\vert-CQ_{t,\tilde{\delta}}^2\tilde{\delta}.\label{eq-lem-J1-S3-separation-claim-6}
\end{align}

Using Lemmas~\ref{lem-estimate-v-perp} and \ref{lem-J1-S2-S3-velocity-stability} it follows that:
\begin{equation}\label{eq-lem-J1-S3-separation-claim-7}
\vert W_{\perp}(s)\vert\ge\frac{\tilde{R}}{8}
\end{equation}
for all $s\in[t-\tilde{\delta},t]$. It follows that $W_{\perp}(s)$ changes sign, by reflection, at most once in the interval $s\in[t-\tilde{\delta},t]$. Combining Lemma~\ref{lem-estimate-v-perp}, \eqref{eq-lem-J1-S3-separation-claim-6} and \eqref{eq-lem-J1-S3-separation-claim-7} it follows that $W_{\perp}(s)-V_{\perp}(s)$ changes sign at most once for $s\in[t-\tilde{\delta},t]$. Suppose that $W_{\perp}(s)$ changes sign (if any) at $s=s_1$, i.e., $(Y,W)$ bounces at $s=s_1$. Since $Y_{\perp}(s)\ge 0$, it follows that $\operatorname{sgn}(W_{\perp}(s))=\operatorname{sgn}(W_{\perp}(s)-V_{\perp}(s))=\operatorname{sgn}(s-s_1)$ for $s\in[t-\tilde{\delta},t]$. Pick $s_0\in[t-\tilde{\delta},t]$ such that
\begin{equation*}
\min_{s\in[t-\tilde{\delta},t]}\vert Y_{\perp}(s)-X_{\perp}(s)\vert=\vert Y_{\perp}(s_0)-X_{\perp}(s_0)\vert.
\end{equation*}
If $s_0\in(t-\tilde{\delta},t)$, then by \eqref{eq-lem-J1-S3-separation-claim-6} and  Lemma~\ref{lem-min-dichotomy}, it must be $s_0=s_1$ and for all $s\in[t-\tilde{\delta},t]$
\begin{equation*}
\vert Y_{\perp}(s)-X_{\perp}(s)\vert =\int_{s_1}^sW_{\perp}(\tau)-V_{\perp}(\tau)\,\mathrm{d}\tau\ge\frac{1}{4}\vert W(\bar{s})-V(\bar{s})\vert\vert s-s_1\vert.
\end{equation*}
If $s_0\notin(t-\tilde{\delta},t)$, we discuss it in four situations:

{\bf a).} If $s_0=t-\tilde{\delta}$, $Y_{\perp}(t-\tilde{\delta})-X_{\perp}(t-\tilde{\delta})\ge 0$, it is obvious that $W_{\perp}(t-\tilde{\delta}+0)<0$ is impossible. Then by $W_{\perp}(t-\tilde{\delta}+0)>0$, we always have $W_{\perp}(s)-V_{\perp}(s)>0$ and for all $s\in[t-\tilde{\delta},t]$
\begin{equation*}
\vert Y_{\perp}(s)-X_{\perp}(s)\vert =Y_{\perp}(t_1)-X_{\perp}(t-\tilde{\delta})+\int_{t-\tilde{\delta}}^sW_{\perp}(\tau)-V_{\perp}(\tau)\,\mathrm{d}\tau\ge\frac{1}{4}\vert W(\bar{s})-V(\bar{s})\vert\vert s-t+\tilde{\delta}\vert.
\end{equation*}

{\bf b).} If $s_0=t-\tilde{\delta}$, $Y_{\perp}(t-\tilde{\delta})-X_{\perp}(t-\tilde{\delta})<0$, since $(Y,W)$ bounces at most once, $(X,V)$ bounce at least twice, it must exists $s_1\in[t-\tilde{\delta},t]$ such that $Y_{\perp}(s_1)-X_{\perp}(s_1)=0$, which is a contradiction. 

{\bf c).} If $s_0=t$, $W_{\perp}(t-0)<0$, we always have $W_{\perp}(s)-V_{\perp}(s)<0$ and $Y_{\perp}(s)-X_{\perp}(s)\ge 0$. Then for all $s\in[t-\tilde{\delta},t]$
\begin{equation*}
\vert Y_{\perp}(s)-X_{\perp}(s)\vert =Y_{\perp}(t)-X_{\perp}(t)+\int_{t}^sW_{\perp}(\tau)-V_{\perp}(\tau)\,\mathrm{d}\tau\ge\frac{1}{4}\vert W(\bar{s})-V(\bar{s})\vert\vert s-t\vert.
\end{equation*}

{\bf d).} If $s_0=t$, $W_{\perp}(t-0)>0$, it is obvious that $Y_{\perp}(t)-X_{\perp}(t)>0$ is impossible. Since $(Y,W)$ bounces at most once, $(X,V)$ bounce at least twice, it must exists $s_2\in(t-\tilde{\delta},t)$ such that $Y_{\perp}(s_2)-X_{\perp}(s_2)=0$, which is a contradiction.

For all situations, we have
\begin{equation*}
\vert Y_{\perp}(s)-X_{\perp}(s)\vert\ge\frac{1}{4}\vert W(\bar{s})-V(\bar{s})\vert\vert s-s_0\vert\quad\forall s\in[t-\tilde{\delta},t].
\end{equation*}
Now assume $\bar{s}\in[t-\tilde{\delta},s_1]$ and there exists also $\tilde{s}\in[s_1,s]$ such that
\begin{equation*}
\vert W_{\perp}(\tilde{s})\tilde{n}(\tilde{s})-V_{\perp}(\tilde{s})n(\tilde{s})\vert\ge \frac{1}{2}\vert W(\tilde{s})-V(\tilde{s})\vert.
\end{equation*}
Then we have by Lemma~\ref{lem-boundary-stability-velocity}
\begin{align*}
\vert Y_{\perp}(s)-X_{\perp}(s)\vert&\ge\frac{1}{4}\max\{\vert W(\bar{s})-V(\bar{s})\vert,\vert W(\tilde{s})-V(\tilde{s})\vert\}\vert s-s_0\vert\\
&\ge\frac{1}{8}\tilde{\mathcal{L}}_0^{-1}(s)\vert s-s_0\vert.
\end{align*}
If $\tilde{s}$ does not exist, we apply the analysis of Case 2 on $[s_1,t]$ and the claim follows for $\bar{t}_i=s_0$, $i=1,2,3$.
\end{proof}

According to the division, we have
\begin{align*}
\tilde{\mathcal{J}}_1=\sum_{i=1}^3\tilde{\mathcal{J}}_{1,i}=\sum_{i=1}^3\int_{\tilde{S}_i}\frac{f(t-\tilde{\delta},y,w)}{\vert X(s)-Y(s)\vert^2}\,\mathrm{d}y\,\mathrm{d}w\,\mathrm{d}s.
\end{align*}

For $\tilde{\mathcal{J}}_{1,1}$, by the measure preserving property and Lemma~\ref{lem-pre-estimates-electric-field}, we have
\begin{equation*}
\tilde{\mathcal{J}}_{1,1}\le \int_{t-\tilde{\delta}}^t\int_{\min\{\vert w\vert,\tilde{L}(w)\}\le \tilde{R}}\frac{f(s,y,w)}{\vert X(s)-y\vert^2}\,\mathrm{d}y\,\mathrm{d}w\,\mathrm{d}s\le C\tilde{R}^{4/3}\tilde{\delta}.
\end{equation*}

For $\tilde{\mathcal{J}}_{1,2}$, by the measure preserving property and Lemma~\ref{lem-J1-S2-S3-velocity-stability}, we have
\begin{align*}
\tilde{\mathcal{J}}_{1,2}&\le \int_{t-\tilde{\delta}}^t\int_{\min\{\vert w\vert,\tilde{L}(w)\}>\tilde{R},\vert X(s)-y\vert\le\frac{4\tilde{l}}{\vert w\vert^2\tilde{L}(w)}}\frac{f(s,y,w)}{\vert X(s)-y\vert^2}\,\mathrm{d}y\,\mathrm{d}w\,\mathrm{d}s\\
&\le C\int_{t-\tilde{\delta}}^t\int_{\tilde{R}\le\min\{\vert w\vert,\tilde{L}(w)\}\le 2Q_{t,\tilde{\delta}}}\frac{4\tilde{l}}{\vert w\vert^2\tilde{L}(w)}\,\mathrm{d}w\,\mathrm{d}s\le C\tilde{l}\ln Q_{t,\tilde{\delta}}\tilde{\delta}.
\end{align*}

For $\tilde{\mathcal{J}}_{1,3}$, we have by Lemma~\ref{lem-J1-S3-separation}
\begin{align*}
\tilde{\mathcal{J}}_{1,3}&\le\iint_{\Gamma}f(t-\tilde{\delta},y,w)\int_{t-\tilde{\delta}}^{t}\frac{\mathds{1}_{\{\vert X(s)-Y(s)\vert>\tilde{\mathcal{L}}\}}}{\vert X(s)-Y(s)\vert^2}\,\mathrm{d}s\,\mathrm{d}y\,\mathrm{d}w\\
&\le C\tilde{l}^{-1}\iint_{\Gamma}\vert w\vert^2f(t-\tilde{\delta},w)\,\mathrm{d}y\,\mathrm{d}w\le C\tilde{l}^{-1}.
\end{align*}

Hence we have
\begin{equation*}
\tilde{\mathcal{J}}_1(t)\le C\tilde{R}^{4/3}\tilde{\delta}+C\tilde{l}\ln Q_{t,\tilde{\delta}}\tilde{\delta}+C\tilde{l}^{-1}\le CQ_{t,\tilde{\delta}}\tilde{\delta}.
\end{equation*}
Note
\begin{equation*}
\mathcal{J}_1=\sum_{i=1}^{\lfloor\frac{\delta}{\tilde{\delta}}\rfloor}\tilde{\mathcal{J}}_1(t-\delta+i\tilde{\delta})+\int_{t-\delta+\lfloor\frac{\delta}{\tilde{\delta}}\rfloor\tilde{\delta}}^t\int_{\bar{A}_1}\frac{f(t-\delta,y,w)}{\vert X(s)-Y(s)\vert^2}\,\mathrm{d}y\,\mathrm{d}w\,\mathrm{d}s.
\end{equation*}
In summary, we have
\begin{equation}\label{eq-estimate-J1}
\mathcal{J}_1\le C\left(\left\lfloor\frac{\delta}{\tilde{\delta}}\right\rfloor+1\right)Q_{T,T}\tilde{\delta}\le CQ_{T,T}\delta.
\end{equation}

\subsection{The estimate on $\mathcal{J}_4$}

We define three parameters $l,R,L$ by setting
\begin{equation}\label{eq-J4-parameters-setting}
l=Q_{t,\delta}^{-\frac{7}{8}},\,R=Q_{t,\delta}^{\frac{3}{4}},\,L=Q_{t,\delta}^{-2}.
\end{equation}
W.L.G., we assume $Q_{t,\delta}\ge K_2>1$ where $K_2$ is a large constant depending only on $\|\mathring{f}\|_{1}$, $\|\mathring{f}\|_{\infty}$, $\mathcal{E}(0)$. We have
\begin{equation}\label{eq-J4-elementary-estimate-1}
(l+1)Q_{t,\delta}^{4/3}\le 2K_2^{-\frac{1}{6}}R^2.
\end{equation}

When the pointwise energy $h_{\alpha_0}(t-\delta,y,w)$ is large, we can improve Lemma~\ref{lem-Separation-singular-sets} in the following lemma.
\begin{lemma}\label{lem-measJ-plasma-near-point charge}
Let $(Y,W)$ be the characteristics. Denote
\begin{equation*}
J=\{s\in(t-\delta,t):\vert Y(s)-\xi_{\alpha_0}(s)\vert<a\},
\end{equation*}
with $a\le\frac{\delta_2}{2}$. Suppose $\sqrt{h_{\alpha_0}}(t-\delta,y,w)>R$, then $J$ is connected and
\begin{equation*}
\operatorname{meas}(J)\le CR^{-1}a.
\end{equation*}
Suppose $\sqrt{h_{\alpha_0}}(t-\delta,y,w)>Q_{t,\delta}/2$, then $J$ is connected and
\begin{equation*}
\operatorname{meas}(J)\le CQ_{t,\delta}^{-1}a.
\end{equation*}
\end{lemma}

\begin{proof}

We first assume $\sqrt{h_{\alpha_0}}(t-\delta,y,w)>R$. Denote $I(s)=\vert Y(s)-\xi_{\alpha_0}(s)\vert^2$. Assume $J$ is non-empty, then by Lemma~\ref{lem-Separation-singular-sets} we have $\sqrt{I(s)}\le\frac{17\delta_2}{32}$ for all $s\in(t-\delta,t)$. Differentiating $I(s)$ we get 
\begin{align*}
&\dot{I}=2(Y-\xi_{\alpha_0})\cdot (W-\eta_{\alpha_0}),\\
&\ddot{I}=2\vert W-\eta_{\alpha_0}\vert^2+2(Y-\xi_{\alpha_0})\cdot\left[E_{\mathrm{N}}^{\rm pla}(s,Y)-E_{\mathrm{N},\alpha_0}^{\rm cha}\right].
\end{align*}

By \eqref{eq-h-alpha-3}, we have
\begin{equation*}
\sqrt{h_{\alpha_0}}(s,Y(s),W(s))\ge\frac{R}{2}.
\end{equation*}
Then by \eqref{eq-J4-elementary-estimate-1} and take $K_2$ large enough such that we have
\begin{align*}
\ddot{I}&\ge h_{\alpha_0}(s,Y(s),W(s))-K_1-\vert Y-\xi_{\alpha_0}\vert\left(\vert E_{\mathrm{N}}(s,Y)\vert+\vert E_{\mathrm{N}}(s,\xi_{\alpha_0})\vert+C\right)\\
&\ge\frac{1}{4}R^2-K_1-C\delta_2Q_{t,\delta}^{4/3}\ge\frac{1}{8}R^{2}.
\end{align*}
Hence $I(s)$ is convex on $(t-\delta,t)$ and the sublevel set $J$ is a convex subset of $(t-\delta,t)$, whence $J$ is connected. Now let $t_0\in\bar{J}$ be the minimizer for $I(s)$, then $\dot{I}(t_0)(s-t_0)\ge 0$ and we get for all $s\in J$
\begin{equation*}
a^2\ge I(s)\ge I(t_0)+\dot{I}(t_0)(s-t_0)+\frac{1}{16}R^2(s-t_0)^2\ge\frac{1}{16}R^2(s-t_0)^2,
\end{equation*}
which implies $\operatorname{meas}(J)\le CR^{-1}a$. The case $\sqrt{h_{\alpha_0}}(t-\delta,y,w)>Q_{t,\delta}/2$ can be proved similarly.

\end{proof}

The following lemma will play a key role as \cite[Lemma 3]{MMP11}.
\begin{lemma}\label{lem-ugly-set-plasma-far-from-point charge}
Let $L>0$. Assume that there exists a time interval
\begin{equation*}
J=(t_1,t_2)\subset(t-\delta,t)
\end{equation*}
such that for all $s\in J$ we have
\begin{equation}\label{eq-ugly-set-plasma-far-from-point charge-1}
\inf_{s\in J}\min\left\{\vert X(s)-\xi_{\alpha_0}(s)\vert,\vert Y(s)-\xi_{\alpha_0}(s)\vert\right\}>l.
\end{equation}
If $\vert V(t_1)-W(t_1)\vert\}\le R$, then
\begin{equation}\label{eq-ugly-set-plasma-far-from-point charge-2}
\vert V(s)-W(s)\vert\le\frac{3}{2}R,\quad\forall s\in J.
\end{equation}
If $\vert V(t_1)-W(t_1)\vert\}>R$, then
\begin{equation}\label{eq-ugly-set-plasma-far-from-point charge-3}
\vert V(s)-W(s)\vert>\frac{1}{2}R,\quad\forall s\in J
\end{equation}
and
\begin{equation}\label{eq-ugly-set-plasma-far-from-point charge-4}
\int_{t_1}^{t_2}\frac{\mathds{1}_{\{\vert X(s)-Y(s)\vert>L\}}}{\vert X(s)-Y(s)\vert^2}\,\mathrm{d}s\le\frac{C}{RL}.
\end{equation}
\end{lemma}

\begin{proof}

Since there is no bounce, we have
\begin{align*}
\left\vert\vert V(s)-W(s)\vert-\vert V(t_1)-W(t_1)\vert\right\vert\le\int_{t_1}^{s}\vert E_{\mathrm{N}}^{{\rm pla}}(\tau,X(\tau))\vert+\vert E_{\mathrm{N}}^{{\rm pla}}(\tau,Y(\tau))\vert\,\mathrm{d}\tau.
\end{align*}
Notice by \eqref{eq-pre-estimates-electric-field}, \eqref{eq-ugly-set-plasma-far-from-point charge-1} and take $c_0$ small enough, we have
\begin{equation}\label{eq-ugly-set-plasma-far-from-point charge-5}
\int_{t_1}^{s}\vert E_{\mathrm{N}}^{{\rm pla}}(\tau,X(\tau))\vert+\vert E_{\mathrm{N}}^{{\rm pla}}(\tau,Y(\tau))\vert\,\mathrm{d}\tau\le CQ_{t,\delta}^{\frac{4}{3}}\delta+2l^{-2}\delta\le\frac{R}{2}.
\end{equation}
\eqref{eq-ugly-set-plasma-far-from-point charge-2}-\eqref{eq-ugly-set-plasma-far-from-point charge-3} follow.

Assume $\vert V(t_1)-W(t_1)\vert\}>R$. As usual, denote $D(s)=X(s)-Y(s)$. Pick $\bar{s}$ such that $\vert D(\bar{s})\vert$ is minimal in $[t_1,t_2]$. Set
\begin{equation*}
\bar{D}(s)=D(\bar{s})+(s-\bar{s})\dot{D}(\bar{s}).
\end{equation*}
Notice
\begin{equation*}
(s-\bar{s})D(\bar{s})\cdot\dot{D}(\bar{s})\ge 0,\quad\forall s\in[t_1,t_2],
\end{equation*}
we have by \eqref{eq-ugly-set-plasma-far-from-point charge-3}
\begin{equation}\label{eq-ugly-set-plasma-far-from-point charge-6}
\vert \bar{D}(s)\vert^2\ge(s-\bar{s})^2\vert\dot{D}(\bar{s})\vert^2=(s-\bar{s})^2\vert V(\bar{s})-W(\bar{s})\vert^2\ge\frac{R^2}{4}(s-\bar{s})^2,\quad\forall s\in[t_1,t_2].
\end{equation}
Now it follows from Taylor's theorem that
\begin{equation*}
\vert\ddot{D}(s)-\ddot{\bar{D}}(s)\vert=\vert E_{\mathrm{N}}^{{\rm pla}}(s,X(s))-E_{\mathrm{N}}^{{\rm pla}}(s,Y(s))\vert
\end{equation*}
so that similarly to \eqref{eq-ugly-set-plasma-far-from-point charge-5} for all $s\in[t_1,t_2]$
\begin{equation}\label{eq-ugly-set-plasma-far-from-point charge-7}
\vert D(s)-\bar{D}(s)\vert\le\left\vert\int_{\bar{s}}^{s}\int_{\bar{s}}^{\tau}\vert E_{\mathrm{N}}^{{\rm pla}}(\zeta,X(\zeta))\vert+\vert E_{\mathrm{N}}^{{\rm pla}}(\zeta,Y(\zeta))\vert\,\mathrm{d}\zeta\,\mathrm{d}\tau\right\vert\le\frac{R}{4}\vert s-\bar{s}\vert.
\end{equation}
Therefore by \eqref{eq-ugly-set-plasma-far-from-point charge-6} and \eqref{eq-ugly-set-plasma-far-from-point charge-7} we have
\begin{equation*}
\vert D(s)\vert\ge\vert\bar{D}(s)\vert-\vert D(s)-\bar{D}(s)\vert\ge\frac{R}{4}\vert s-\bar{s}\vert.
\end{equation*}
Hence
\begin{align*}
\int_{t_1}^{t_2}\frac{\mathds{1}_{\{\vert X(s)-Y(s)\vert>L\}}}{\vert X(s)-Y(s)\vert^2}\,\mathrm{d}s&\le\int_{\vert s-\bar{s}\vert\le m}+\int_{\vert s-\bar{s}\vert>m}\frac{\mathds{1}_{\{\vert X(s)-Y(s)\vert>L\}}}{\vert X(s)-Y(s)\vert^2}\,\mathrm{d}s\\
&\le\int_{\vert s-\bar{s}\vert\le m}L^{-2}\,\mathrm{d}s+\int_{\vert s-\bar{s}\vert>m}\frac{64}{R^2\vert s-\bar{s}\vert^2}\,\mathrm{d}s\\
&\le 6L^{-2}m+\frac{6\times 64}{R^2m}.
\end{align*}
Optimizing $m$ we obtain \eqref{eq-ugly-set-plasma-far-from-point charge-4}.

\end{proof}

Recall \eqref{eq-lower-bound-h-alpha0-t-delta}. By virtue of Lemma~\ref{lem-measJ-plasma-near-point charge}, let $(t_-,t_+)$ be the connected interval in which $\vert X(s)-\xi_{\alpha_0}(s)\vert<2l$, combining with \eqref{eq-pre-estimates-electric-field} we obtain
\begin{equation}\label{eq-estimate-J4-X-near-xi}
\int_{t_{-}}^{t_{+}}\int_{\bar{A}_4}\frac{f(t-\delta,y,w)}{\vert X(s)-Y(s)\vert^2}\,\mathrm{d}y\,\mathrm{d}w\,\mathrm{d}s\le CQ_{t,\delta}^{\frac{1}{3}}l.
\end{equation}

To estimate $\mathcal{J}_4$, it remains to control the integrals
\begin{equation*}
\int_{t-\delta}^{t_{-}}\int_{\bar{A}_4}\frac{f(t-\delta,y,w)}{\vert X(s)-Y(s)\vert^2}\,\mathrm{d}y\,\mathrm{d}w\,\mathrm{d}s\text{ and }\int_{t_{+}}^{t}\int_{\bar{A}_4}\frac{f(t-\delta,y,w)}{\vert X(s)-Y(s)\vert^2}\,\mathrm{d}y\,\mathrm{d}w\,\mathrm{d}s
\end{equation*}
which can be handled in the same way. Denote $A_4^{-}:=(t-\delta,t_-)\times\bar{A}_4$ we split the integral into four parts:
\begin{align*}
\int_{A_4^{-}}\frac{f(t-\delta,y,w)}{\vert X(s)-Y(s)\vert^2}\,\mathrm{d}y\,\mathrm{d}w\,\mathrm{d}s=\sum_{j=1}^4\int_{S_j}\frac{f(t-\delta,y,w)}{\vert X(s)-Y(s)\vert^2}\,\mathrm{d}y\,\mathrm{d}w\,\mathrm{d}s=\sum_{j=1}^4\mathcal{J}_{4,j},
\end{align*}
where
\begin{align*}
S_1&=\left\{(s,y,w)\in A_4^{-}:\sqrt{h_{\alpha_0}}(t-\delta,y,w)\le R\right\},\\
S_2&=\left\{(s,y,w)\in A_4^{-}:\vert X(s)-Y(s)\vert\le L\right\}\setminus S_1,\\
S_3&=\left\{(s,y,w)\in A_4^{-}:\vert Y(s)-\xi_{\alpha_0}(s)\vert\le l\right\}\setminus(S_1\cup S_2),\\
S_4&=A_4^{-}\setminus\cup_{i=1}^3S_i.
\end{align*}

For $\mathcal{J}_{4,1}$, using the measure-preserving property, \eqref{eq-h-alpha-2} and \eqref{eq-pre-estimates-electric-field} we have
\begin{equation}\label{eq-estimate-J41}
\mathcal{J}_{4,1}\le\int_{t-\delta}^{t_-}\int_{\sqrt{h_{\alpha_0}}(t-\delta,y,w)\le 2R}\frac{f(s,y,w)}{\vert X(s)-y\vert^2}\,\mathrm{d}y\,\mathrm{d}w\,\mathrm{d}s\le CR^{4/3}\delta.
\end{equation}

For $\mathcal{J}_{4,2}$, using the measure-preserving property, we have
\begin{equation}\label{eq-estimate-J42}
\mathcal{J}_{4,2}\le \int_{t-\delta}^{t_-}\int_{\vert X(s)-y\vert\le L}\frac{C\|\mathring{f}\|_{\infty}Q_{t,\delta}^3}{\vert X(s)-y\vert^2}\,\mathrm{d}y\,\mathrm{d}s\le CLQ_{t,\delta}^3\delta.
\end{equation}
In $S_3$, we have $\vert X(s)-Y(s)\vert\ge\vert X(s)-\xi_{\alpha_0}(s)\vert-\vert Y(s)-\xi_{\alpha_0}(s)\vert\ge l$. By Lemma~\ref{lem-measJ-plasma-near-point charge} and Proposition~\ref{prn-conservation-laws}, we have
\begin{align}
\mathcal{J}_{4,3}&\le l^{-2}\int_{\sqrt{h_{\alpha_0}}(t-\delta,y,w)>R}f(t-\delta,y,w)\int_{t-\delta}^{t_-}\mathds{1}_{\{\vert Y(s)-\xi_{\alpha_0}(s)\vert\le l\}}\,\mathrm{d}s\,\mathrm{d}y\,\mathrm{d}w\nonumber\\
&\le l^{-2}R^{-2}\int h_{\alpha_0}f(t-\delta,y,w)\int_{t-\delta}^{t_-}\mathds{1}_{\{\vert Y(s)-\xi_{\alpha_0}(s)\vert\le l\}}\,\mathrm{d}s\,\mathrm{d}y\,\mathrm{d}w\le Cl^{-1}R^{-3}.\label{eq-estimate-J43}
\end{align}

Finally, by virtue of Lemma~\ref{lem-measJ-plasma-near-point charge}, for each $(y,w)$ such that $\sqrt{h_{\alpha_0}}(t-\delta,y,w)>R$, we may split the set
\begin{equation*}
\left\{s\in(t-\delta,t_-):\vert Y(s)-\xi_{\alpha_0}(s)\vert>l\right\}
\end{equation*}
into two at most intervals $J^1(y,w)$, $J^2(y,w)$ such that
\begin{equation*}
\mathcal{J}_{4,4}\le \sum_{k=1,2}\int_{\sqrt{h_{\alpha_0}}(t-\delta,y,w)>R}f(t-\delta,y,w)\int_{J^k(y,w)}\frac{\mathds{1}_{\{\vert Y(s)-X(s)\vert>L\}}}{\vert Y(s)-X(s)\vert^2}\,\mathrm{d}s\,\mathrm{d}y\,\mathrm{d}w=\sum_{k=1,2}J_{4,4}^k.
\end{equation*}
It suffices to control the integral $J_{4,4}^1$ because $J_{4,4}^2$ can be handled in the same way. We set $J^1(y,w)=(t_1,t_2)$ and split further the integration domain as follows:
\begin{align*}
S_4^1=\left\{(y,w): \sqrt{h_{\alpha_0}}(t-\delta,y,w)>R\;\text{and}\;\vert V(t_1)-W(t_1)\vert\le R\right\},\\
S_4^2=\left\{(y,w): \sqrt{h_{\alpha_0}}(t-\delta,y,w)>R\;\text{and}\;\vert V(t_1)-W(t_1)\vert>R\right\}.
\end{align*}
Then
\begin{equation*}
J_{4,4}^1\le\sum_{k=1,2}\int_{S_4^k}f(t-\delta,y,w)\int_{t_1}^{t_2}\frac{\mathds{1}_{\{\vert Y(s)-X(s)\vert>L\}}}{\vert Y(s)-X(s)\vert^2}\,\mathrm{d}s\,\mathrm{d}y\,\mathrm{d}w=\sum_{k=1,2}J_{4,4}^{1,k}.
\end{equation*}

By \eqref{eq-pre-estimates-electric-field}, \eqref{eq-ugly-set-plasma-far-from-point charge-2} and the measure-preserving property, we have
\begin{equation}\label{eq-estimate-J441}
\mathcal{J}_{4,4}^{1,1}\le\int_{t-\delta}^{t_-}\int_{\vert V(s)-w\vert\le \frac{3}{2}R}\frac{f(s,y,w)}{\vert y-X(s)\vert^2}\,\mathrm{d}y\,\mathrm{d}w\,\mathrm{d}s\le CR^{4/3}\delta.
\end{equation}

By \eqref{eq-ugly-set-plasma-far-from-point charge-4} and Proposition~\ref{prn-conservation-laws}, we have
\begin{align}
\mathcal{J}_{4,4}^{1,2}&\le R^{-2}\int_{S_4^2}h_{\alpha_0}f(t-\delta,y,w)\int_{t_1}^{t_2}\frac{\mathds{1}_{\{\vert Y(s)-X(s)\vert>L\}}}{\vert Y(s)-X(s)\vert^2}\,\mathrm{d}s\,\mathrm{d}y\,\mathrm{d}w\nonumber\\
&\le CL^{-1}R^{-3}.\label{eq-estimate-J442}
\end{align}

Hence by \eqref{eq-estimate-J4-X-near-xi}, \eqref{eq-estimate-J41}, \eqref{eq-estimate-J42}, \eqref{eq-estimate-J43}, \eqref{eq-estimate-J441}, \eqref{eq-estimate-J442}, we have
\begin{equation}\label{eq-estimate-J4}
\mathcal{J}_4\le C\left(Q_{t,\delta}^{\frac{1}{3}}l+LQ_{t,\delta}^3\delta+l^{-1}R^{-3}+R^{4/3}\delta+L^{-1}R^{-3}\right).
\end{equation}

\subsection{Estimates on $\bar{\mathcal{I}}_{\alpha_0}$ and $\tilde{\mathcal{I}}$}
In both Case I and II, we have $\min_{\beta:\beta\ne\alpha_0}\inf_{s}\vert X(s)-\xi_{\beta}(s)\vert\ge\frac{\delta_2}{4}$. Hence
\begin{equation}\label{eq-estimate-I-alpha-0}
\bar{\mathcal{I}}_{\alpha_0}\le \frac{16M}{\delta_2^2}\delta.
\end{equation}

Note $\tilde{\mathcal{I}}$ is the 0-th order singular moment. By Lemma~\ref{lem-L0-by-Hk}, we have
\begin{equation}\label{eq-estimate-I}
\tilde{\mathcal{I}}\le C\left(1+H_{\tilde{k}}(t)^{\frac{a}{\tilde{k}-2}}\delta\right)\le C\left(1+Q_{T,T}\delta\right),
\end{equation}
where $C$ depending only on $\|\mathring{f}\|_{1}$, $\|\mathring{f}\|_{\infty}$, $\mathcal{E}(0)$.

\subsection{Proof of Proposition~\ref{prn-uniform-bound-pointwise-energy}}

By Lemma~\ref{lem-estimate-a0-ne-bar-alpha}, we have
\begin{align*}
Q_{t,\delta}=\sqrt{h_{\bar{\alpha}}}(X(\bar{s}),V(\bar{s}))&\le\sqrt{h_{\alpha_0}}(X(\bar{s}),V(\bar{s}))+C\\
&\le\sqrt{h_{\alpha_0}}(\bar{x},\bar{v})+\bar{\mathcal{I}}_{\alpha_0}+\tilde{\mathcal{I}}+\mathcal{J}+C
\end{align*}

We sum up the estimates \eqref{eq-estimate-J2-J3}, \eqref{eq-estimate-J1}, \eqref{eq-estimate-J4} on $\mathcal{J}_i$ and by the setting of the parameters \eqref{eq-setting-delta}, \eqref{eq-J1-parameters-setting}, \eqref{eq-J4-parameters-setting} to obtain
\begin{align*}
\sum_{i=1}^4\mathcal{J}_i&\le C\left(\delta+\tilde{R}^{4/3}\delta+\tilde{l}\ln Q_{t,\delta}\delta+\tilde{l}^{-1}+Q_{t,\delta}^{\frac{1}{3}}l+LQ_{t,\delta}^3\delta+l^{-1}R^{-3}+R^{4/3}\delta+L^{-1}R^{-3}\right)\\
&\le CQ_{t,\delta}\delta.
\end{align*}
Combining with \eqref{eq-estimate-I-alpha-0} and \eqref{eq-estimate-I}, we have
\begin{align*}
Q_{t,\delta}\le Q_{t-\delta,\delta}+CQ_{T,T}\delta.
\end{align*}
Then Proposition~\ref{prn-uniform-bound-pointwise-energy} follows by induction.

\section{The justification of boundary effects}\label{sec-justification}

In this section we aim to verify the boundary effects on the point charges by deducing the point charge model from a modified Vlasov-Poisson system. Through this derivation, the boundary term $H_{\#}$ in \eqref{eq-Newton} arises naturally. 

\subsection{The modified Vlasov-Poisson system}\label{subsec-justification-VP}

The modified Vlasov-Poisson system is written by
\begin{equation}\label{eq-modified-Vlasov-Poisson}
\partial_tf_{\varepsilon}+v\cdot\nabla_xf_{\varepsilon}+\nabla_x\phi_{\#,\varepsilon}^{\varsigma}\cdot\nabla_vf_{\varepsilon}=0,
\end{equation}
where $\phi_{\#,\varepsilon}^{\varsigma}$ is given by
\begin{equation*}
\phi_{\#,\varepsilon}^{\varsigma}(t,x)=\int_{\Omega} G_{\#}^{\varsigma}(x,y)\rho_{\varepsilon}(t,y)\,\mathrm{d}y-\int_{\partial\Omega} G_{\#}(x,y)h_{\#}^{\rm cha}(y)\,\mathrm{d}S_y,
\end{equation*}
and $f_{\varepsilon}$ satisfies the reflection boundary condition
\begin{equation*}
f_{\varepsilon}(t,x,v)=f_{\varepsilon}(t,x,\mathcal{R}_xv),\quad\forall x\in\partial\Omega.
\end{equation*}

$G_{\#}^{\varsigma}$ is a modified Green function, given by
\begin{equation*}
G_{\#}^{\varsigma}=G^{\varsigma}+\bar{g}_{\#},
\end{equation*}
where
\begin{equation*}
G^{\varsigma}(x,y)=G(x,y)\tilde{\chi}_{\varsigma}(\vert x-y\vert).
\end{equation*}

Let $0<\varepsilon<\min_{\alpha}(\min_{\beta:\beta\ne\alpha}\vert\mathring{\xi}_{\alpha}-\mathring{\xi}_{\beta}\vert,\operatorname{d}_{\partial\Omega}(\mathring{\xi}_{\alpha}))/2$ and $\varepsilon\ll 1$. The cut-off parameter $\varsigma$ will be given later, depending on $\varepsilon$ and $\lim_{\varepsilon\to 0}\varsigma=0$.

The initial datum $f_{\varepsilon}\vert_{t=0}=f_{0,\varepsilon}$ satisfies
\begin{equation*}
f_{0,\varepsilon}(x,v)=\sum_{\alpha}f_{0\alpha,\varepsilon}(x),
\end{equation*}
where $f_{0\alpha,\varepsilon}\in C_c^{1,\mu}(\bar{\Omega}\times\mathbb{R}^3)$ for some $\mu>0$ satisfying
\begin{align*}
&\operatorname{supp}f_{0\alpha,\varepsilon}\subset B(\mathring{\xi}_{\alpha},\varepsilon)\times B(\mathring{\eta}_{\alpha},\varepsilon),\\
&0\le f_{0\alpha,\varepsilon}(x,v)\le 2\varepsilon^{-6}\omega_3^2,\quad\forall (x,v)\in\Omega\times\mathbb{R}^3,\\
&\iint_{\Omega\times\mathbb{R}^3}f_{0\alpha,\varepsilon}(x,v)\,\mathrm{d}x\,\mathrm{d}v=1.
\end{align*}
$h_{\mathrm{D}}^{\rm cha}=0$, $h_{\mathrm{N}}^{\rm cha}$ is a $C^3$ given function satisfying
\begin{equation*}
\int_{\partial\Omega}h_{\mathrm{N}}^{\rm cha}(x)\,\mathrm{d}S_x=\iint_{\Omega\times\mathbb{R}^3}\mathring{f}(x,v)\,\mathrm{d}x\,\mathrm{d}v=M.
\end{equation*}

The generalized characteristics $(X_{\varepsilon}(s,t,x,v),V_{\varepsilon}(s,t,x,v))$ are defined as follows. For $X_{\varepsilon}(s,t,x,v)\in\Omega$, the generalized characteristics satisfy the following ODEs
\begin{equation}\label{eq-modified-ODEs}
\left\{\begin{split}
&\frac{\mathrm{d}}{\mathrm{d}s}X_{\varepsilon}(s,t,x,v)=V_{\varepsilon}(s,t,x,v),\\
&\frac{\mathrm{d}}{\mathrm{d}s}V_{\varepsilon}(s,t,x,v)=\nabla_x\phi_{\#,\varepsilon}^{\varsigma}(s,X_{\varepsilon}(s,t,x,v)),\\
&X_{\varepsilon}(t,t,x,v)=x,\quad V_{\varepsilon}(t,t,x,v)=v.
\end{split}\right.
\end{equation}
If $X_{\varepsilon}(\bar{s},t,x,v)\in\partial\Omega$, and $X_{\varepsilon}(s,t,x,v)\in\Omega$ for $s\in(\bar{s}-\epsilon,\bar{s}+\epsilon)\setminus\{\bar{s}\}$ with some $\epsilon>0$, the generalized characteristics satisfy the following reflection boundary condition
\begin{equation*}
\mathcal{R}_{X_{\varepsilon}(\bar{s},t,x,v)}V_{\varepsilon}(\bar{s}+0,t,x,v)=V_{\varepsilon}(\bar{s}-0,t,x,v).
\end{equation*}

It is no difficulty to see that Theorem~\ref{thm-reflection-Neumann-case} holds also for the modified Vlasov-Poisson system. Hence there exists a unique global classical solution to \eqref{eq-modified-Vlasov-Poisson} satisfying $\forall T>0$, $f_{\varepsilon}\in C^{1}([0,T];C_c^{1,\lambda}(\bar{\Omega}\times\mathbb{R}^3;\mathbb{R}_+))$, $\phi_{\#,\varepsilon}^{\varsigma}\in C^{1}([0,T];C^{3,\lambda}(\bar{\Omega}))$ for some $\lambda>0$, which can be represented by
\begin{equation*}
f_{\varepsilon}(t,x,v)=\sum_{\alpha}f_{\alpha,\varepsilon}(t,x,v)=\sum_{\alpha}f_{0\alpha,\varepsilon}(X_{\varepsilon}(0,t,x,v),V_{\varepsilon}(0,t,x,v)).
\end{equation*}
Moreover, by the representation of $\phi_{\mathrm{N},\varepsilon}^{\varsigma}$ and the elliptic estimates, we have
\begin{equation}\label{eq-electric-field-cutoff-C2-estimate}
\|\phi_{\#,\varepsilon}^{\varsigma}\|_{C^2}\le C\varsigma^{-3},\quad \vert \nabla_x\phi_{\#,\varepsilon}^{\varsigma}(t,x)-\nabla_x\phi_{\#,\varepsilon}^{\varsigma}(t,y)\vert\le C\varsigma^{-3}\vert x-y\vert,
\end{equation}
where $C$ is a constant depending only on the $C^3$ regularity of $h_{\#}^{\rm cha}$ and $\partial\Omega$.


\subsection{The point charge model}

The point charge model is written by
\begin{equation}\label{eq-point-charge}
\left\{\begin{split}
&\dot{\xi}_{\alpha}(t)=\eta_{\alpha}(t),\\
&\dot{\eta}_{\alpha}(t)=\tilde{E}_{\#,\alpha}^{\rm cha}(t),\\
&\xi_{\alpha}(0)=\mathring{\xi}_{\alpha},\quad \eta_{\alpha}(0)=\mathring{\eta}_{\alpha},
\end{split}\right.
\end{equation}
where
\begin{equation*}
\tilde{E}_{\#,\alpha}^{\rm cha}(t)=F_{\#,\alpha}(t)+\nabla_xH_{\#}(\xi_{\alpha}(t)).
\end{equation*}
It is no difficult to obtain global well-posedness for $\#=\mathrm{N}$ and local well-posedness for $\#=\mathrm{D}$ if initial data satisfy the first condition in \eqref{eq-initial-singular-sets}. Moreover, there exist $T_0,\delta_2>0$ depending on $\delta_1$ such that
\begin{equation}\label{eq-T0-distance-point-charges-each-other}
\inf_{t\in[0,T_0]}\min_{\alpha}\left\{\operatorname{d}_{\partial\Omega}(\xi_{\alpha}(t)),\min_{\beta:\beta\ne\alpha}\vert\xi_{\alpha}(t)-\xi_{\beta}(t)\vert\right\}\ge\delta_2.
\end{equation}
Note $T_0=+\infty$ if $\#=\mathrm{N}$. Let $0<\varsigma<\frac{1}{8}\delta_2$, then the point charge model is equivalent to the modified point charge model for $t\in[0,T_0]$
\begin{equation}\label{eq-modified-point-charge}
\left\{\begin{split}
&\dot{\xi}_{\alpha}(t)=\eta_{\alpha}(t),\\
&\dot{\eta}_{\alpha}(t)=\tilde{E}_{\#,\alpha}^{{\rm cha},\varsigma}(t),\\
&\xi_{\alpha}(0)=\mathring{\xi}_{\alpha},\quad \eta_{\alpha}(0)=\mathring{\eta}_{\alpha},
\end{split}\right.
\end{equation}
where
\begin{equation*}
\tilde{E}_{\#,\alpha}^{{\rm cha},\varsigma}(t)=\sum_{\beta:\beta\ne\alpha}\nabla_xG_{\#}^{\varsigma}(\xi_{\alpha}(t),\xi_{\beta}(t))+\nabla_xH_{\#}(\xi_{\alpha}(t)).
\end{equation*}

\subsection{The desingularisation argument}

Define
\begin{align*}
&\xi_{\alpha,\varepsilon}(t)=\iint_{\Omega\times\mathbb{R}^3} X_{\varepsilon}(t)f_{0\alpha,\varepsilon}(x,v)\,\mathrm{d}x\,\mathrm{d}v,\\
&\eta_{\alpha,\varepsilon}(t)=\iint_{\Omega\times\mathbb{R}^3} V_{\varepsilon}(t)f_{0\alpha,\varepsilon}(x,v)\,\mathrm{d}x\,\mathrm{d}v,
\end{align*}
where $X_{\varepsilon}(t)=X_{\varepsilon}(t,0,x,v)$, $V_{\varepsilon}(t)=V_{\varepsilon}(t,0,x,v)$ for short. Now we state the main result in section.
\begin{theorem}\label{thm-justification-boundary-effect}
Assume that $\{\xi_{\alpha},\eta_{\alpha}\}$ is a solution to \eqref{eq-point-charge} as well as \eqref{eq-modified-point-charge} in some interval $[0,T_0]$ satisfying \eqref{eq-T0-distance-point-charges-each-other}. Denote
\begin{equation*}
p_{\varepsilon}(t)=\sum_{\beta}\left[\vert \xi_{\beta,\varepsilon}(t)-\xi_{\beta}(t)\vert+\vert \eta_{\beta,\varepsilon}(t)-\eta_{\beta}(t)\vert\right].
\end{equation*}
Then for $T=T_0$ if $\#=\mathrm{D}$ and for $T>0$ arbitrary if $\#=\mathrm{N}$, setting $\varsigma=\left(\frac{CT}{\vert\ln\varepsilon\vert}\right)^{\frac{1}{3}}$, it holds
\begin{equation*}
\lim_{\varepsilon\to 0}\sup_{t\in[0,T]}p_{\varepsilon}(t)=0.
\end{equation*}
\end{theorem}

\begin{proof}
We define 
\begin{equation*}
T_{\varepsilon}:=\sup\left\{t\ge 0: \inf_{s\in[0,t]}\inf_{(x,v)\in\operatorname{supp}f_{0,\varepsilon}}\operatorname{d}_{\partial\Omega}(X_{\varepsilon}(t,0,x,v))\ge\frac{1}{2}\delta_2\right\}.
\end{equation*}
By the regularity of the solution $f_{\varepsilon}$ and $\phi_{\#,\varepsilon}^{\varsigma}$ stated in Subsection~\ref{subsec-justification-VP}, for each $\varepsilon\in(0,\frac{1}{8}\delta_2)$, we have $T_{\varepsilon}>0$. 

Define
\begin{equation*}
\mathcal{I}_{\alpha,\varepsilon}(t,x,v)=\iint_{\Omega\times\mathbb{R}^3} \left(\vert Y_{\varepsilon}(t)-x\vert^2+\vert W_{\varepsilon}(t)-v\vert^2\right)f_{0\alpha,\varepsilon}(y,w)\,\mathrm{d}y\,\mathrm{d}w,
\end{equation*}
where $Y_{\varepsilon}(t)=Y_{\varepsilon}(t,0,y,w)$, $W_{\varepsilon}(t)=W_{\varepsilon}(t,0,y,w)$. Denote
\begin{align*}
\bar{\mathcal{I}}_{\alpha,\varepsilon}(t)&=\mathcal{I}_{\alpha,\varepsilon}(t,\xi_{\alpha}(t),\eta_{\alpha}(t)),\\
P_{\alpha,\varepsilon}(t,x,v)&=\mathds{1}_{\operatorname{supp}f_{0\alpha,\varepsilon}}(x,v)\mathcal{I}_{\alpha,\varepsilon}(t,X_{\varepsilon}(t),V_{\varepsilon}(t)).
\end{align*}

For $t\in(0,T_{\varepsilon})$, $(X_{\varepsilon}(t,0,x,v),V_{\varepsilon}(t,0,x,v))$ satisfies the ODEs \eqref{eq-modified-ODEs}, we have 
\begin{align*}
\frac{\mathrm{d}}{\mathrm{d}t}\bar{\mathcal{I}}_{\alpha,\varepsilon}(t)=\iint_{\Omega\times\mathbb{R}^3} 2(W_{\varepsilon}(t)-\eta_{\alpha}(t))\cdot\Big[Y_{\varepsilon}(t)&-\xi_{\alpha}(t)+\nabla_x\phi_{\#,\varepsilon}^{\varsigma}(t,Y_{\varepsilon}(t))\\
&-\tilde{E}_{\#,\alpha}^{{\rm cha},\varsigma}(t)\Big]f_{0\alpha,\varepsilon}(y,w)\,\mathrm{d}y\,\mathrm{d}w,
\end{align*}
and
\begin{align*}
\frac{\mathrm{d}}{\mathrm{d}t}P_{\alpha,\varepsilon}(t,x,v)=\iint_{\Omega\times\mathbb{R}^3} 2&(W_{\varepsilon}(t)-V_{\varepsilon}(t))\cdot\Big[Y_{\varepsilon}(t)-X_{\varepsilon}(t)+\nabla_x\phi_{\#,\varepsilon}^{\varsigma}(t,Y_{\varepsilon}(t))\\
&-\nabla_x\phi_{\#,\varepsilon}^{\varsigma}(t,X_{\varepsilon}(t))\Big]\mathds{1}_{\operatorname{supp}f_{0\alpha,\varepsilon}}(x,v)f_{0\alpha,\varepsilon}(y,w)\,\mathrm{d}y\,\mathrm{d}w,
\end{align*}
where
\begin{equation*}
\dot{\eta}_{\alpha,\varepsilon}(t)=\iint_{\Omega\times\mathbb{R}^3} \nabla_x\phi_{\#,\varepsilon}^{\varsigma}(t,X_{\varepsilon}(t))f_{0\alpha,\varepsilon}(x,v)\,\mathrm{d}x\,\mathrm{d}v.
\end{equation*}

By \eqref{eq-electric-field-cutoff-C2-estimate} and the definition of $G_{\#}^{\varsigma}$, it can be calculated directly that
\begin{align*}
&\left \vert \nabla_x\phi_{\#,\varepsilon}^{\varsigma}(t,Y_{\varepsilon}(t))-\tilde{E}_{\#,\alpha}^{{\rm cha},\varsigma}(t)\right\vert \\
&\le\left \vert\nabla_x\phi_{\#,\varepsilon}^{\varsigma}(t,Y_{\varepsilon}(t))-\nabla_x\phi_{\#,\varepsilon}^{\varsigma}(t,\xi_{\alpha}(t))\right\vert+\left \vert \nabla_x\phi_{\#,\varepsilon}^{\varsigma}(t,\xi_{\alpha}(t))-\tilde{E}_{\#,\alpha}^{{\rm cha},\varsigma}(t)\right\vert\\
&\le C\varsigma^{-3}\left \vert Y_{\varepsilon}(t)-\xi_{\alpha}(t)\right\vert\\
&\quad+\sum_{\beta}\iint_{\Omega\times\mathbb{R}^3} \left\vert\nabla_xG_{\#}^{\varsigma}(\xi_{\alpha}(t),y)-\nabla_xG_{\#}^{\varsigma}(\xi_{\alpha}(t),\xi_{\beta}(t))\right\vert f_{\beta,\varepsilon}(t,y,w)\,\mathrm{d}y\,\mathrm{d}w\\
&\le C\varsigma^{-3}\left \vert Y_{\varepsilon}(t)-\xi_{\alpha}(t)\right\vert+C\varsigma^{-3}\sum_{\beta}\iint_{\Omega\times\mathbb{R}^3}\vert y-\xi_{\beta}(t)\vert f_{\beta,\varepsilon}(t,y,w)\,\mathrm{d}y\,\mathrm{d}w.
\end{align*}
Hence
\begin{equation*}
\left\vert\frac{\mathrm{d}}{\mathrm{d}t}\sum_{\alpha}\bar{\mathcal{I}}_{\alpha,\varepsilon}(t)\right\vert\le C\varsigma^{-3}\sum_{\alpha}\bar{\mathcal{I}}_{\alpha,\varepsilon}(t).
\end{equation*}

Similarly, we have
\begin{equation*}
\left\vert\frac{\mathrm{d}}{\mathrm{d}t}P_{\alpha,\varepsilon}(t,x,v)\right\vert\le C\varsigma^{-3}P_{\alpha,\varepsilon}(t,x,v).
\end{equation*}

By Gr\"onwall's inequality, we have
\begin{align*}
\sum_{\alpha}\bar{\mathcal{I}}_{\alpha,\varepsilon}(t)&\le  \sum_{\alpha}\bar{\mathcal{I}}_{\alpha,\varepsilon}(0)e^{C\varsigma^{-3}t},\\
P_{\alpha,\varepsilon}(t,x,v)&\le P_{\alpha,\varepsilon}(0,x,v)e^{C\varsigma^{-3}t}.
\end{align*}
Note
\begin{equation*}
\bar{\mathcal{I}}_{\alpha,\varepsilon}(0)\le C\varepsilon^2,\quad P_{\alpha,\varepsilon}(0,x,v)\le C\varepsilon^2.
\end{equation*}

Hence for arbitrary $T>0$ and $\varsigma=\left(\frac{CT}{\vert\ln\varepsilon\vert}\right)^{\frac{1}{3}}$, we have for $0<t\le\min(T_{\varepsilon},T)$
\begin{align*}
\sum_{\alpha}\bar{\mathcal{I}}_{\alpha,\varepsilon}(t)\le C\varepsilon,\quad P_{\alpha,\varepsilon}(t,x,v)\le C\varepsilon.
\end{align*}

Take $\varepsilon>0$ small enough depending only on $\delta_2$ such that
\begin{align*}
\mathds{1}_{\operatorname{supp}f_{0\alpha,\varepsilon}}(x,v)\left[\vert X_{\varepsilon}(t)-\xi_{\alpha,\varepsilon}(t)\vert^2+\vert V_{\varepsilon}(t)-\eta_{\alpha,\varepsilon}(t)\vert^2\right]&\le P_{\alpha,\varepsilon}(t,x,v)\le\frac{\delta_2^2}{64},\\
\vert \xi_{\alpha,\varepsilon}(t)-\xi_{\alpha}(t)\vert^2+\vert \eta_{\alpha,\varepsilon}(t)-\eta_{\alpha}(t)\vert^2&\le  \bar{\mathcal{I}}_{\alpha,\varepsilon}(t)\le\frac{\delta_2^2}{64}.
\end{align*}
Hence
\begin{equation*}
\mathds{1}_{\operatorname{supp}f_{0\alpha,\varepsilon}}(x,v)\vert \xi_{\alpha}(t)-X_{\varepsilon}(t)\vert\le\frac{1}{4}\delta_2.
\end{equation*}
Combining with \eqref{eq-T0-distance-point-charges-each-other}
\begin{equation*}
\inf_{s\in[0,T]}\inf_{(x,v)\in\operatorname{supp}f_{0,\varepsilon}}\operatorname{d}_{\partial\Omega}(X_{\varepsilon}(t,0,x,v))\ge\frac{3}{4}\delta_2.
\end{equation*}
According to the definition of $T_{\varepsilon}$, we have $T_{\varepsilon}\ge T$. Hence for $0<t\le T$
\begin{equation*}
\sum_{\alpha}\bar{\mathcal{I}}_{\alpha,\varepsilon}(t)\le C\varepsilon,
\end{equation*}
which implies
\begin{equation*}
\lim_{\varepsilon\to 0}\sup_{t\in[0,T]}p_{\varepsilon}(t)=0.
\end{equation*}

\end{proof}

\paragraph{Funding} This work is partially supported by the National Natural Science Foundation of China (Grant No. 12501282).

\paragraph{Data Availability} Not applicable.

\section*{Declarations}

\paragraph{Conflict of interest} The authors have no Conflict of interest.
 

\bibliographystyle{abbrv}
{\footnotesize\bibliography{ref}}
\end{document}